\documentclass{article}

\usepackage{hyperref}
\usepackage[utf8]{inputenc}
\usepackage[UKenglish]{babel}
\usepackage[UKenglish]{isodate}
\usepackage[inline]{enumitem}
\usepackage{math}
\usepackage[table]{xcolor}
\usepackage{booktabs}
\usepackage{array}
\usepackage{authblk}
\usepackage{calc}
\usepackage{multirow}

\newcommand{\mathref}[1]{\text{\ref{#1}}}

\newcommand{\xpower}{b}
\newcommand{\ppower}{a}
\newcommand{\Oh}{\mathcal{O}}
\newcommand{\Ha}{\mathcal{H}}
\newcommand{\Ly}{\mathcal{V}}
\newcommand{\Wy}{\mathcal{W}}
\newcommand{\Pa}{\mathcal{P}}
\newcommand{\grad}{\nabla \!\!\;}
\newcommand{\hess}{\nabla^2 \!\!\;}
\newcommand{\fc}{f_c}
\newcommand{\f}{f}
\newcommand{\K}{k}
\newcommand{\xmin}{x_{\star}}
\newcommand{\maxeigenf}{\lambda_{\max}}
\newcommand{\maxeigen}[1]{\maxeigenf^{\norm{\c}} \!  \l(#1\r)}
\newcommand{\maxeigenconj}[1]{\maxeigenf^{\norm{\c}_*} \! \l(#1\r)}
\newcommand{\maxeigennorm}[2]{\maxeigenf^{#2} \!  \l(#1\r)}

\newcommand{\Ca}{C_{\alpha, \gamma}}
\newcommand{\Cfk}{C_{\f,\K}}

\newcommand{\Ck}{C_{\K}}

\newcommand{\Dfk}{D_{\f,\K}}
\newcommand{\Df}{D_{\f}}
\newcommand{\Dk}{D_{\K}}

\newcommand{\Lf}{L_{\f}}
\newcommand{\Lk}{L_{\K}}

\newcommand{\MaA}{\max\{a,A\}}
\newcommand{\ie}{i.e.,}
\newcommand{\eg}{e.g.,}
\newcommand{\etal}{\emph{et al.}}

\newcommand{\alphastar}{\alpha_{\star}}

\renewcommand{\c}{\cdot}
\renewcommand{\l}{\left}
\renewcommand{\r}{\right}

\newcounter{asscounter}
\renewcommand{\theasscounter}{\Alph{asscounter}}

\newenvironment{assumptions}[1]
    {	
    	\refstepcounter{asscounter}
        \label{ass:#1}
	\begin{center}
        	\begin{tabular}{| @{\hspace{1em}}m{0.9\textwidth}@{\hspace{1em}} |}
	\hline
        \vspace{1ex}\textbf{Assumptions \theasscounter.}
        \begin{enumerate}[label=  \theasscounter.\arabic*]
    }
    {
        \end{enumerate}\\
        \hline
    	\end{tabular}
	\end{center}
    }

\newenvironment{titledbox}[1]
    {	
	\begin{center}
        	\begin{tabular}{| @{\hspace{1em}}m{0.9\textwidth}@{\hspace{1em}} |}
	\hline
        \vspace{1ex}\textbf{#1.}
    }
    {
        \\\hline
    	\end{tabular}
	\end{center}
    }

\cleanlookdateon

\title{Hamiltonian Descent Methods}
\author[1,2,*]{Chris J. Maddison}
\author[1,*]{Daniel Paulin}
\author[1,2]{Yee Whye Teh}
\author[2]{Brendan O'Donoghue}
\author[1]{Arnaud Doucet}
\affil[1]{Department of Statistics, University of Oxford}
\affil[2]{DeepMind, London, UK}
\affil[*]{Both authors contributed equally to this work.}
\date{\today}
\begin{document}

\maketitle

\begin{abstract}
We propose a family of optimization methods that achieve linear convergence using first-order gradient information and constant step sizes on a class of convex functions much larger than the smooth and strongly convex ones. This larger class includes functions whose second derivatives may be singular or unbounded at their minima. Our methods are discretizations of conformal Hamiltonian dynamics, which generalize the classical momentum method to model the motion of a particle with non-standard kinetic energy exposed to a dissipative force and the gradient field of the function of interest. They are first-order in the sense that they require only gradient computation. Yet, crucially the kinetic gradient map can be designed to incorporate information about the convex conjugate in a fashion that allows for linear convergence on convex functions that may be non-smooth or non-strongly convex. We study in detail one implicit and two explicit methods. For one explicit method, we provide conditions under which it converges to stationary points of non-convex functions. For all, we provide conditions on the convex function and kinetic energy pair that guarantee linear convergence, and show that these conditions can be satisfied by functions with power growth. In sum, these methods expand the class of convex functions on which linear convergence is possible with first-order computation.
\end{abstract}

\section{Introduction}
\label{sec:introduction}

We consider the problem of unconstrained minimization of a differentiable function $f:\R^d\to\R$,
\begin{equation}
\label{eq:problem} \min_{x \in \R^d} f(x),
\end{equation}
by iterative methods that require only the partial derivatives $\grad f(x) = (\partial f(x) / \partial x^{(n)}) \in \R^d$ of $f$, known also as first-order methods \cite{nemirovsky1983problem, polyak1987introduction, nesterov2013introductory}. These methods produce a sequence of iterates $x_i \in \R^d$, and our emphasis is on those that achieve linear convergence, \ie{} as a function of the iteration $i$ they satisfy $f(x_i) - f(\xmin) = \mathcal{O}(\lambda^{-i})$ for some rate $\lambda > 1$ and $\xmin \in \R^d$ a global minimizer. We briefly consider non-convex differentiable $f$, but the bulk of our analysis focuses on the case of convex differentiable $f$. Our results will also occasionally require twice differentiability of $f$.

The convergence rates of first-order methods on convex functions can be broadly separated by the properties of strong convexity and Lipschitz smoothness. Taken together these properties for convex $f$ are equivalent to the conditions that the following left hand bound (strong convexity) and right hand bound (smoothness) hold for some $\mu, L \in (0, \infty)$ and all $x,y \in \R^d$,
\begin{equation}
\label{eq:strgcvxsmooth} \frac{\mu}{2}\norm{x-y}^2_2 \leq f(x) - f(y) - \inner{\grad f(y)}{x-y} \leq \frac{L}{2}\norm{x-y}^2_2,
\end{equation}
where $\inner{x}{y} = \sum_{n=1}^d x^{(n)} y^{(n)}$ is the standard inner product and $\norm{x}_2 = \sqrt{\inner{x}{x}}$ is the Euclidean norm.
For twice differentiable $f$, these properties are equivalent to the conditions that eigenvalues of the matrix of second-order partial derivatives $\hess f(x) = (\partial^2 f(x) / \partial x^{(n)} \partial x^{(m)}) \in \R^{d \times d}$ are everywhere lower bounded by $\mu$ and upper bounded by $L$, respectively. Thus, functions whose second derivatives are continuously unbounded or approaching 0, cannot be both strongly convex and smooth.
Both bounds play an important role in the performance of first-order methods. On the one hand, for smooth and strongly convex $f$, the iterates of many first-order methods converge linearly. On the other hand, for any first-order method, there exist smooth convex functions and non-smooth strongly convex functions on which its convergence is sub-linear, \ie{} $f(x_i) - f(\xmin) \geq \mathcal{O}(i^{-2})$ for any first-order method on smooth convex functions. See \cite{nemirovsky1983problem, polyak1987introduction, nesterov2013introductory} for these classical results and \cite{juditsky2014deterministic} for other more exotic scenarios. Moreover, for a given method it can sometimes be very easy to find examples on which its convergence is slow; see Figure \ref{fig:introfig}, in which gradient descent with a fixed step size converges slowly on $f(x) =  [x^{(1)}+x^{(2)}]^4 +  [x^{(1)}/2 - x^{(2)}/2]^4$,  which is not strongly convex as its Hessian is singular at $(0,0)$.

\begin{figure}[t]
    \centering
    \includegraphics[scale=1]{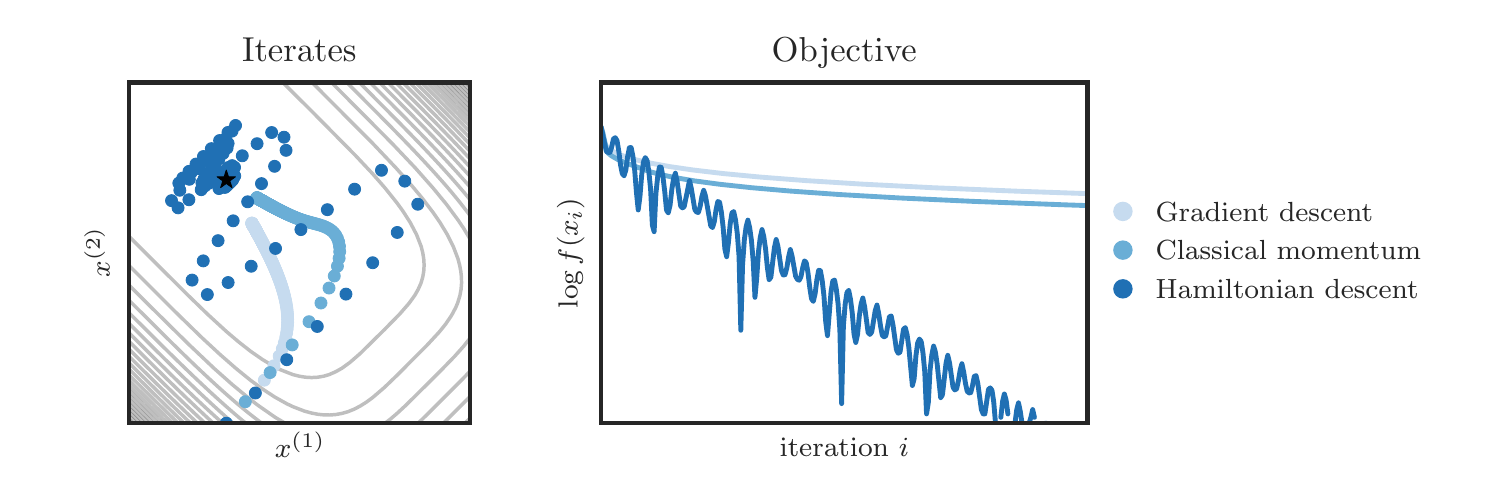}
    \caption{Optimizing $f(x) =  [x^{(1)}+x^{(2)}]^4 +  [x^{(1)}/2 - x^{(2)}/2]^4$ with three methods: gradient descent with fixed step size equal to $1/L_0$ where $L_0 = \lambda_{\max}(\hess f(x_0))$ is the maximum eigenvalue of the Hessian $\hess f$ at $x_0$; classical momentum, which is a particular case of our first explicit method with $\K(p) = [(p^{(1)})^{2} + (p^{(2)})^{2}]/2$ and fixed step size equal to $1/L_0$; and Hamiltonian descent, which is our first explicit method with $\K(p) = (3/4)[(p^{(1)})^{4/3} + (p^{(2)})^{4/3}]$ and a fixed step size.}
    \label{fig:introfig}
\end{figure}

The central assumption in the worst case analyses of first-order methods is that information about $f$ is restricted to black box evaluations of $f$ and $\grad f$ locally at points $x \in \R^d$, see \cite{nemirovsky1983problem, nesterov2013introductory}. In this paper we assume additional access to first-order information of a second differentiable function $\K : \R^d \to \R$ and show how $\grad \K$ can be designed to incorporate information about $f$ to yield practical methods that converge linearly on convex functions. These methods are derived by discretizing the conformal Hamiltonian system \cite{mclachlan2001conformal}. These systems are parameterized by $f, \K : \R^d \to \R$ and $\gamma \in (0, \infty)$ with solutions $(x_t, p_t) \in \R^{2d}$,
\begin{equation}
\label{eq:ode}
\begin{aligned}
    x_t' &= \grad \K(p_t)\\
    p_t' & = -\grad \f(x_t)-\gamma p_t.
\end{aligned}
\end{equation}
From a physical perspective, these systems model the dynamics of a single particle located at $x_t$ with momentum $p_t$ and kinetic energy $\K(p_t)$ being exposed to a force field $\grad \f $ and a dissipative force. For this reason we refer to $\K$ as, the \emph{kinetic energy}, and $\grad \K$, the \emph{kinetic map}. When the kinetic map $\grad \K$ is the identity, $\grad \K(p) = p$, these dynamics are the continuous time analog of Polyak's heavy ball method \cite{polyak1964some}.  Let $\fc(x)=f(x+\xmin)-f(\xmin)$ denote the centered version of $f$, which takes its minimum at $0$, with minimum value $0$. Our key observation in this regard is that when $f$ is convex, and $\K$ is chosen as $\K(p)= (\fc^*(p)+\fc^*(-p))/2$ (where $\fc^*(p)=\sup\{ \inner{x}{p}-\fc(x) : x\in \R^d\}$ is the convex conjugate of $\fc$), these dynamics have linear convergence with rate independent of $f$. In other words, this choice of $\K$ acts as a preconditioner, a generalization of using $\K(p)=\inner{p}{A^{-1} p}/2$ for $f(x)=\inner{x}{A x}/2$. Thus $\grad \K$ can exploit global information provided by the conjugate $\fc^*$ to condition convergence for generic convex functions.

\begin{figure}[t]
    \centering
    \includegraphics[scale=1]{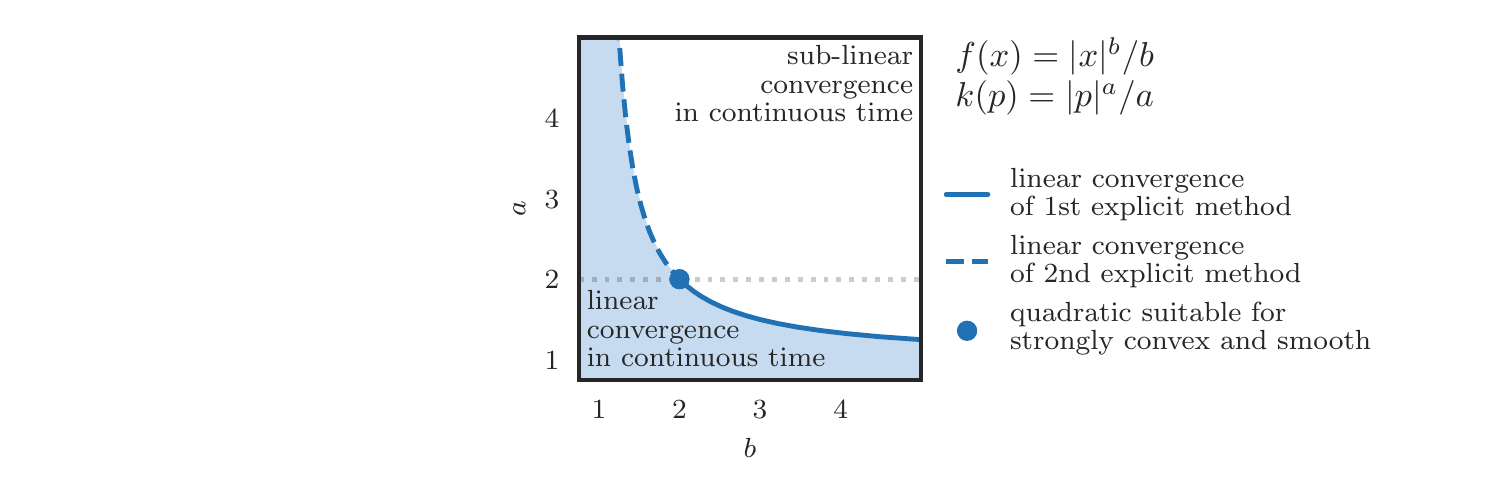}
    \caption{Convergence Regions for Power Functions. Shown are regions of distinct convergence types for Hamiltonian descent systems with $f(x) = |x|^b/b, k(p) = |p|^a/a$ for $x,p \in \R$ and $a,b \in (1, \infty)$. We show in Section \ref{sec:continuous} convergence is linear in continuous time iff $1/a + 1/b \geq 1$. In Section \ref{sec:kin} we show that the assumptions of the explicit discretizations can be satisfied if $1/a + 1/b = 1$, leaving this as the only suitable pairing for linear convergence. Light dotted line is the line occupied by classical momentum with $\K(p) = p^2/2$.}
    \label{fig:regions}
\end{figure}

To preview the flavor of our results in detail, consider the special case of optimizing the power function $f(x) = |x|^b/b$ for $x \in \R$ and $b \in (1, \infty)$ initialized at $x_0 > 0$ using system \eqref{eq:ode} (or discretizations of it) with $\K(p) = |p|^a/a$ for $p \in \R$ and $a \in (1, \infty)$. FOr this choice of $f$, it can be shown that $\fc^*(p) = \fc^*(-p) = \K(p)$ when $a = b/(b-1)$. In line with this, in Section \ref{sec:continuous} we show that \eqref{eq:ode} exhibits linear convergence in continuous time if and only if $1/a + 1/b \geq 1$. In Section \ref{sec:discrete} we propose two explicit discretizations with fixed step sizes; in Section \ref{sec:kin} we show that the first explicit discretization converges if $1/a + 1/b = 1$ and $b \geq 2$, and the second converges if $1/a + 1/b = 1$ and $1 < b \leq 2$. This means that the only suitable pairing corresponds in this case to the choice $\K(p) \propto \fc^*(p) + \fc^*(-p)$. Figure \ref{fig:regions} summarizes this discussion. Returning to Figure \ref{fig:introfig}, we can compare the use of the kinetic energy of Polyak's heavy ball with a kinetic energy that relates appropriately to the convex conjugate of $f(x) =  [x^{(1)}+x^{(2)}]^4 +  [x^{(1)}/2 - x^{(2)}/2]^4$.

Most convex functions are not simple power functions, and computing $\fc^*(p)+\fc^*(-p)$ exactly is rarely feasible. To make our observations useful for numerical optimization, we show that linear convergence is still achievable in continuous time even if $\K(p)\ge \alpha \max\{\fc^*(p),\fc^*(-p)\}$ for some $0 < \alpha \leq 1$ within a region defined by $x_0$. We study three discretizations of \eqref{eq:ode}, one implicit method and two explicit ones (which are suitable for functions that grow asymptotically fast or slow, respectively). We prove linear convergence rates for these under appropriate additional assumptions. We introduce a family of kinetic energies that generalize the power functions to capture distinct power growth near zero and asymptotically far from zero. We show that the additional assumptions of discretization can be satisfied for this family of $\K$. We derive conditions on $f$ that guarantee the linear convergence of our methods when paired with a specific choice of $\K$ from this family. These conditions generalize the quadratic growth implied by smoothness and strong convexity, extending it to general power growth that may be distinct near the minimum and asymptotically far from the minimum, which we refer to as tail and body behavior, respectively. Step sizes can be fixed independently of the initial position (and often dimension), and do not require adaptation, which often leads to convergence problems, see \cite{wilson2017marginal}. Indeed, we analyze a kinetic map $\grad \K$ that resembles the iterate updates of some popular adaptive gradient methods \cite{duchi2011adaptive, zeiler2012adadelta, hinton2014rmsprop, kingma2014adam}, and show that it conditions the optimization of strongly convex functions with very fast growing tails (non-smooth). Thus, our methods provide a framework optimizing potentially non-smooth or non-strongly convex functions with linear rates using first-order computation.

The organization of the paper is as follows. In the rest of this section, we cover notation, review a few results from convex analysis, and give an overview of the related literature.
In Section \ref{sec:continuous}, we show the linear convergence of \eqref{eq:ode} under conditions on the relation between the kinetic energy $\K$ and $f$. We show a partial converse that in some settings our conditions are necessary. In Section \ref{sec:discrete}, we present the three discretizations of the continuous dynamics and study the assumptions under which linear rates can be guaranteed for convex functions. For one of the discretizations, we also provide conditions under which it converges to stationary points of non-convex functions.  In Section \ref{sec:kin}, we study a family of kinetic energies suitable for functions with power growth. We describe the class of functions for which the assumptions of the discretizations can be satisfied when using these kinetic energies.

\subsection{Notation and Convex Analysis Review}
We let $\inner{x}{y} = \sum_{n=1}^d x^{(n)} y^{(n)}$ denote the standard inner product for $x,y \in \R^d$ and $\norm{x}_2 = \sqrt{\inner{x}{x}}$ the Euclidean norm. For a differentiable function $f: \R^d \to \R$, the gradient $\grad f(x) = (\partial f(x) / \partial x^{(n)}) \in \R^d$ is the vector of partial derivatives at $x$. For twice-differentiable $f$, the Hessian $\hess h(x) = (\partial^2 f(x) / \partial x^{(n)} \partial x^{(m)}) \in \R^{d \times d}$ is the matrix of second-order partial derivatives at $x$. The notation $x_t$ denotes the solution $x_t : [0, \infty) \to \R^d$ to a differential equation with derivative in $t$ denoted $x_t'$. $x_i$ denotes the iterates $x_i : \{0, 1, \ldots\} \to \R^d$ of a discrete system.

Consider a convex function $h : C \to \R$ that is defined on a convex domain $C \subseteq \R^d$ and differentiable on the interior $\interior(C)$. The convex conjugate $h^* : \R^d \to \R$ is defined as
\begin{equation}\label{eq:convconj}
h^*(p)=\sup\{ \inner{x}{p}-h(x) : x\in C\}
\end{equation}
and it is itself convex. It is easy to show from the definition that if $g: C \to \R$ is another convex function such that $g(x) \leq h(x)$ for all $x \in C$, then $h^*(p) \leq g^*(p)$ for all $p \in \R^d$. Because we make such extensive use of it, we remind readers of the Fenchel-Young inequality: for $x \in C$ and $p \in \R^d$,
\begin{align}
\label{eq:fenchelyoung} \inner{x}{p} &\leq h(x) + h^*(p),
\intertext{which is easily derived from the definition of $h^*$, or see Section 12 of \cite{rockafellar1970convex}. For $x \in \interior(C)$ by Theorem 26.4 of \cite{rockafellar1970convex},}
\label{eq:fenchelyoungtight} \inner{x}{\grad h(x)} &= h(x) + h^*(\grad h(x)).
\end{align}
Let $y \in \R^{d}$, $c \in \R \setminus \{0\}$. If $g(x) = h(x+y) - c$, then $g^*(p) = h^*(p) - \inner{p}{y} + c$ (Theorem 12.3 \cite{rockafellar1970convex}). If $h(x) = |x|^b/b$ for $x \in \R$ and $b \in (1, \infty)$, then $h^*(p) = |p|^a/a$ where $a = b/(b-1)$ (page 106 of \cite{rockafellar1970convex}). If $g(x) = c h( x)$, then $g^*(p) = c h^*(p/c)$ (Table 3.2 \cite{borwein2010convex}). For these and more on $h^*$, we refer readers to \cite{rockafellar1970convex, boyd2004convex, borwein2010convex}.

\subsection{Related Literature}
Standard references on convex optimization and the convergence analysis of first-order methods include \cite{nemirovsky1983problem, polyak1987introduction, bertsekas2003convex, boyd2004convex, nesterov2013introductory, bubeck2015convex}.

The heavy ball method was introduced by Polyak in his seminal paper \cite{polyak1964some}. In this paper, local convergence with linear rate was shown (\ie{} when the initial position is sufficiently close to the local minimum). For quadratic functions, it can be shown that the convergence rate for optimally chosen step sizes is proportional to the square root of the conditional number of the Hessian, similarly to conjugate gradient descent (see \eg{} \cite{rechtlecturenotes}). As far as we know, global convergence of the heavy ball method for non-quadratic functions was only recently established in \cite{ghadimi2015global} and \cite{lessard2016analysis}, see \cite{gurbuzbalaban2017convergence} for an extension to stochastic average gradients. The heavy ball method forms the basis of the some of the most successful optimization methods for deep learning, see \eg{} \cite{sutskever2013importance, kingma2014adam}, and the recent review \cite{bottou2018optimization}. Hereafter, classical momentum refers to any first-order discretization of the continuous analog of Polyak's heavy ball (with possibly suboptimal step sizes).

Nesterov obtained upper and lower bounds of matching order for first-order methods for smooth convex functions and smooth strongly convex functions, see \cite{nesterov2013introductory}. In Necoara \etal{} \cite{necoara2018linear}, the assumption of strong convexity was relaxed, and under a weaker quadratic growth condition, linear rates were obtained by several well known optimization methods. Several other authors obtained linear rates for various classes of non-strongly convex or non-uniformly smooth functions, see \eg{} \cite{nemirovskii1985optimal, karimi2016linear, drusvyatskiy2018error, yang2015rsg, fazlyab2017analysis, roulet2017sharpness}.

In recent years, there has been interest in the optimization community in looking at the continuous time ODE limit of optimization methods, when the step size tends to zero.
Su \etal{} \cite{su2014differential, su2016differential} have found the continuous time limit of Nesterov's accelerated gradient descent. This result improves the intuition about Nesterov's method, as the proofs of convergence rates in continuous time are rather elegant and clear, while the previous proofs in discrete time are not as transparent. Follow-ups have studied the continuous time counterparts to accelerated mirror descent \cite{krichene2015accelerated} as well as higher order discretizations of such systems \cite{wibisono2016variational, wilson2016lyapunov}. Studying continuous time systems for optimization can separate the concerns of designing an optimizer from the difficulties of discretization. This perspective has resulted in numerous other recent works that propose new optimization methods, and study existing ones via their continuous time limit, see \eg{}
\cite{betancourt2018symplectic, allen2016even, flammarion2015averaging, jin2017accelerated, drusvyatskiy2018optimal, francca2018admm, francca2018relax}.

Conformal Hamiltonian systems \eqref{eq:ode} are studied in geometry \cite{mclachlan2001conformal, bhatt2016second}, because their solutions preserve symplectic area up to a constant; when $\gamma = 0$ symplectic area is exactly preserved, when $\gamma > 0$ symplectic area dissipates uniformly at an exponential rate \cite{mclachlan2001conformal}. In classical mechanics, Hamiltonian dynamics (system \eqref{eq:ode} with $\gamma = 0$) are used to describe the motion of a particle exposed to the force field $\grad f$. Here, the most common form for $\K$ is $\K(p) = \inner{p}{p}/2m$, where $m$ is the mass, or in relativistic mechanics, $\K(p) = c\sqrt{\inner{p}{p} + m^2c^2}$ where $c$ is the speed of light, see \cite{goldstein2011classical}. In the Markov Chain Monte Carlo literature, where (discretized) Hamiltonian dynamics (again $\gamma = 0$) are used to propose moves in a Metropolis--Hastings algorithm \cite{metropolis1953equation, hastings1970monte, duane1987hybrid, neal2011mcmc}, $\K$ is viewed as a degree of freedom that can be used to improve the mixing properties of the Markov chain  \cite{girolami2011riemann, livingstone2017kinetic}. Stochastic differential equations similar to \eqref{eq:ode} with $\gamma > 0$ have been studied from the perspective of designing $\K$ \cite{lu2016relativistic, stoltz2018langevin}.

\section{Continuous Dynamics}\label{sec:continuous}
In this section, we motivate the discrete optimization algorithms by introducing their continuous time counterparts. These systems are differential equations described by a Hamiltonian vector field plus a dissipation field. Thus, we briefly review Hamiltonian dynamics, the continuous dynamics of Hamiltonian descent, and derive convergence rates for convex $f$ in continuous time.

\subsection{Hamiltonian Systems}
 In the Hamiltonian formulation of mechanics, the evolution of a particle exposed to a force field $\grad f$ is described by its location $x_t : [0, \infty) \to \R^d$ and momentum $p_t :  [0, \infty) \to \R^d$ as functions of time. The system is characterized by the total energy, or Hamiltonian,
\begin{align}
	\label{eq:hamiltonian}
	\Ha(x,p) = \K(p) + \f(x) - \f(\xmin),
\end{align}
where $\xmin$ is one of the global minimizers of $f$ and $\K : \R^d \to \R$ is called the kinetic energy. Throughout, we consider kinetic energies $\K$ that are a strictly convex functions with minimum at $\K(0) = 0$. The Hamiltonian $\Ha$ defines the trajectory of a particle $x_t$ and its momentum $p_t$ via the ordinary differential equation,
\begin{equation}
\label{eq:hamiltonianode}
\begin{aligned}
    x_t' &= \grad_p \Ha(x_t, p_t) = \grad \K(p_t)\\
    p_t' &= -\grad_x \Ha(x_t, p_t) = -\grad \f(x_t).
\end{aligned}
\end{equation}
For any solution of this system, the value of the total energy over time $\Ha_t = \Ha(x_t, p_t)$ is conserved as
$\Ha_t' = \inner{\grad \K(p_t)}{p_t'} + \inner{\grad \f(x_t)}{x_t'} = 0.$
Thus, the solutions of the Hamiltonian field oscillate, exchanging energy from $x$ to $p$ and back again.

\begin{figure}[t]
    \centering
    \includegraphics[scale=1]{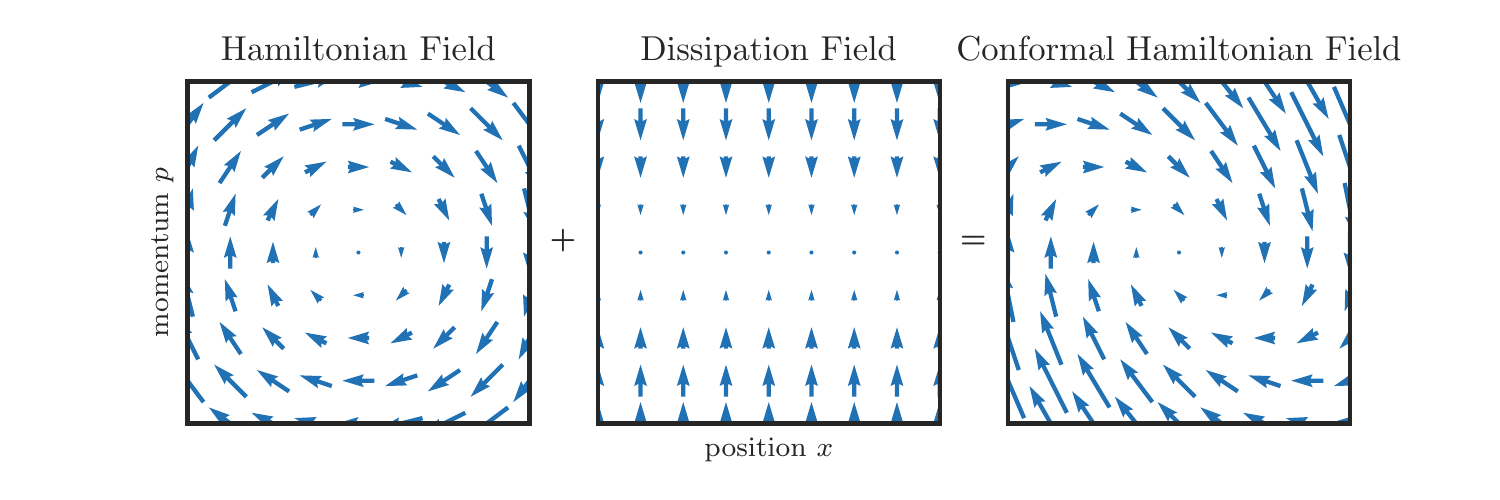}
    \caption{A visualization of a conformal Hamiltonian system.}
    \label{fig:fieldviz}
\end{figure}

\subsection{Continuously Descending the Hamiltonian}
\label{sec_continuous_Lyapunov}
The solutions of a Hamiltonian system remain in the level set $\{(x_t, p_t) : \Ha_t = H_0\}$. To drive
such a system towards stationary points, the total energy must reduce over time. Consider as a motivating example the continuous system $x_t'' = -\grad \f(x_t) - \gamma x_t'$, which describes Polyak's heavy ball algorithm in continuous time \cite{polyak1964some}. Letting $x_t' = p_t$, the heavy ball system can be rewritten as
\begin{align}
\begin{split}
    x_t' &= p_t\\
    p_t' &= -\grad \f(x_t) - \gamma p_t.
\end{split}
\end{align}
Note that this system can be viewed as a combination of a Hamiltonian field with $\K(p) = \inner{p}{p}/2$ and a dissipation field, \ie{} $(x_t', p_t') = F(x_t, p_t) + G(x_t, p_t)$ where $F(x_t, p_t) = (p_t, - \grad f(x_t))$ and $G(x_t, p_t) = (0, - \gamma p_t$), see Figure \ref{fig:fieldviz} for a visualization. This is naturally extended to define the more general conformal Hamiltonian system \cite{mclachlan2001conformal},
\begin{equation}
\begin{aligned}
    x_t' &= \grad \K(p_t)\\
    p_t' &= -\grad \f(x_t) - \gamma p_t.
\end{aligned}
\tag{\ref*{eq:ode} revisited}
\end{equation}
with $\gamma \in (0, \infty)$.
When $\K$ is convex with a minimum $\K(0)=0$, these systems descend the level sets of the Hamiltonian. We can see this by showing that the total energy $\Ha_t$ is reduced along the trajectory $(x_t, p_t)$,
\begin{align}\label{eq:Hdecreases}
    \Ha_t' = \inner{\grad \K(p_t)}{p_t'} + \inner{\grad \f(x_t)}{x_t'} = - \gamma \inner{\grad \K(p_t)}{p_t}\le -\gamma \K(p_t) \leq 0,
\end{align}
where we have used the convexity of $\K$, and the fact that it is minimised at $\K(0)=0$.

The following proposition shows some existence and uniqueness results for the dynamics \eqref{eq:ode}. We say that $\Ha$ is radially unbounded if $\Ha(x,p) \to \infty$ when $\norm{(x, p)}_2 \to \infty$, \eg{} this would be implied if $f$ and $\K$ were strictly convex with unique minima.
\begin{prop}[Existence and uniqueness]
If $\grad f$ and $\grad \K$ are continuous, $\K$ is convex with a minimum $\K(0)=0$, and $\Ha$ is radially unbounded, then for every $x,p\in \R^d$, there exists a solution $(x_t,p_t)$ of  \eqref{eq:ode}  defined for every $t\ge 0$ with $(x_0, p_0) = (x, p)$. If in addition, $\grad f$ and $\grad \K$ are continuously differentiable, then this solution is unique.
\end{prop}
\begin{proof}
First, only assuming continuity, it follows from Peano's existence theorem \cite{peano1990demonstration} that there exists a local solution on an interval $t\in [-a,a]$ for some $a>0$. Let $[0,A)$ denote the right maximal  interval where a solution of \eqref{eq:ode} satisfying that $x_0=x$ and $p_0=p$ exist. From \eqref{eq:Hdecreases}, it follows that $\Ha_t'\le 0$, and hence $\Ha_t\le \Ha_0$ for every $t\in [0,A)$. Now by the radial unboundedness of $\Ha$, and the fact that $\Ha_t\le \Ha_0$, it follows that the compact set $\{(x,p): \Ha(x,p)\le \Ha_0\}$ is never left by the dynamics, and hence by Theorem 3 of \cite{Perko} (page 91), we must have $A=\infty$. The uniqueness under continuous differentiability follows from the Fundamental Existence--Uniqueness Theorem on page 74 of \cite{Perko}.
\end{proof}

As shown in the next proposition, \eqref{eq:Hdecreases} implies that conformal Hamiltonian systems approach stationary points of $\f$.

\begin{prop}[Convergence to a stationary point]
Let $(x_t, p_t)$ be a solution to the system \eqref{eq:ode} with initial conditions $(x_0, p_0) = (x, p) \in \R^{2d}$, $\f$ continuously differentiable, and $\K$ continuously differentiable, strictly convex with minimum at $0$ and $\K(0)=0$. If $f$ is bounded below and $\Ha$ is radially unbounded, then $\norm{\nabla f(x_t)}_2 \to 0$.
\end{prop}
\begin{proof} Since $\f$ is bounded below, $\Ha_t \geq 0$. Since $\Ha$ is radially unbounded, the set $B:=\{(x,p)\in \R^{2d}: \Ha(x,p)\le \Ha(x_0,p_0)+1\}$ is a compact set that contains $(x_0,p_0)$ in its interior. Moreover, by \eqref{eq:Hdecreases}, we also have $(x_t, p_t) \in B$ for all $t>0$.
Consider the set $M = \{(x_t, p_t) : \Ha_t^{\prime} = 0\} \cap B$. Since $\K$ is strictly convex, this set is equivalent to $\{(x_t, p_t) : \norm{p_t}_2 = 0\} \cap B$.
The largest invariant set of the dynamics \eqref{eq:ode} inside $M$ is $I=\{(x, p)\in \R^{2d} : \norm{p}_2 = 0, \norm{\grad \f(x)}_2 = 0\} \cap B$. By LaSalle's principle \cite{lasalle1960some}, all trajectories started from $B$ must approach $I$. Since $f$ is a continuous bounded function on the compact set $B$, there is a point $x_*\in B$ such that $f(x_*)\le f(x)$ for every $x\in B$ (i.e. the minimum is attained in $B$) by the extreme value theorem (see \cite{rudin}). Moreover, due to the definition of $B$, $x_*$ is in its interior, hence $\norm{\grad \f(x_*)}_2=0$ and therefore $(x_*,0)\in I$. Thus the set $I$ is non-empty (note that $I$ might contain other local minima as well).
\end{proof}

\begin{remark}
This construction can be generalized by modifying the $-\gamma p_t$ component of \eqref{eq:ode} to a more general dissipation field $-\gamma D(p_t)$. If the dissipation field is everywhere aligned with the kinetic map, $\inner{\grad \K(p)}{D(p)} \geq 0$,  then these systems dissipate energy. We have not found alternatives to $D(p) = \gamma p$ that result in linear convergence in general.\end{remark}

\subsection{Continuous Hamiltonian Descent on Convex Functions}
\label{sec_continuous_convex_lyap}

In this section we study how $\K$ can be designed to condition the system \eqref{eq:ode} for linear convergence in $\log(f(x_t) - f(\xmin))$. Although the solutions $x_t, p_t$ of \eqref{eq:ode} approach stationary points under weak conditions, to derive rates we consider the case when $\f$ is convex. To motivate our choice of $\K$, consider the quadratic function $f(x)=\inner{x}{A x}/2$ with $\K(p)= \inner{p}{A^{-1}p}/2$ for positive definite symmetric $A \in \R^{d \times d}$. Now \eqref{eq:ode} becomes,
\begin{align}
\label{eq:odeA}
\begin{split}
x_t' &= A^{-1} p_t\\
p_t' &= -A x_t - \gamma p_t.
\end{split}
\end{align}
By the change of variables $v_t = A^{-1}p_t$, this is equivalent to
\begin{align}
\label{eq:odeA2}
\begin{split}
x_t' &= v_t\\
v_t' &= -x_t - \gamma v_t,
\end{split}
\end{align}
which is a universal equation and hence the convergence rate of \eqref{eq:odeA} is independent of $A$. Although this kinetic energy implements a constant preconditioner for any $f$, for this specific $f$ $\K$ is its convex conjugate $f^*$. This suggests the core idea of this paper: taking $\K$ related in some sense to $f^*$ for more general convex functions may condition the convergence of \eqref{eq:ode}. Indeed, we show in this section that, if the kinetic energy $\K(p)$ upper bounds a centered version of $f^*(p)$, then the convergence of \eqref{eq:ode} is linear.

More precisely, define the following centered function $\fc : \R^d \to \R$,
\begin{align}
\label{eq:centredf}
\fc(x) = f(x + \xmin) - f(\xmin).
\end{align}
The convex conjugate of $\fc$ is given by $\fc^*(p) = f^*(p) - \inner{\xmin}{p} + f(\xmin)$ and is minimized at $\fc^*(0) = 0$. Importantly, as we will show in the final lemma of this section, taking a kinetic energy such that $\K(p) \geq \alpha \max(\fc^*(p),\fc^*(-p))$ for some $\alpha \in (0,1]$ suffices to achieve linear rates on any differentiable convex $f$ in continuous time. The constant $\alpha$ is included to capture the fact that $\K$ may under estimate $\fc^*$ by some constant factor, so long as it is positive. If $\alpha$ does not depend in any fashion on $f$, then the convergence rate of \eqref{eq:ode} is independent of $f$. In Section \ref{sec:lower} we also show a partial converse --- for some simple problems taking a $\K$ not satisfying those assumptions results in sub-linear convergence for almost every path (except for one unique curve and its mirror).

\begin{remark}
	There is an interesting connection to duality theory for a specific choice of $\K$. In a slight abuse of representation, consider rewriting the original problem as
	\begin{equation*}
		\min_{x \in \R^d} f(x) =  \min_{x \in \R^d} \frac{1}{2}(f(x) + f(x)).
	\end{equation*}
	The Fenchel dual of this problem is equivalent to the following problem after a small reparameterization of $p$ (see Chapter 31 of \cite{rockafellar1970convex}),
	\begin{equation*}
		\max_{p \in \R^d} \frac{1}{2}(- f^*(p) - f^*(-p)).
	\end{equation*}
	The Fenchel duality theorem guarantees that for a given pair of primal-dual variables $(x, p) \in \R^d$, the duality gap between the primal objective $f(x)$ and the dual objective $(- f^*(p) - f^*(-p))/2$ is positive. Thus,
	\begin{align*}
		f(x) - (-f^*(p) - f^*(-p))/2 &= f(x) - f(\xmin) + (f^*(p) + f^*(-p))/2 + f(\xmin)\\
		&= f(x) - f(\xmin) + (\fc^*(p) + \fc^*(-p))/2 \geq 0.
	\end{align*}
	Thus, for the choice $\K(p) = (\fc^*(p) + \fc^*(-p))/2$, which as we will show implies linear convergence of \eqref{eq:ode}, the Hamiltonian $\Ha(x,p)$ is exactly the duality gap between the primal and dual objectives.
\end{remark}

Linear rates in continuous time can be derived by a Lyapunov function $\Ly : \R^{d \times d} \to  [0, \infty)$ that summarizes the total energy of the system, contracts exponentially (or linearly in $\log$-space), and is positive unless $(x_t, p_t) = (\xmin, 0)$. Ultimately we are trying to prove a result of the form $\Ly_t' \leq - \lambda \Ly_t$ for some rate $\lambda > 0$. As the energy $\Ha_t$ is decreasing, it suggests using $\Ha_t$ as a Lyapunov function. Unfortunately, this will not suffice, as $\Ha_t$ plateaus instantaneously ($\Ha_t' = 0$) at points on the trajectory where $p_t = 0$ despite $x_t$ possibly being far from $\xmin$. However, when $p_t = 0$, the momentum field reduces to the term $-\grad f(x_t)$ and the derivative of $\inner{x_t-\xmin}{p_t}$ in $t$ is instantaneously strictly negative  $-\inner{x_t- \xmin}{\grad f(x_t)} < 0$ for convex $f$ (unless we are at $(\xmin, 0)$). This suggests the family of Lyapunov functions that we study in this paper,
\begin{align}
\label{eq:lyapdef}\Ly(x,p) = \Ha(x,p) + \beta \inner{x-\xmin}{p},
\end{align}
where $\beta \in (0, \gamma)$ (see the next lemma for conditions that guarantee that it is non-negative).  As with $\Ha$, $\Ly_t$ is used to indicate $\Ly(x_t, p_t)$ at  time $t$ along a solution to \eqref{eq:ode}. Before moving on to the final lemma of the section, we prove two technical lemmas that will give us useful control over $\Ly$ throughout the paper.

The first lemma describes how $\beta$ must be constrained for $\Ly$ to be positive and to track $\Ha$ closely, so that it is useful for the analysis of the convergence of $\Ha$ and ultimately $f$.

\begin{restatable}[Bounding the ratio of $\Ha$ and $\Ly$]{lem}{betachoicealpha}
	\label{lem:betachoicealpha}
	Let  $x \in \R^d$, $f : \R^d \to \R$ convex with unique minimum $\xmin$, $\K : \R^d \to \R$ strictly convex with minimum $\K(0) = 0$, $\alpha \in (0, 1]$ and $\beta\in (0,\alpha]$.\\
	If $p \in \R^d$ is such that $\K(p) \geq \alpha \fc^*(-p)$, then
	\begin{align}
	&\label{eq:Handinnerbound1}  \inner{x-\xmin}{p} \geq -\l(\K(p) / \alpha + f(x) - f(\xmin)\r) \ge -\frac{\Ha(x,p)}{\alpha},\\
	&\label{eq:VandHbound1} \tfrac{\alpha - \beta}{\alpha} \Ha(x,p) \leq \Ly(x,p).
	\end{align}	
	If $p \in \R^d$ is such that $\K(p) \geq \alpha \fc^*(p)$, then
	\begin{align}
	&\label{eq:Handinnerbound2}  \inner{x-\xmin}{p} \leq \K(p) / \alpha + f(x) - f(\xmin) \leq \frac{\Ha(x,p)}{\alpha},\\
	&\label{eq:VandHbound2} \Ly(x,p) \leq \tfrac{\alpha + \beta}{\alpha} \Ha(x,p).
	\end{align}
\end{restatable}

\begin{proof}
Assuming that  $\K(p) \geq \alpha \fc^*(-p)$, we have
\begin{align*}
    \K(p)/\alpha + \fc(x-\xmin) &\ge \fc^*(-p) + \fc(x-\xmin) \\
    &\geq \inner{x-\xmin}{-p} - \fc(x-\xmin) + \fc(x-\xmin)\\
    &=-\inner{x-\xmin}{p},
\end{align*}
hence we have \eqref{eq:Handinnerbound1}. \eqref{eq:VandHbound1} follows by rearrangement. The proof of \eqref{eq:Handinnerbound2} and \eqref{eq:VandHbound2} is similar.
\end{proof}

Lemma \ref{lem:betachoicealpha} constrains $\beta$ in terms of $\alpha$. For a result like $\Ly_t' \leq - \lambda \Ly_t$ , we will need to control $\beta$ in terms of the magnitude $\gamma$ of the dissipation field. The following lemma provides constraints on $\beta$ and, under those constraints, the optimal $\beta$. The proof can be found in Section \ref{sec:proofscontinuous} of the Appendix.

\begin{restatable}[Convergence rates in continuous time for fixed $\alpha$]{lem}{Vderivativebound}
	\label{lem:Vderivativebound}
	Given $\gamma \in (0, 1)$, $f : \R^d \to \R$ differentiable and convex with unique minimum $\xmin$, $\K : \R^d \to \R$ differentiable and strictly convex with minimum $\K(0) = 0$. Let $x_t,p_t \in \R^d$ be the value at time $t$ of a solution to the system \eqref{eq:ode} such that there exists $\alpha \in (0, 1]$ where $\K(p_t) \geq \alpha \fc^*(-p_t)$. Define
	\begin{equation}
	\label{eq:lamdaabgdef} \lambda(\alpha, \beta, \gamma) = \min\l(\frac{\alpha \gamma - \alpha \beta - \beta \gamma}{\alpha - \beta}, \frac{\beta(1-\gamma)}{1-\beta}\r).
	\end{equation}
	If $\beta \in (0, \min(\alpha, \gamma)]$, then
	\begin{equation*}
		\Ly_t' \leq -  \lambda(\alpha, \beta, \gamma) \Ly_t.
	\end{equation*}
	Finally,
	\begin{enumerate}
	\item \label{enum:optimalbeta} The optimal $\beta \in (0, \min(\alpha, \gamma)]$, $\beta^{\star} = \arg \max_{\beta} \lambda(\alpha, \beta, \gamma)$ and $\lambda^{\star} = \lambda(\alpha, \beta^{\star}, \gamma)$ are given by,
            \begin{align}\label{eq:betastar}
             \beta^{\star}&=\tfrac{1}{1+\alpha} \left(\alpha+\tfrac{\gamma}{2} - \sqrt{(1-\gamma) \alpha^2 + \tfrac{\gamma^2}{4}}\right),\\
            \label{eq:lambdastar}
             \lambda^{\star}&=\begin{cases}
            	&\frac{1}{1-\alpha} \left((1-\gamma) \alpha + \frac{\gamma}{2} - \sqrt{(1-\gamma) \alpha^2 + \frac{\gamma^2}{4}}\right) \text{ for }0<\alpha<1,\\
             	&\frac{\gamma (1-\gamma)}{2-\gamma}\text{ for }\alpha=1,
             \end{cases}
            \end{align}
            \item \label{enum:goodenoughbeta} If $\beta \in (0, \alpha \gamma /2]$, then \begin{align}\label{eq:goodenoughbeta}
			&\lambda(\alpha, \beta, \gamma) =  \frac{\beta(1-\gamma)}{1-\beta},
			\text{ and }\\
			\nonumber
			&-\l(\gamma - \beta-\gamma^2(1-\gamma)/4\r) \K(p_t) - \beta \gamma \inner{x_t-\xmin}{p_t} - \beta\inner{x_t-\xmin}{\grad \f(x_t)}\\
			\label{eq:goodenoughbeta2}
			&\leq -\beta(1-\gamma) (\K(p_t) + \f(x_t)-\f(\xmin)+\beta \inner{x_t-\xmin}{p_t}).
            \end{align}
\end{enumerate}
\end{restatable}


These two lemmas are sufficient to prove the linear contraction of $\Ly$ and the contraction $f(x_t) - f(\xmin) \leq \tfrac{\alpha}{\alpha-\beta^{\star}}\Ha_0\exp(-\lambda^{\star}t)$ under the assumption of constant $\alpha$ and $\beta$. Still, the constant $\alpha$, which controls our approximation of $\fc^*$ may be quite pessimistic if it must hold globally along $x_t, p_t$ as the system converges to its minimum. Instead, in the final lemma that collects the convergence result for this section, we consider the case where $\alpha$ may increase as convergence proceeds. To support an improving $\alpha$, our constant $\beta$ will now have to vary with time and we will be forced to take slightly suboptimal $\beta$ and $\lambda$ given by \eqref{eq:goodenoughbeta} of Lemma \ref{lem:Vderivativebound}. Still, the improving $\alpha$ will be important in future sections for ensuring that we are able to achieve position independent step sizes.

We are now ready to present the central result of this section. Under Assumptions \ref{ass:cont} we show linear convergence of \eqref{eq:ode}. In general, the dependence of the rate of linear convergence on $f$ is via the function $\alpha$ and the constant $C_{\alpha, \gamma}$ in our analysis.

\begin{assumptions}{cont}
    	\item \label{ass:cont:fconvex} $f : \R^d \to \R$ differentiable and convex with unique minimum $\xmin$.
	\item \label{ass:cont:kconvex} $\K : \R^d \to \R$ differentiable and strictly convex with minimum $\K(0) = 0$.
        \item $\gamma \in (0,1)$.
        \item \label{ass:cont:alpha} There exists some differentiable non-increasing convex function $\alpha:  [0, \infty) \to (0,1]$ and constant $\Ca \in (0, \gamma\,]$ such that for every $p \in \R^d$,
	\begin{equation}\label{eq:alphacond_fcstar}
        		\K(p) \geq \alpha(\K(p)) \max(\fc^*(p),\fc^*(-p))
	\end{equation}
	and that for every $y \in [0, \infty)$
	\begin{equation}\label{eq:alphacond_ca}
	-\Ca\alpha'(y)y<\alpha(y).
	\end{equation}
	In particular, if $\K(p) \geq \alphastar \max(\fc^*(p),\fc^*(-p))$ for a constant $\alphastar \in (0, 1]$, then the constant function $\alpha(y) = \alphastar$ serves as a valid, but pessimistic choice.
\end{assumptions}

\begin{remark}
Assumption \ref{ass:cont:alpha} can be satisfied if a symmetric lower bound on $f$ is known. For example, strong convexity implies
\begin{align*}
    f(x+\xmin) - f(\xmin) \geq \frac{\mu}{2}\norm{x}_2^2.
\end{align*}
This in turn implies $\fc^*(p) \leq \norm{p}_2^2/(2\mu)$. Because $\K(p) =  \norm{p}_2^2/(2\mu)$ is symmetric, it satisfies \ref{ass:cont:alpha} which explains why conditions relating to strong convexity are necessary for linear convergence of Polyak's heavy ball.
\end{remark}

\begin{restatable}[Convergence bound in continuous time with general $\alpha$]{thm}{continuouslyap}
	\label{lem:continuouslyap}
	Given $f$, $\K$, $\gamma$, $\alpha$, $\Ca$ satisfying Assumptions \ref{ass:cont}. Let $(x_t, p_t)$ be a solution to the system \eqref{eq:ode} with initial states $(x_0, p_0) = (x, 0)$ where $x \in \R^d$. Let $\alphastar = \alpha(3\Ha_0)$, $\lambda = \tfrac{(1-\gamma)\Ca }{4}$, and $\Wy :  [0, \infty) \to  [0, \infty)$ be the solution of
	\[\Wy_t'=-\lambda \cdot \alpha(2\Wy_t)\Wy_t,\]
	with $\Wy_0:=\Ha_0=f(x_0)-f(\xmin)$.
	Then for every $t \in [0, \infty)$, we have
	\begin{align}\label{eq:Hdecay}	
	f(x_t) - f(\xmin) \le 2 \Ha_0 \exp\l(- \lambda \int_0^t \alpha(2\Wy_t)\r) \leq 2 \Ha_0 \exp\l(- \lambda \alphastar t\r).
	\end{align}
\end{restatable}

\begin{proof}
By \eqref{eq:alphacond_fcstar} in assumption \ref{ass:cont:alpha}, the conditions of Lemma \ref{lem:betachoicealpha} hold, and by \eqref{eq:Handinnerbound1} and \eqref{eq:Handinnerbound2} we have
\begin{align}
    \label{eq:Kperalphaf2sided} |\inner{x_t-\xmin}{p_t}| \leq \K(p_t)/\alpha(\K(p_t)) + \f(x_t) - \f(\xmin) \leq \frac{\Ha_t}{\alpha(\K(p_t))}.
\end{align}
 Instead of defining the Lyapunov function $\Ly_t$ exactly as in \eqref{eq:lyapdef} we take a time-dependent $\beta_t$. Specifically, for every $t \ge 0$ let $\Ly_t$ be the unique solution $v$ of the equation
\begin{equation}
	\label{eq:Vtalphadef}v=\Ha_t+\frac{\Ca \alpha(2v)}{2}  \inner{x_t-\xmin}{p_t}
\end{equation}
in the interval $v \in [\Ha_t/2, 3\Ha_t/2]$.
To see why this equation has a unique solution in $v \in [\Ha_t/2, 3\Ha_t/2]$, note that from \eqref{eq:Kperalphaf2sided} it follows that
\begin{align*}
	|\alpha\l(2v\r) \inner{x_t-\xmin}{p_t}|\le \Ha_t\text{ for every }v\ge \frac{\Ha_t}{2},
\end{align*}
and hence for any such $v$, we have
\begin{align}
	\label{eq:boundonv} \frac{\Ha_t}{2} \leq \Ha_t+\frac{\Ca \alpha(2v)}{2}  \inner{x_t-\xmin}{p_t} \leq \frac{3}{2}\Ha_t.
\end{align}
This means that for $v=\frac{\Ha_t}{2}$, the left hand side of \eqref{eq:Vtalphadef} is smaller than the right hand side, while for $v=\frac{3 \Ha_t }{2}$, it is the other way around.
Now using \eqref{eq:alphacond_ca} in assumption \ref{ass:cont:alpha} and \eqref{eq:Kperalphaf2sided}, we have
\begin{equation}
\label{eq:boundinner} \l|\Ca  \alpha'\l(2\Ly_t\r)  \inner{x_t-\xmin}{p_t}\r|\le \l|\Ca  \frac{ \alpha'(2\Ly_t)  2\Ly_t}{\alpha(2\Ly_t)}\r|<1,
\end{equation}
Thus, by differentiation, we can see that \eqref{eq:boundinner} implies that
\begin{align*}
	\frac{\partial}{\partial v}\l(v-\Ha_t-\tfrac{\Ca}{2} \alpha(2v) \inner{x_t-\xmin}{p_t}\r)>0,
\end{align*}
which implies that \eqref{eq:Vtalphadef} has a unique solution $\Ly_t$ in $[\frac{\Ha}{2},\frac{3 \Ha_t}{2}]$.  Let $\alpha_t = \alpha(2 \Ly_t)$ and $\beta_t = \frac{\Ca}{2} \alpha\l(2\Ly_t\r)$. By the implicit function theorem, it follows that $\Ly_t$ is  differentiable in $t$. Morover, since
\begin{equation}
\label{eq:VtalphadefLyt}\Ly_t=\Ha_t+\frac{\Ca \alpha(2\Ly_t)}{2} \inner{x_t-\xmin}{p_t}
\end{equation}
for every $t\ge 0$, by differentiating both sides, we obtain that
\begin{align*}
	\Ly_t'&=-(\gamma - \beta_t) \inner{\grad \K(p_t)}{p_t} - \beta_t \gamma \inner{x_t-\xmin}{p_t} - \beta_t\inner{x_t-\xmin}{\grad \f(x_t)}+\beta_t' \inner{x_t-\xmin}{p_t}
\intertext{The first three terms are equivalent to the temporal derivative of $\Ly_t$ with constant $\beta = \beta_t$. Since $\alpha_t \leq \alpha(\K(p_t))$ and $\beta_t \leq \gamma$, the assumptions of Lemma \ref{lem:Vderivativebound} are satisfied locally for $\alpha_t$, $\beta_t$ and we get}
	\Ly_t'&\le -\lambda(\alpha_t,\beta_t,\gamma) \Ly_t+\beta_t' \inner{x_t-\xmin}{p_t} =  -\lambda(\alpha_t,\beta_t,\gamma) \Ly_t+\Ca \alpha_t'  \inner{x_t-\xmin}{p_t} \Ly_t'.
\end{align*}
Using \eqref{eq:goodenoughbeta} of Lemma \ref{lem:Vderivativebound} for $\alpha_t, \beta_t$, we have $\lambda(\alpha_t, \beta_t, \gamma) = \tfrac{\beta_t (1-\gamma)}{1-\beta_t} \geq \beta_t (1-\gamma)$ and
\begin{equation*}
	\Ly_t'\le -\beta_t(1-\gamma) \Ly_t+\Ca \alpha_t' \inner{x_t-\xmin}{p_t} \Ly_t'.
\end{equation*}
Using  \eqref{eq:boundinner} we have $\Ly_t'\le -\frac{\beta_t(1-\gamma)}{2} \Ly_t$. Notice that $\Ly_0=\Ha_0$ since we have assumed that $p_0=0$, and the claim of the lemma follows by Gr{\"o}nwall's inequality. The final inequality \eqref{eq:Hdecay} follows from the fact that $\alpha(2 \Ly_t) \geq \alpha(3\Ha_0) = \alphastar$.
\end{proof}

\begin{figure}[t]
    \centering
    \includegraphics[scale=1]{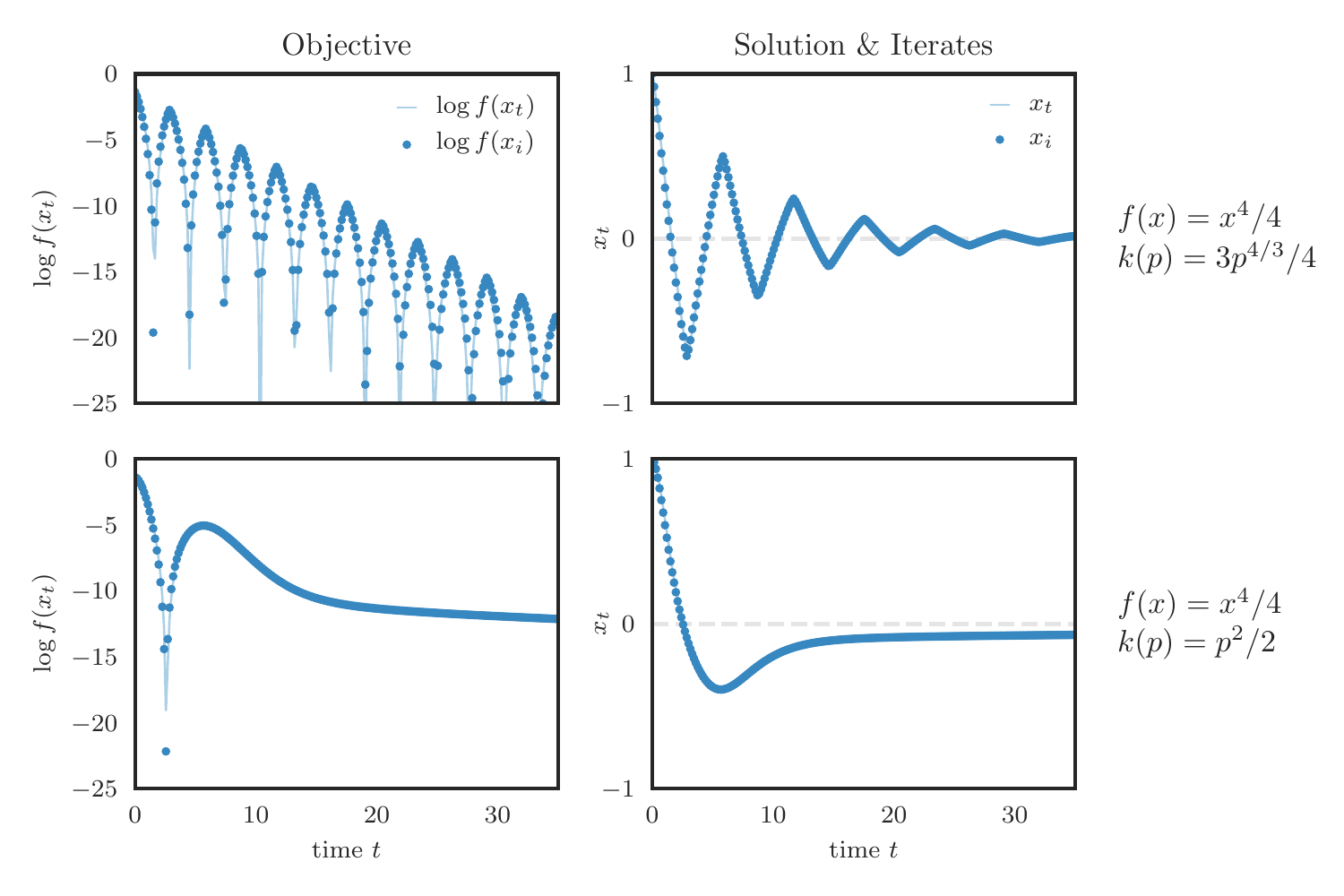}
    \caption{Importance of Assumptions \ref{ass:cont}. Solutions $x_t$ and iterates $x_i$ of our first explicit method on $f(x) = x^4/4$ with two different choices of $\K$. Notice that $\fc^*(p) = 3p^{4/3}/4$ and thus $\K(p) = p^2/2$ cannot be made to satisfy assumption \ref{ass:cont:alpha}.}
    \label{fig:contass}
\end{figure}

\subsection{Partial Lower Bounds}
\label{sec:lower}

In this section we consider a partial converse of Proposition \ref{lem:continuouslyap}, showing in a simple setting that if the assumption $k(p)\ge \alpha \max(f^*_c(p),f^*_c(-p))$ of \ref{ass:cont:alpha} is violated, then the ODE \eqref{eq:ode} contracts sub-linearly. Figure \ref{fig:contass} considers the example $f(x) = x^4/4$. If $\K(p) = |p|^{\ppower}/{\ppower}$, then assumptions \ref{ass:cont} cannot be satisfied for small $p$ unless ${\xpower} \geq 4/3$. Figure \ref{fig:contass} shows that an inappropriate choice of $\K(p) = p^2/2$ leads to sub-linear convergence both in continuous time and for one of the discretizations of Section \ref{sec:discrete}. In contrast, the choice of $\K(p) = 3p^{4/3}/4$ results in linear convergence, as expected.

Let ${\xpower}, {\ppower} >1$ and $\gamma > 0$. For $d=1$ dimension, with the choice $f(x):=|x|^{\xpower}/{\xpower}$ and $k(p):=|p|^{\ppower}/{\ppower}$, \eqref{eq:ode} takes the following form,
\begin{equation}\label{eq:odepowerab}
\begin{aligned}
	x_t' &= |p_t|^{{\ppower}-1} \sgn(p_t), \\
	p_t' &= -|x_t|^{{\xpower}-1} \sgn(x_t) - \gamma p_t.
\end{aligned}
\end{equation}

\begin{figure}[t]
    \centering
    \includegraphics[scale=1]{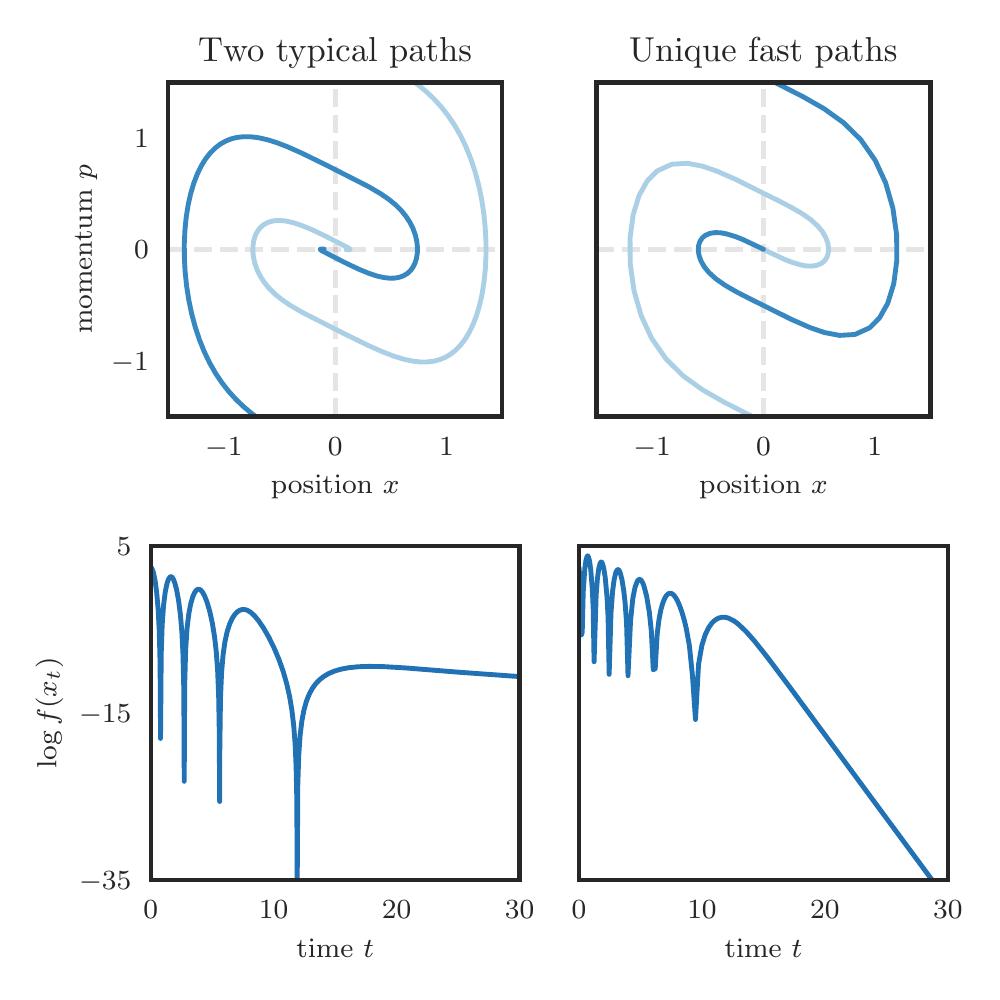}
    \caption{Solutions to the Hamiltonian descent system with $f(x) = x^4/4$ and $k(p) = x^2/2$. The right plots show a numerical approximation of $(x^{(\eta)}_t, p^{(\eta)}_t)$ and $(-x^{(\eta)}_t, -p^{(\eta)}_t)$. The left plots show a numerical approximation of $(x^{(\theta)}_t, p^{(\theta)}_t)$ and $(-x^{(\theta)}_t, -p^{(\theta)}_t)$ for $\theta = \eta + \delta \in \R$, which represent typical paths.}
    \label{fig:fastandslow}
\end{figure}

Since $f(x)$ takes its minimum at $0$, $(x_t,p_t)$ are expected to converge to $(0,0)$ as $t\to \infty$. There is a trivial solution: $x_t = p_t = 0$ for every $t\in \R$. The following Lemma shows an existence and uniqueness result for this equation. The proof is included in Section \ref{sec:proofslowerbounds} of the Appendix.

\begin{restatable}[Existence and uniqueness of solutions of the ODE]{lem}{Lemmasolutionsexistunique}
\label{Lemma-solutions-exist-unique}
	Let ${\ppower},{\xpower},\gamma \in (0, \infty)$. For every $t_0 \in \R$ and $(x, p) \in \R^2$, there is a unique solution $(x_t,p_t)_{t\in \R}$ of the ODE \eqref{eq:odepowerab} with $x_{t_0} = x$, $p_{t_0} = p$.	Either $x_t = p_t = 0$ for every $t \in \R$, or $(x_t, p_t) \neq (0,0)$ for every $t \in \R$.
\end{restatable}

Note that if $(x_t, p_t)$ is a solution, and $\Delta \in \R$, then $(x_{t+\Delta}, p_{t+\Delta})$ is also a solution (time translation), and $(-x_t, -p_t)$ is also a solution (central symmetry).

Note also that $f^*(p)=f^*(-p)=|p|^{{\xpower}^*}/{\xpower}^*$ for ${\xpower}^*:=(1-\frac{1}{{\xpower}})^{-1}$. Hence if ${\ppower}\le {\xpower}^*$, or equivalently, if $\frac{1}{{\xpower}}+\frac{1}{{\ppower}}\ge 1$, the conditions of Proposition \ref{lem:continuouslyap} are satisfied for some $\alpha>0$ (in particular, if ${\ppower}={\xpower}^*$, then $\alpha=1$ independently of $x_0,p_0$). Hence in such cases, the speed of convergence is linear. For ${\ppower}>{\xpower}^*$,  $\lim_{p\to 0}
\frac{K(p)}{f^*(p)}=0$, so the conditions of Proposition \ref{lem:continuouslyap} are violated.

Now we are ready to state the main result in this section, a theorem characterizing the convergence speeds of $(x_t,p_t)$ to $(0,0)$ in this situation. The proof is included in Section \ref{sec:proofslowerbounds} of the Appendix.

\begin{restatable}[Lower bounds on the convergence rate in continuous time]{prop}{thmspeedofconv}
\label{thmspeedofconv}
Suppose that $\frac{1}{{\xpower}}+\frac{1}{{\ppower}}<1$. For any $\theta\in \R$, we denote by $(x^{(\theta)}_t, p^{(\theta)}_t)$ the unique solution of \eqref{eq:odepowerab} with $x_0=\theta, p_0=0$. 
Then there exists a constant $\eta \in (0, \infty)$ depending on ${\ppower}$ and ${\xpower}$ such that the path $(x^{(\eta)}_t, p^{(\eta)}_t)$
and its mirrored version $(x^{(-\eta)}_t, p^{(-\eta)}_t)$ satisfy that
\[
|x^{(-\eta)}_t|=|x^{(\eta)}_t|\le \Oh(\exp(-\alpha t)) \text{ for every }\alpha<\gamma({\ppower}-1) \text{ as }t\to \infty.
\]
For any path $(x_t,p_t)$ that is not a time translation of $(x^{(\eta)}_t, p^{(\eta)}_t)$ or $(x^{(-\eta)}_t, p^{(-\eta)}_t)$, we have
\[\l|x_t^{-1}\r| = O(t^{\frac{1}{{\xpower}{\ppower}-{\xpower}-{\ppower}}})\text{ as }t\to \infty,\]
so the speed of convergence is sub-linear and not linearly fast.
\end{restatable}

Figure \ref{fig:fastandslow} illustrates the two paths where the convergence is linearly fast for ${\ppower}=2,{\xpower}=4$. The main idea in the proof of Proposition \ref{thmspeedofconv} is that we establish the existence of a class of trapping sets, i.e. once the path of the ODE enters one of them, it never escapes. Convergence rates within such sets can be shown to be logarithmic, and it is established that only two paths (which are symmetric with respect to the origin) avoid each one of the trapping sets, and they have linear convergence rate.

\section{Optimization Algorithms}
\label{sec:discrete}

In this section we consider three discretizations of the continuous system \eqref{eq:ode}, one implicit and two explicit. For these discretizations we must assume more about the relationship between $f$ and $\K$. The implicit method defines the iterates as solution of a local subproblem. The first and second explicit methods are fully explicit, and we must again make stronger assumptions on $f$ and $\K$. The proofs of all of the results in this section are given in Section \ref{sec:proofsdiscretizations} of the Appendix.

\subsection{Implicit Method}
Consider the following discrete approximation $(x_i, p_i)$ to the continuous system, making the fixed $\epsilon > 0$ finite difference approximation, $\tfrac{x_{i+1} - x_i}{\epsilon} = x_t'$ and $\tfrac{p_{i+1} - p_i}{\epsilon} = p_t'$, which approximates the field at the forward points.
\begin{equation}
\label{eq:implicit} \begin{aligned}
\frac{x_{i+1} - x_i}{\epsilon} &= \nabla \K(p_{i+1})\\
\frac{p_{i+1} - p_i}{\epsilon} &= - \gamma p_{i+1} - \nabla f(x_{i+1}).
\end{aligned}
\end{equation}
Since $\grad \K^*(\grad \K(p)) = p$, this system of equations corresponds to the stationary condition of the following subproblem iteration, which we introduce as our implicit method.

\begin{titledbox}{Implicit Method}
Given $f, \K : \R^d \to \R$, $\epsilon, \gamma \in (0, \infty)$, $x_0, p_0 \in \R^d$.

Let $\delta = (1 + \gamma \epsilon)^{-1}$ and
\begin{equation}
\label{eq:implicit2}
\begin{aligned}
x_{i+1} &= \arg \min_{x \in \R^d}\l \{ \epsilon \K^*(\tfrac{x-x_i}{\epsilon}) + \epsilon \delta f(x) - \delta \inner{p_i}{x}\r\}\\
p_{i+1}  &= \delta p_i - \epsilon \delta \grad \f(x_{i+1}).
\end{aligned}
\end{equation}
\end{titledbox}
The following lemma shows that the formulation \eqref{eq:implicit2} is well defined. The proof is included in Section \ref{sec:proofsdiscretizations} of the Appendix.
\begin{restatable}[Well-definedness of the implicit scheme]{lem}{impwelldeflemma}
\label{lem:impwelldef}
	Suppose that $f$ and $\K$  satisfy assumptions \ref{ass:cont:fconvex} and \ref{ass:cont:kconvex}, and $\epsilon,\gamma\in (0,\infty)$. Then \eqref{eq:implicit2} has a unique solution for every $x_i,p_i\in \R^d$, and this solution also satisfies  \eqref{eq:implicit}.
\end{restatable}

As this discretization involves solving a potentially costly subproblem at each iteration, it requires a relatively light assumption on the compatibility of $f$ and $\K$.

\begin{assumptions}{imp}
		\item\label{ass:imp:innerproductfk} There exists $\Cfk \in (0, \infty)$ such that for all $x,p\in \R^d$,
    			\begin{equation}\label{eq:gfgKimp}
    			|\inner{\grad f(x)}{\grad \K(p)}| \le \Cfk\Ha(x,p).
	    		\end{equation}
\end{assumptions}

\begin{remark}
Smoothness of $f$ implies $\tfrac{1}{2}\norm{\grad f(x)}_2^2 \leq L (f(x) - f(\xmin))$ (see (2.1.7) of Theorem 2.1.5 of \cite{nesterov2013introductory}). Thus, if $f$ is smooth and $\K(p) = \tfrac{1}{2} \norm{p}_2^2$, then the assumption \ref{ass:imp:innerproductfk} can be satisfied by $\Cfk = \max\{1, L\}$, since
\begin{equation*}
|\inner{\grad f(x)}{\grad \K(p)}| \le \tfrac{1}{2}\norm{\grad f(x)}_2^2 + \tfrac{1}{2}\norm{\grad \K(p)}_2^2 \leq L(f(x) - f(\xmin)) + \K(p).
\end{equation*}
\end{remark}

The following proposition shows a convergence result for the implicit scheme.

\begin{restatable}[Convergence bound for the implicit scheme]{prop}{implemma}
\label{lem:imp}
	Given $f$, $\K$, $\gamma$, $\alpha$, $\Ca$, and $\Cfk$ satisfying assumptions \ref{ass:cont} and \ref{ass:imp}. Suppose that $\epsilon<\frac{1-\gamma}{2\max(\Cfk,1)}$. Let   $\alphastar = \alpha(3\Ha_0)$, and let $\Wy_0=f(x_0)-f(\xmin)$ and for $i\ge 0$,
\[\Wy_{i+1}=\Wy_i \l[1+ \epsilon \Ca   (1-\gamma- 2 \Cfk \epsilon) \alpha(2\Wy_i)/4\r]^{-1}.\]
 Then for any $(x_0,p_0)$ with $p_0=0$, the iterates of \eqref{eq:implicit} satisfy for every $i\ge 0$,
	\[
	f(x_i)-f(\xmin) \le 2\Wy_i \leq 2\Wy_0 [1 + \epsilon \Ca   (1-\gamma- 2 \Cfk \epsilon) \alphastar/4]^{-i}.
	\]
\end{restatable}

\begin{remark}
\label{remark:imp}
Proposition \ref{lem:imp} means that we can fix any step size $0<\epsilon<\frac{1-\gamma}{2\max(\Cfk,1)}$ independently of the initial point, and have linear convergence with contraction rate that is proportional to $\alpha(3\Ha_0)$ initially and possibly increasing as we get closer to the optimum. In Section \ref{sec:kin} we introduce kinetic energies $\K(p)$ that behave like $\norm{p}_2^a$ near 0 and $\norm{p}_2^A$ in the tails. We will show that for functions $f(x)$ that behave like $\norm{x-\xmin}_2^b$ near their minima and $\norm{x-\xmin}_2^B$ in the tails the conditions of assumptions \ref{ass:imp} are satisfied as long as $\frac{1}{a}+\frac{1}{b}=1$ and $\frac{1}{A}+\frac{1}{B}\ge 1$. In particular, if we choose $\K(p)=\sqrt{\norm{p}_2^2 + 1}-1$ (relativistic kinetic energy), then $a=2$ and $A=1$, and assumptions \ref{ass:imp} can be shown to hold for every $f$ that has quadratic behavior near its minimum and no faster than exponential growth in the tails.
\end{remark}



\subsection{First Explicit Method, with Analysis via the Hessian of $f$}
The following discrete approximation $(x_i, p_i)$ to the continuous system makes a similar finite difference approximation, $\tfrac{x_{i+1} - x_i}{\epsilon} = x_t'$ and $\tfrac{p_{i+1} - p_i}{\epsilon} = p_t'$ for $\epsilon > 0$. In contrast to the implicit method, it approximates the field at the point $(x_i, p_{i+1})$, making it fully explicit without any costly subproblem,
\begin{align*}
\frac{x_{i+1} - x_i}{\epsilon} &= \grad \K(p_{i+1})\\
\frac{p_{i+1} - p_i}{\epsilon} &= - \gamma p_{i+1} - \grad \f(x_i).
\end{align*}
This method can be rewritten as our first explicit method.
\begin{titledbox}{First Explicit Method}
Given $f, \K : \R^d \to \R$, $\epsilon, \gamma \in (0, \infty)$, $x_0, p_0 \in \R^d$.

Let $\delta = (1 + \gamma \epsilon)^{-1}$ and
\begin{equation}
\label{eq:semiA} \begin{aligned}
p_{i+1}  &=\delta p_i - \epsilon \delta\grad \f(x_i)\\
x_{i+1} &= x_i + \epsilon \grad \K(p_{i+1}).
\end{aligned}
\end{equation}
\end{titledbox}
This discretization exploits the convexity of $\K$ by approximating the continuous dynamics at the forward point $p_{i+1}$, but is made explicit by approximating at the backward point $x_i$. Because this method approximates the field at the backward point $x_i$ it requires a kind of smoothness assumption to prevent $f$ from changing too rapidly between iterates. This assumption is in the form of a condition on the Hessian of $f$, and thus we require twice differentiability of $f$ for the first explicit method. Because the accumulation of gradients of $f$ in the form of $p_i$ are modulated by $\K$, this condition in fact expresses a requirement on the interaction between $\grad \K$ and $\hess f$, see assumption \ref{ass:semiA:gKHfsemi}.

\begin{assumptions}{semiA}
	\item \label{ass:semiA:gKpsem}There exists $\Ck \in (0, \infty)$ such that for every $p\in \R^d$,
            	\begin{equation}
			\inner{\grad \K(p)}{p}\le \Ck \K(p).
            	\end{equation}
	\item \label{ass:semiA:ftwicediff} $\f : \R^d \to \R$  convex with a unique minimum at $\xmin$ and twice continuously differentiable for every $x\in \R^d\setminus \{\xmin\}$.
	\item \label{ass:semiA:gKHfsemi} There exists $\Dfk \in (0, \infty)$ such that for every $p\in \R^d$, $x\in \R^d\setminus \{\xmin\}$,
                	\begin{equation}
			\inner{\grad \K(p)}{ \grad^2 f(x) \grad \K(p) }\le \Dfk  \alpha(3\Ha(x,p))  \Ha(x,p).
                	\end{equation}
\end{assumptions}

\begin{remark}
If $f$ smooth and twice differentiable then $\inner{v}{\hess f(x) v}$ is everywhere bounded by $L$ for $v \in \R^d$ such that $\norm{v}_2 = 1$ (see Theorem 2.1.6 of \cite{nesterov2013introductory}). Thus, using $\K(p) = \tfrac{1}{2} \norm{p}_2^2$, this allows us to satisfy assumption \ref{ass:semiA:gKHfsemi} with $\Dfk = \max\{1, 2L\}$, since
\begin{equation*}
\inner{\grad \K(p)}{ \grad^2 f(x) \grad \K(p) } \le L \norm{\grad \K(p)}^2_2  = 2L \K(p) \leq f(x) - f(\xmin) + 2L \K(p).
\end{equation*}
Assumption \ref{ass:semiA:gKpsem} is clearly satisfied in this case by $\Ck = 2$.
\end{remark}


The following lemma shows a convergence result for this discretization.

\begin{restatable}[Convergence bound for the first explicit scheme]{prop}{semiAlemma}
\label{lem:semiA}
	Given $f$, $\K$, $\gamma$, $\alpha$, $\Ca$, $\Cfk$, $\Ck$, $\Dfk$ satisfying assumptions \ref{ass:cont}, \ref{ass:imp}, and \ref{ass:semiA}, and that $0<\epsilon<\min\l(\frac{1-\gamma}{2\max(\Cfk+6\Dfk/\Ca,1)},\frac{\Ca}{10\Cfk+5\gamma \Ck} \r)$.	Let $\alphastar = \alpha(3\Ha_0)$, $\Wy_0:=f(x_0)-f(\xmin)$, and for $i\ge 0$, let
	\[\Wy_{i+1}=\Wy_i \l(1+ \frac{\epsilon  \Ca}{4} \l[1-\gamma-2\epsilon(\Cfk+6\Dfk/\Ca)\r] \alpha(2\Wy_i) \r)^{-1}.\]
 Then for any $(x_0,p_0)$ with $p_0=0$, the iterates \eqref{eq:semiA} satisfy for every $i\ge 0$,
	\[
	f(x_i)-f(\xmin)\le  2\Wy_i\le 2\Wy_0 \l(1+ \frac{\epsilon  \Ca}{4} \l[1-\gamma-2\epsilon(\Cfk+6\Dfk/\Ca)\r] \alphastar \r)^{-i}.
	\]
\end{restatable}

\begin{remark}
Similar to Remark \ref{remark:imp}, Proposition \ref{lem:semiA} implies that, under suitable assumptions and position independent step sizes, the first explicit method can achieve linear convergence with contraction rate that is proportional to $\alpha(3\Ha_0)$ initially and possibly increasing as we get closer to the optimum. In particular, again as remarked in Remark \ref{remark:imp}, for $f(x)$ that behave like $\norm{x-\xmin}_2^b$ near their minima and $\norm{x-\xmin}_2^B$ in the tails the conditions of assumptions \ref{ass:semiA} can be satisfied for kinetic energies that grow like $\norm{p}_2^a$ in the body and $\norm{p}_2^A$ in the tails as long as $\frac{1}{a}+\frac{1}{b}=1$, $\frac{1}{A}+\frac{1}{B}\ge 1$. The distinction here is that for the first explicit method we will require $b, B \ge 2$.
\end{remark}


\subsection{Second Explicit Method, with Analysis via the Hessian of $\K$}
Our second explicit method inverts relationship between $f$ and $\K$ from the first. Again, it makes a fixed $\epsilon$ step approximation $\tfrac{x_{i+1} - x_i}{\epsilon} = x_t'$ and $\tfrac{p_{i+1} - p_i}{\epsilon} = p_t'$. In contrast to the implicit \eqref{eq:implicit} and first explicit \eqref{eq:semiA} methods, it approximates the field at the point $(x_{i+1}, p_{i})$.

\begin{titledbox}{Second Explicit Method}
Given $f, \K : \R^d \to \R$, $\epsilon, \gamma \in (0, \infty)$, $x_0, p_0 \in \R^d$. Let,
\begin{equation}
\label{eq:semiB} \begin{aligned}
x_{i+1} &= x_i + \epsilon \grad \K(p_{i})\\
p_{i+1} &= (1 - \epsilon \gamma) p_i -\epsilon \grad \f(x_{i+1}).
\end{aligned}
\end{equation}
\end{titledbox}

This discretization exploits the convexity of $f$ by approximating the continuous dynamics at the forward point $x_{i+1}$, but is made explicit by approximating at the backward point $p_i$. As with the other explicit method, it requires a smoothness assumption to prevent $\K$ from changing too rapidly between iterates, which is expressed as a requirement on the interaction between $\grad f$ and $\hess \K$, see assumption \ref{ass:semiB:gfHKsemi}. These assumptions can be satisfied for $\K$ that have quadratic or higher power growth and are suitable for $f$ that may have unbounded second derivatives at their minima (for such $f$, Assumptions \ref{ass:semiA} can not hold).

\begin{assumptions}{semiB}
	\item \label{ass:semiB:Ktwicediff}$\K : \R^d \to \R$  strictly convex with minimum $\K(0) = 0$ and twice continuously differentiable for every $p\in \R^d\setminus \{0\}$.
	\item \label{ass:semiB:gKpsem}There exists $\Ck \in (0, \infty)$ such that for every $p\in \R^d$,
	\begin{equation}
	\inner{\grad \K(p)}{p}\le \Ck \K(p).
	\end{equation}
	\item \label{ass:semiB:gHpsem} There exists $\Dk \in (0, \infty)$ such that for every $p\in \R^d\setminus \{0\}$,
            	\begin{equation}
			\inner{p}{\hess \K(p)p}\le \Dk \K(p).
            	\end{equation}
	\item \label{ass:semiB:weird}There exists $E_{\K}, F_{\K} \in (0, \infty)$ such that for every $p,q\in \R^d$,
            	\begin{equation}
			\K(p) - \K(q) \leq E_{\K} \K(q) + F_{\K} \inner{\grad \K(p) - \grad \K(q)}{p-q}.
            	\end{equation}
	\item \label{ass:semiB:gfHKsemi} There exists $\Dfk \in (0, \infty)$ such that for every $x\in \R^d$, $p\in \R^d\setminus \{0\}$,
                	\begin{equation}
			\inner{\grad f(x)}{ \hess \K(p) \grad f(x) }\le \Dfk  \alpha(3\Ha(x,p))  \Ha(x,p).
                	\end{equation}
\end{assumptions}

\begin{remark}
Smoothness of $f$ implies $\tfrac{1}{2}\norm{\grad f(x)}_2^2 \leq L (f(x) - f(\xmin))$ (see (2.1.7) of Theorem 2.1.5 of \cite{nesterov2013introductory}). Thus, if $f$ is smooth and $\K(p) = \tfrac{1}{2} \norm{p}_2^2$, then the assumption \ref{ass:semiB:gfHKsemi} can be satisfied by $\Dfk = \max\{1, 2L\}$, since $\hess \K(p) = I$ and
\begin{equation*}
\inner{\grad f(x)}{ \hess \K(p) \grad f(x) } = \norm{\grad f(x)}_2^2 \leq 2L(f(x) - f(\xmin)) \leq 2L(f(x) - f(\xmin)) + \K(p).
\end{equation*}
The $\K$-specific assumptions \ref{ass:semiB:gKpsem} and \ref{ass:semiB:gHpsem} can clearly be satisfied with $\Ck = \Dk = 2$ in this case. We show that \ref{ass:semiB:weird} can be satisfied in Section \ref{sec:kin}.
\end{remark}

\begin{restatable}[Convergence bound for the second explicit scheme]{prop}{semiBlemma}
\label{lem:semiB}
	Given $f$, $\K$, $\gamma$, $\alpha$, $\Ca$, $\Cfk$, $\Ck$, $\Dk$, $\Dfk$, $E_{\K}$, $F_{\K}$ satisfying assumptions \ref{ass:cont}, \ref{ass:imp}, and \ref{ass:semiB}, and that
	\[0<\epsilon<\min\l(\frac{1-\gamma}{2(\Cfk+6\Dfk/\Ca)},\frac{1-\gamma}{8\Dk(1+E_{\K})},\frac{\Ca}{6 (5\Cfk+2\gamma \Ck)+12\gamma\Ca },\sqrt{\frac{1}{6\gamma^2 \Dk F_{\K}}}  \r).\]
	Let  $\alphastar=\alpha(3 \Ha_0)$, $\Wy_0:=f(x_0)-f(\xmin)$, and for $i\ge 0$, let
	\[\Wy_{i+1}=\Wy_i \l(1- \frac{\epsilon  \Ca}{4} \l[1-\gamma-2\epsilon(\Cfk+6\Dfk/\Ca)  \r] \alpha(2\Wy_i) \r).\]
	Then for any $(x_0,p_0)$ with $p_0=0$, the iterates \eqref{eq:semiB} satisfy for every $i\ge 0$,
	\[
	f(x_i)-f(\xmin)\le  2\Wy_i\le 2\Wy_0\cdot  \l(1- \frac{\epsilon  \Ca}{4}\l[1-\gamma -2\epsilon(\Cfk+6\Dfk/\Ca)  \r] \alphastar\r)^{i}.
	\]
\end{restatable}

\begin{remark}
Similar to Remark \ref{remark:imp}, Proposition \ref{lem:semiB} implies that, under suitable assumptions and for a fixed step size independent of the initial point, the second explicit method can achieve linear convergence with contraction rate that is proportional to $\alpha(3\Ha_0)$ initially and possibly increasing as we get closer to the optimum. In particular, again as remarked in Remark \ref{remark:imp}, for $f(x)$ that behave like $\norm{x-\xmin}_2^b$ near their minima and $\norm{x-\xmin}_2^B$ in the tails the conditions of assumptions \ref{ass:semiB} can be satisfied for kinetic energies that grow like $\norm{p}_2^a$ in the body and $\norm{p}_2^A$ in the tails as long as $\frac{1}{a}+\frac{1}{b}=1$, $\frac{1}{A}+\frac{1}{B}\ge 1$. The distinction here is that for the second explicit method we will require $b, B \le 2$.
\end{remark}

\begin{figure}[t]
    \centering
    \includegraphics[scale=1]{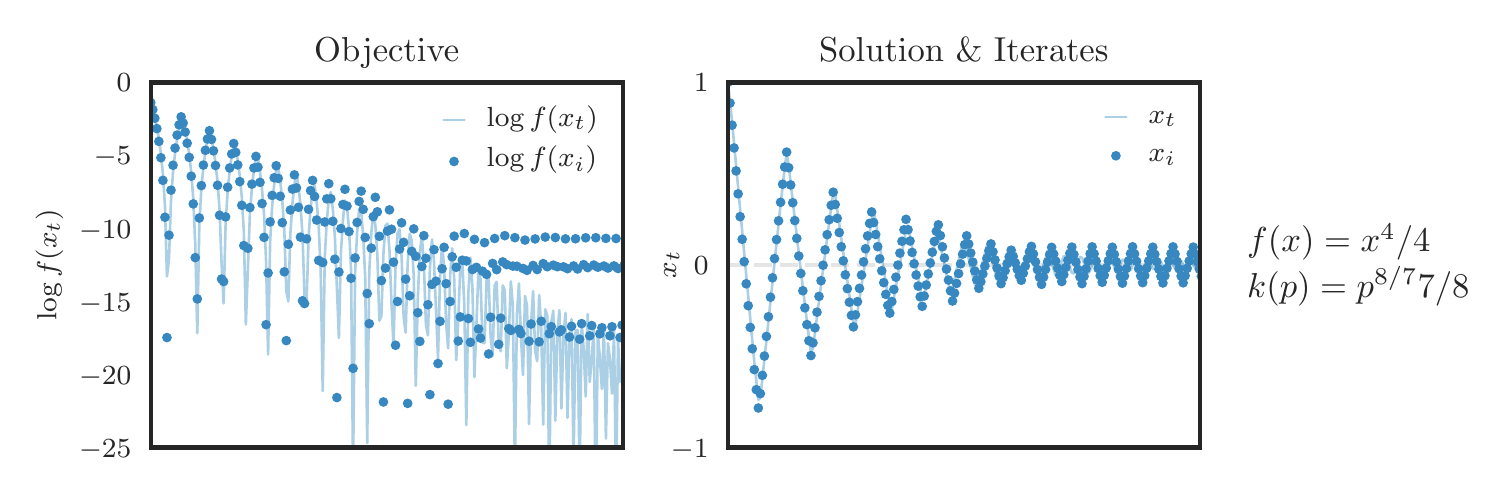}
    \caption{Importance of discretization assumptions. Solutions $x_t$ and iterates $x_i$ of our first explicit method on $f(x) = x^4/4$. With an inappropriate choice of kinetic energy, $k(p) = p^{8/7}7/8$, the continuous solution converges at a linear rate but the iterates do not.}
    \label{fig:discass}
\end{figure}

To conclude the analysis of our methods on convex functions, consider the example $f(x) = x^4/4$ from Figure \ref{fig:contass}. If we take $\K(p) = |p|^a/a$, then assumption \ref{ass:cont:alpha} requires that $a \leq 4/3$. Assumptions \ref{ass:imp} and \ref{ass:semiA} cannot be satisfied as long as $a < 4/3$, which suggests that $\K(p) = f^*(p)$ is the only suitable choice in this case. Indeed, in Figure \ref{fig:discass}, we see that the choice of $\K(p) = p^{8/7}7/8$ results in a system whose continuous dynamics converge at a linear rate and whose discrete dynamics fail to converge. Note that as the continuous systems converge the oscillation frequency increases dramatically, making it difficult for a fixed step size scheme to approximate.

\subsection{First Explicit Method on Non-Convex $f$}

We close this section with a brief analysis of the convergence of the first explicit method on non-convex $f$.  A traditional requirement of discretizations is some degree of smoothness to prevent the function changing too rapidly between points of approximation. The notion of Lipschitz smoothness is the standard one, but the use of the kinetic map $\grad \K$ to select iterates allows Hamiltonian descent methods to consider the broader definition of uniform smoothness, as discussed in \cite{zalinescu1983uniformly, penot1995unifconvex, zalinescu2002convex} but specialized here for our purposes.

Uniform smoothness is defined by a norm $\norm{\c}$ and a convex non-decreasing function $\sigma : [0, \infty) \to [0, \infty]$ such that $\sigma(0) = 0$. A function $f : \R^d \to \R$ is $\sigma$-uniformly smooth, if for all $x,y \in \R^d$,
\begin{equation}
\label{eq:defunifsmooth} f(y) \leq f(x) + \inner{\grad f(x)}{y-x} + \sigma(\norm{y-x}).
\end{equation}
Lipschitz smoothness corresponds to $\sigma(t) = \tfrac{1}{2}t^2$, and generally speaking there exist non-trivial uniformly smooth functions for $\sigma(t) = \tfrac{1}{b}t^b$ for $1 < b \leq 2$, see, e.g., \cite{nesterov2008accelerating, zalinescu1983uniformly, penot1995unifconvex, zalinescu2002convex}.

\begin{assumptions}{smoothnonconvex}
	\item $\f : \R^d \to \R$ differentiable.
	\item $\gamma \in (0, \infty)$.
	\item \label{ass:semiAsmoothness} There exists a norm $\norm{\cdot}$ on $\R^d$, $b  \in (1, \infty)$, $\Dk \in (0, \infty)$, $\Df \in (0, \infty)$, $\sigma : [0, \infty) \to [0, \infty]$ non-decreasing convex such that $\sigma(0) = 0$ and $\sigma(ct) \leq c^b\sigma(t)$ for $c,t \in (0, \infty)$; for all $p \in \R^d$,
		\begin{equation}
			\sigma(\norm{\grad \K(p)}) \leq \Dk \K(p);
		\end{equation}
		and for all $x,y \in \R^d$,
		\begin{equation}
			f(y) \leq f(y) + \inner{\grad f(x)}{y-x} + \Df \sigma(\norm{y-x}).
		\end{equation}
\end{assumptions}

\begin{restatable}[Convergence of the first explicit scheme without convexity]{lem}{explicitmethodnonconvex}
\label{lem:explicitmethodnonconvex}
Given $\norm{\cdot}$, $f$, $\K$, $\gamma$, $b$, $\Dk$, $\Df$, $\sigma$ satisfying assumptions \ref{ass:smoothnonconvex} and \ref{ass:cont:kconvex}. If $\epsilon \in (0,  \sqrt[b-1]{\gamma/\Df\Dk}]$, then the iterates \eqref{eq:semiA} of the first explicit method satisfy
\begin{equation}
\Ha_{i+1}  - \Ha_i \leq (\epsilon^b \Df\Dk - \epsilon \gamma) \K(p_{i+1}) \leq 0,
\end{equation}
and $\norm{\grad \f(x_i)}_2 \to 0$.
\end{restatable}


\begin{remark}
$L$-Lipschitz continuity of the gradients $\norm{\grad f(x) - \grad f(y)}_2 \leq L \norm{x-y}_2$ for $L > 0$ with Euclidean norm $\norm{\c}_2$ implies both $f(y) \leq f(y) + \inner{\grad f(x)}{y-x} + \tfrac{L}{2}\norm{y-x}_2^2$ and $\tfrac{1}{2} \norm{\grad f(x)}_2^2 \leq L(f(x) - f(\xmin))$. Thus, if $f, \K$ are $\Lf, \Lk$ smooth, respectively, then the condition for convergence simplifies to  $\epsilon \leq \gamma / \Lf \Lk$.
\end{remark}

\section{Kinetic Maps for Functions with Power Behavior}
\label{sec:kin}

In this section we design a family of kinetic maps $\grad \K$ suitable for a class of functions $f$ that exhibit power growth, which we will describe precisely as a set of assumptions. This class includes strongly convex and smooth functions. However, it is much broader, including functions with possibly non-quadratic power behavior and singular or unbounded Hessians. First, we show that this family of kinetic energies satisfies the $\K$-specific assumptions of Section \ref{sec:discrete}. Then we use the generic analysis of Section \ref{sec:discrete} to provide a specific set of assumptions on $f$s and their match to the choice of $\K$. As a consequence, this analysis greatly extends the class of functions for which linear convergence is possible with fixed step size first order computation. Still, this analysis is not meant to be an exhaustive catalogue of possible kinetic energies for Hamiltonian descent. Instead, it serves as an example of how known properties of $f$ can be used to design $\K$. Note that, with a few exceptions, the proofs of all of our results in this section are deferred to Section \ref{sec:proofskinetic} of the Appendix.

\subsection{Power Kinetic Energies}

\begin{figure}[t]
    \centering
    \includegraphics[scale=1]{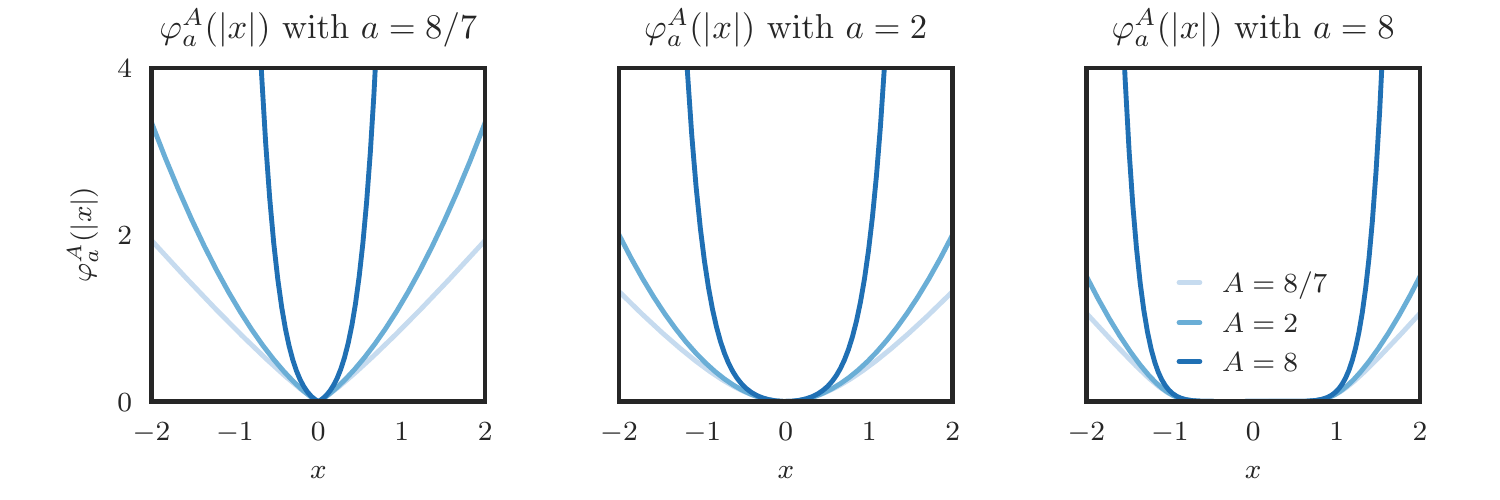}
    \caption{Power kinetic energies in one dimension.}
    \label{fig:kin}
\end{figure}

We assume a given norm $\norm{x}$ and its dual $\norm{p}_* = \sup \{ \inner{x}{p} : \norm{x} \leq 1\}$ for $x,p \in \R^d$. Define the family of power kinetic energies $\K$,
\begin{equation}
\label{eq:kindef} \K(p) = \varphi_a^A(\norm{p}_*) \text{ where } \varphi_a^A(t) = \tfrac{1}{A}\l(t^a + 1\r)^{\tfrac{A}{a}} - \tfrac{1}{A} \text{ for } t \in [0, \infty) \text{ and } a,A \in [1, \infty).
\end{equation}
For $a = A$ we recover the standard power functions, $\varphi_a^a(t) = t^a/a$. For distinct $a \neq A$, we have $(\varphi_a^A)'(t) \sim t^{A-1}$ for large $t$ and $(\varphi_a^A)'(t)  \sim t^{a-1}$ for small $t$. Thus, $\K(p) \sim  \norm{p}_*^{A}/A$ as $\norm{p}_* \uparrow \infty$ and $\K(p) \sim \norm{p}_*^a/a$ as $\norm{p}_* \downarrow 0$. See Figure \ref{fig:kin} for examples from this family in one dimension.

Broadly speaking, this family of kinetic energies must be matched in a conjugate fashion to the body and tail behavior of $f$. Informally, for this choice of $\K$ we will require conditions on $f$ that  correspond to requiring that it grows like $\norm{x-\xmin}^{b}$ in the body (as $\norm{x-\xmin} \to 0$) and $\norm{x-\xmin}^{B}$ in the tails (as $\norm{x-\xmin} \to \infty$) for some $b,B \in (1, \infty)$. In particular, our growth conditions in the case of $f$ ``growing like'' $\norm{x}_2^2 = \inner{x}{x}$ everywhere will be necessary conditions of strong convexity and smoothness. More generally, $a,A,b,B$ will be well-matched if $1/a + 1/b = 1/A + 1/B = 1$, but other scenarios are possible. Of these, the conjugate relationship between $a$ and $b$ is the most critical; it captures the asymptotic match between $f$ and $\K$ as $(x_i, p_i) \to (\xmin, 0)$, and our analysis requires that $1/a + 1/b = 1$. The match between $A$ and $B$ is less critical. In the ideal case, $B$ is known and $A = B/(B-1)$. In this case, the discretizations will converge at a constant fast linear rate. If $B$ is not known, it suffices for $1/A + 1/B \geq 1$. The consequence of underestimating $A < B/(B-1)$ will be reflected in a linear, but non-constant, rate of convergence (via $\alpha$ of Assumption \ref{ass:cont:alpha}), which depends on the initial $x_0$ and slowly improves towards a fast rate as the system converges and the regime switches. We present a complete analysis and set of conditions on $f$ for two of the most useful scenarios. In Proposition \ref{lem:assvarphiaA} we consider the case that $f$ grows like $\varphi_b^B(\norm{x-\xmin})$ where $b,B > 1$ are exactly known. In this case convergence proceeds at a fast constant linear rate when matched with $\K(p) = \varphi_{a}^{A}(\norm{p}_*)$ where $a = b/(b-1)$ and $A = B/(B-1)$. In Proposition \ref{lem:assvarphi21} we consider the case that $f$ grows like $\varphi_2^B(\norm{x-\xmin})$ where $B \geq 2$ is unknown. Here, the convergence is linear with a non-constant rate when matched with the relativistic kinetic energy $\K(p) = \varphi_2^1(\norm{p}_*)$. The case covered by relativistic kinetic $\grad \K$ is particularly valuable, as it covers a large class of globally non-smooth, but strongly convex functions. Table \ref{table:conditionsummary} summarizes this, and throughout the remaining subsections we flesh out the details of these claims.

For these kinetic energies to be suitable in our analysis, they must at minimum satisfy assumptions \ref{ass:cont:kconvex}, \ref{ass:semiA:gKpsem}, \ref{ass:semiB:Ktwicediff}, \ref{ass:semiB:gHpsem}, and \ref{ass:semiB:weird}. Assumptions \ref{ass:semiA:gKpsem} and \ref{ass:semiB:gHpsem} are clearly satisfied by $\K(p) = |p|^a/a$ for $p \in \R$ with constants $C_{\K} = a$ and $D_{\K} = a(a-1)$. In the remainder of this subsection, we provide conditions on the norms and $a, A$ under which assumptions like these hold for $\varphi_a^A$ with multiple power behavior in any finite dimension.

In general, the problematic terms of $\grad \K(p)$ and $\hess \K(p)$ that arise in high dimensions involve the gradient and Hessian of the norm. The gradient of norm can be dealt with cleanly, but our analysis requires additional control on the Hessian of the norm. To control terms involving $\hess \norm{p}_*$ we define a generalization of the maximum eigenvalue induced by the norm $\norm{\c}$. Let $\maxeigenf^{\norm{\c}} : \R^{d\times d} \to \R$ be the function defined by
\begin{equation}
\label{eq:matrixnorm} \maxeigen{M} = \sup \{ \inner{v}{Mv} : v \in \R^d, \norm{v} = 1\}.
\end{equation}
For symmetric $M \in \R^{d\times d}$ and Euclidean $\norm{\c}$ this is exactly the maximum eigenvalue of $M$. Now we are able to state our lemma analyzing power kinetic energies.

\begin{table}[t]
\renewcommand{\arraystretch}{1.25}
\begin{tabular}{@{}llllrr@{}}
\toprule
	&\multicolumn{3}{c}{$f(x)$ grows like $\varphi_b^B(\norm{x})$}& \multicolumn{2}{c}{appropriate $\K(p)=\varphi_a^A(\norm{p}_*)$}\\
	\cmidrule(rl){2-4} \cmidrule(rl){5-6}
	method & powers known? & body power $b$ & tail power $B$ & body power $a$  & tail power $A$\\
	\midrule
	implicit & known& $b > 1$ & $B > 1$  & $a = b/(b-1)$  & $A=B/(B-1)$\\
	& unknown & $b = 2$ & $B \geq 2$  & $a = 2$  & $A=1$\\
	\midrule
	1st explicit & known & $b \geq 2$ & $B \geq 2$  & $a = b/(b-1)$  & $A=B/(B-1)$\\
	& unknown & $b = 2$ & $B \geq 2$  & $a = 2$  & $A=1$\\
	\midrule
	2nd explicit & known & $1 < b \leq 2$ & $1 < B \leq 2$  & $a = b/(b-1)$  & $A=B/(B-1)$\\
\bottomrule
\end{tabular}
\caption{A summary of the conditions on $f$ and power kinetic $\K$ considered in this section that satisfy the assumptions of Section \ref{sec:discrete}. Here ``grows like'' is an imprecise term meaning that $f$s growth can be bounded in an appropriate way by $\varphi_b^B(\norm{x})$ ($\varphi_b^B$ is defined in \eqref{eq:kindef}). The full precise assumptions on $f$ are laid out in Propositions \ref{lem:assvarphiaA} and \ref{lem:assvarphi21}. In particular, $b=B=2$ corresponds to assumptions similar in spirit to strong convexity and smoothness. Other combinations of $b,B$ and $a,A$ are possible.}
\label{table:conditionsummary}
\end{table}

\begin{restatable}[Verifying assumptions on $\K$]{lem}{kinrobust}
\label{kin:robust}
Given a norm $\norm{p}_*$ on $p \in \R^d$, $a, A \in [1, \infty)$, and $\varphi_a^A$ in \eqref{eq:kindef}. Define the constant,
\begin{equation}
C_{a,A} = \l(1 - \l(\tfrac{a-1}{A-1}\r)^{\tfrac{a-1}{A-a}} + \l(\tfrac{a-1}{A-1}\r)^{\tfrac{A-1}{A-a}}\r)^{\tfrac{B-b}{b}}.
\end{equation}
$\K(p) = \varphi_a^A(\norm{p}_*)$ satisfies the following.
\begin{enumerate}
\item Convexity. If $a > 1$ or $A > 1$, then $\K$ is strictly convex with a unique minimum at $0 \in \R^d$.
\item Conjugate. For all $x \in \R^d$, $\K^*(x) = (\varphi_a^A)^*(\norm{x})$.
\item Gradient. If $\norm{p}_*$ is differentiable at $p \in \R^d \setminus \{0\}$ and $a > 1$, then $\K$ is differentiable for all $p \in \R^d$, and for all $p \in \R^d$,
\begin{align}
\label{kin:robust:innerbound} \inner{\grad \K(p)}{p} &\leq \MaA\K(p),\\
\label{kin:robust:slowgrowth1} (\varphi_a^A)^*(\norm{\grad \K(p)}) &\leq (\MaA-1) \K(p).
\intertext{Additionally, if $a,A > 1$, define $B = A/(A-1)$, $b = a/(a-1)$, and then}
\label{kin:robust:slowgrowth2} \varphi_b^B(\norm{\grad \K(p)}) &\leq C_{a,A}  (\MaA-1)  \K(p).
\intertext{Additionally, if $a,A \geq 2$, then for all $p,q \in \R^d$,}
\label{kin:robust:unifineq} \K(p) &\leq \inner{\grad \K(q)}{q} + \inner{\grad \K(p)  - \grad \K(q)}{p-q}.
\end{align}
\item Hessian. If $\norm{p}_*$ is twice continuously differentiable at $p \in \R^d \setminus \{0\}$, then $\K$ is twice continuously differentiable for all $p \in \R^d \setminus \{0\}$, and for all $p \in \R^d \setminus \{0\}$,
\begin{align}
\label{kin:robust:hessianpp} \inner{p}{\hess \K(p) p} \leq \MaA (\MaA - 1) \K(p).
\intertext{Additionally, if $a, A \geq 2$ and there exists $N \in [0, \infty)$ such that $\norm{p}_*\maxeigenconj{\hess \norm{p}_*} \leq N $ for $p \in \R^d \setminus \{0\}$, then for all $p \in \R^d \setminus \{0\}$}
\label{kin:robust:hessiannorm}  (\varphi_{a/2}^{A/2})^*\l(\frac{\maxeigenconj{\hess \K(p)}}{\MaA - 1 + N}\r) \leq (\MaA - 2) \K(p).
\end{align}
\end{enumerate}
\end{restatable}

\begin{remark}
\eqref{kin:robust:innerbound}, \eqref{kin:robust:unifineq}, and \eqref{kin:robust:hessianpp} together directly confirm that these $\K$ satisfy \ref{ass:semiA:gKpsem}, \ref{ass:semiB:gHpsem}, and \ref{ass:semiB:weird} with constants $C_{\K} = \max\{a, A\}$,  $D_{\K} = \max\{a,A\}( \max\{a, A\} - 1)$, $E_{\K} = \max\{a, A\} - 1$, and $F_{\K}  = 1$. The other results \eqref{kin:robust:slowgrowth1}, \eqref{kin:robust:slowgrowth2}, and \eqref{kin:robust:hessiannorm} will be used in subsequent lemmas along with assumptions on $f$ to satisfy the remaining assumptions of discretization.
\end{remark}

The assumption that $\norm{p}_*\maxeigenconj{\hess \norm{p}_*} \leq N$ in Lemma \ref{kin:robust} is satisfied by $b$-norms for $b \in [2, \infty)$, as the following lemma confirms. It implies that if $\|p \|_{*}=\|p\|_{b}$ for $b\ge 2$, we can take $N=b-1$ in \eqref{kin:robust:hessiannorm}.

\begin{restatable}[Bounds on $\maxeigenconj{\hess \norm{p}_*}$ for $b$-norms]{lem}{normhess}
\label{lem:normhess}
Given  $b \in [2, \infty)$, let $\norm{x}_b = \l(\sum_{n=1}^d |x^{(n)}|^b\r)^{1/b}$ for $x \in \R^d$. Then for $x \in \R^d \setminus \{0\}$,
\begin{equation*}
\norm{x}_b \maxeigennorm{\hess \norm{x}_b}{\norm{\c}_b} \leq (b-1).
\end{equation*}
\end{restatable}

The remaining assumptions \ref{ass:imp:innerproductfk}, \ref{ass:semiA:gKHfsemi}, and \ref{ass:semiB:gfHKsemi} involve inner products between derivatives of $f$ and $\K$. To control these terms we will use the Fenchel-Young inequality. To this end, the conjugates of $\varphi_a^A$ will be a crucial component of our analysis.

\begin{restatable}[Convex conjugates of $\varphi_a^A$]{lem}{robustonedimconj}
\label{kin:robustonedimconj}
Given $a, A \in (1, \infty)$ and $\varphi_a^A$ in \eqref{eq:kindef}. Define $B = A/(A-1)$, $b = a/(a-1)$. The following hold.
\begin{enumerate}
\item Near Conjugate. $\varphi_b^B$ upper bounds the conjugate $(\varphi_a^A)^*$ for all $t \in [0, \infty)$,
\begin{equation}
\label{kin:robustonedim:conjugatebound} (\varphi_a^A)^*(t) \leq \varphi_b^B(t).
\end{equation}
\item Conjugate. For all $t \in [0, \infty)$,
\begin{align}
\label{kin:robustonedim:conjugatebound1} (\varphi_a^a)^*(t) &= \varphi_b^b(t).\\
\label{kin:robustonedim:conjugatebound3} (\varphi_1^A)^*(t) &= \begin{cases}
0 &\quad t \in [0,1]\\
\tfrac{1}{B}t^{B} - t + \tfrac{1}{A} &\quad t \in (1, \infty)
\end{cases}.\\
\label{kin:robustonedim:conjugatebound4} (\varphi_a^1)^*(t) &= \begin{cases}
1 - (1 - t^b)^{\tfrac{1}{b}} & \quad t \in [0,1]\\
\infty & \quad t \in (1,\infty)
\end{cases}.\\
\label{kin:robustonedim:conjugatebound5} (\varphi_1^1)^*(t) &= \begin{cases}
0 &\quad t \in [0,1]\\
\infty & \quad t \in (1,\infty)
\end{cases}.
\end{align}
\end{enumerate}
\end{restatable}

\subsection{Matching power kinetic $\grad \K$ with assumptions on $f$}
\label{sec:singlepower}

In this subsection and the next we study assumptions on $f$ that imply the suitability of $\K(p) = \varphi_a^A(\norm{p}_*)$ with the discretizations of Section \ref{sec:discrete}. The preceding subsection is an analysis that verifies that such $\K$ satisfy the $\K$-specific assumptions \ref{ass:cont}, \ref{ass:semiA}, and \ref{ass:semiB}. We now consider the remaining assumptions of \ref{ass:cont}, \ref{ass:imp}, \ref{ass:semiA}, and \ref{ass:semiB}, which require an appropriate match between $f$ and $\K$. This includes the derivation of $\alpha$ and control of terms of the form $\inner{\grad f(x)}{\grad \K(p)}$ and $\inner{\grad \K(p)}{\hess f(x)\grad \K(p)}$ by the total energy $\Ha(x, p)$. Here we consider the case that $f$ exhibits power behavior with known, but possibly distinct, powers in the body and the tails.

\begin{figure}[t!]
    \centering
    \includegraphics[scale=1]{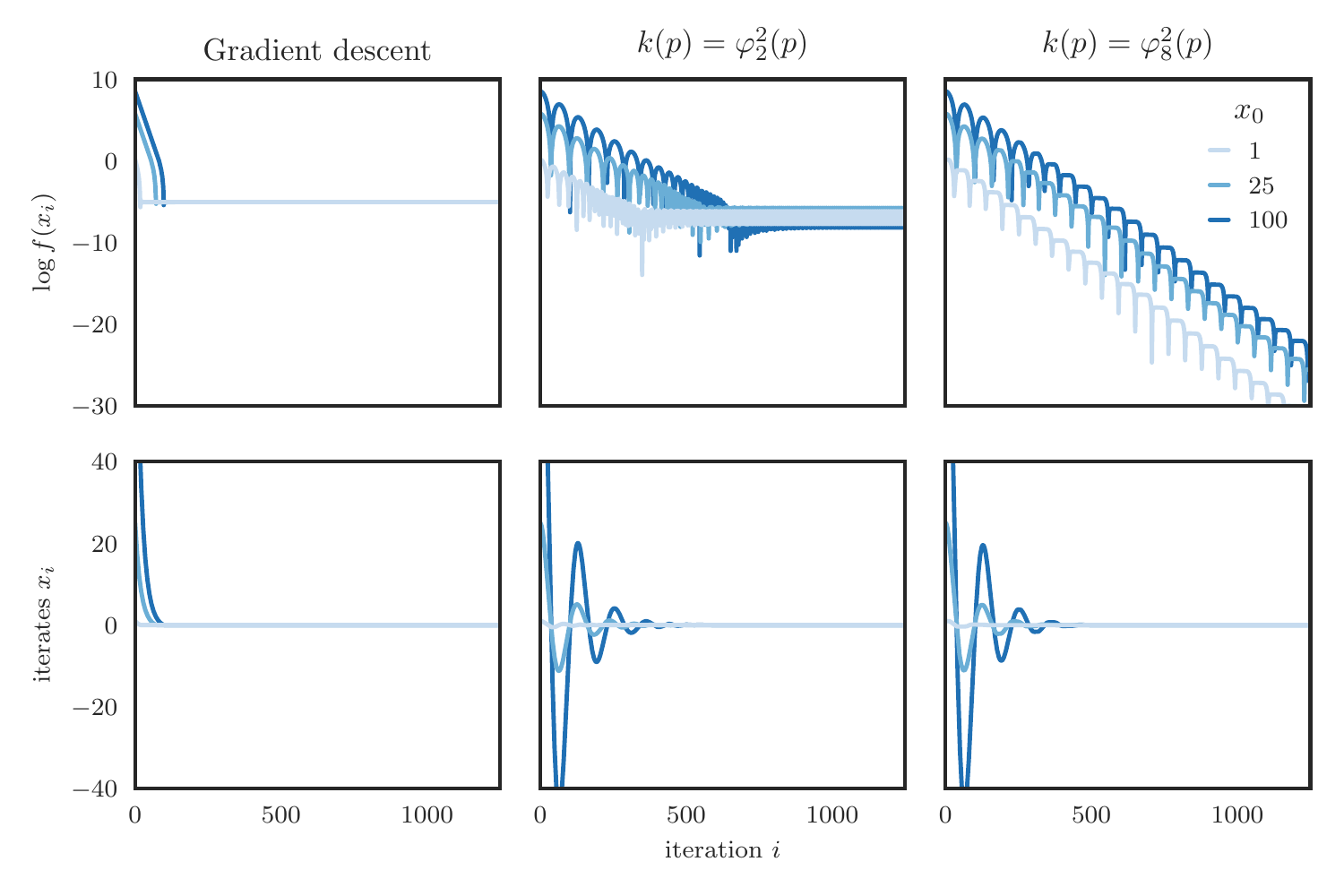}
    \caption{Optimizing $f(x) = \varphi_{8/7}^2(x)$ with three different methods with fixed step size: gradient descent, classical momentum, and our second explicit method. Because the second derivative of $f$ is infinite at its minimum, only second explicit method with $\K(p) = \varphi_8^2(p)$ is able to converge with a fixed step size.}
    \label{fig:sharp}
\end{figure}

To see a complete example of this type of analysis, take the case $f(x) = |x|^b/b$ and $\K(p) = |p|^a/a$ with $x,p \in \R$,  $b > 2$, $a < 2$, and $1/a + 1/b = 1$. For $\alpha$ the strategy will be to find a lower bound on $f$ that is symmetric about $f$'s minimum. The conjugate of the centered lower bound can be used to construct an upper bound on $\fc^*$ with which the gap between $\K$ and $\fc^*$ can be studied. In this case it is simple, as we have $\fc^*(p) = \K(p)$ and $\alpha = 1$. The strategy for terms of the form $\inner{\grad f(x)}{\grad \K(p)}$ and $\inner{\grad \K(p)}{\hess f(x)\grad \K(p)}$ will be a careful application of the Fenchel-Young inequality. Using $a-1 = a/b$, the conjugacy of $b$ and $b/(b-1)$, and the Fenchel-Young inequality,
\begin{align*}
|\inner{\grad f(x)}{\grad \K(p)}| = |x|^{b-1} |p|^{a/b} \leq \tfrac{b-1}{b}(|x|^{b-1})^{\frac{b}{b-1}} +  \tfrac{1}{b}(|p|^{a/b})^{b} &= (b-1)f(x) + (a-1)\K(p)\\
&\leq (\max\{a, b\} - 1) \Ha(x,p).
\end{align*}
Finally, using the conjugacy of $b/2$ and $b/(b-2)$ and again the Fenchel-Young inequality,
\begin{align*}
\inner{\grad \K(p)}{\hess f(x) \grad \K(p)} = (b-1) |x|^{b-2} |p|^{2a/b} &\leq \tfrac{(b-1)(b-2)}{b} (|x|^{b-2})^{\frac{b}{b-2}} + \tfrac{(b-1)2}{b}(|p|^{2a/b})^{\frac{b}{2}} \\
&= (b-1)(b-2) f(x)  + 2(b-1)(a-1) \K(p) \\
&\leq (b-1)\max\{b-2, 2(a-1)\} \Ha(x,p).
\end{align*}
Along with Lemma \ref{kin:robust}, this covers Assumptions \ref{ass:cont}, \ref{ass:imp}, and \ref{ass:semiA}. Thus, we can justify the use of the first explicit method for this $f, \K$. All of the analyses of this section essentially follow this outline.

\begin{remark} These strategies apply naturally when $f$ is twice differentiable and smooth. In this case, we have $\tfrac{1}{2}\norm{\grad f(x)}_2^2 \leq L (f(x) - f(\xmin))$ and $\maxeigennorm{\hess f(x)}{\norm{\c}_2} \leq L$. Thus, using $\K(p) = \tfrac{1}{2} \norm{p}_2^2$ is appropriate and $\inner{\grad f(x)}{\grad \K(p)} \leq \max\{L,1\} \Ha(x,p)$ and $\inner{\grad \K(p)}{\hess f(x)\grad \K(p)} \leq 2 L \K(p)$.
\end{remark}

We are now ready to consider the case of $f$ growing like $\varphi_b^B(\norm{x-\xmin})$ matched with $\K(p) = \varphi_a^A(\norm{p}_*)$ for $1/a + 1/b = 1/A + 1/B = 1$. Assumptions \ref{ass:knownpower}, below, will be used in different combinations to confirm that the assumptions of the different discretizations are satisfied. Assumptions \ref{knownpower:fconvex}, \ref{knownpower:norm}, \ref{knownpower:fpower}, and \ref{knownpower:impcondition} are required for all methods. The explicit methods each require an additional assumption: \ref{knownpower:semiAcondition} for the first explicit method and \ref{knownpower:semiBcondition} for the second. Thus, for $f:\R^d\to \R$ and $\K(p)=\varphi_2^1(\norm{p}_*)$, Proposition \ref{lem:assvarphiaA} can be summarised as
\begin{align*}
\mathref{knownpower:fconvex}\wedge\mathref{knownpower:norm}\wedge\mathref{knownpower:fpower}\wedge\mathref{knownpower:impcondition} &\Rightarrow \mathref{ass:cont}\wedge\mathref{ass:imp},\\
\mathref{knownpower:fconvex}\wedge\mathref{knownpower:norm}\wedge\mathref{knownpower:fpower}\wedge\mathref{knownpower:impcondition}\wedge\mathref{knownpower:semiAcondition}&\Rightarrow
\mathref{ass:cont}\wedge\mathref{ass:imp}\wedge\mathref{ass:semiA},\\
\mathref{knownpower:fconvex}\wedge\mathref{knownpower:norm}\wedge\mathref{knownpower:fpower}\wedge\mathref{knownpower:impcondition}\wedge\mathref{knownpower:semiBcondition}&\Rightarrow
\mathref{ass:cont}\wedge\mathref{ass:imp}\wedge\mathref{ass:semiB}.
\end{align*}
Note, that Lemma \ref{kin:robust} implies that the power kinetic energies are themselves examples of functions satisfying Assumptions \ref{ass:knownpower}. Figure \ref{fig:sharp} illustrates a consequence of this proposition; $f(x) = \varphi_{8/7}^2(x)$ for $x \in \R$ is a difficult function to optimize with a first order method using a fixed step size; the second derivative grows without bound as $x \to 0$. As shown, Hamiltonian descent with the matched $\K(p) = \varphi_8^2(p)$ converges, while gradient descent and classical momentum do not.


\begin{assumptions}{knownpower}
	\item \label{knownpower:fconvex} $f : \R^d \to \R$ differentiable and convex with unique minimum $\xmin$.
	\item \label{knownpower:norm} $\norm{p}_*$ is differentiable at $p \in \R^d \setminus \{0\}$ with dual norm $\norm{x} = \sup \{ \inner{x}{p} : \norm{p}_* = 1\}$.
	\item \label{knownpower:fpower} $B= A/(A-1)$, and $b = a/(a-1)$.
	\item \label{knownpower:impcondition} There exist $\mu,L \in (0, \infty)$ such that for all $x\in\R^d$
            \begin{equation}
            \label{knownpower:alphaimpcondition}
                \begin{aligned}
                    f(x) - f(\xmin) &\geq \mu \varphi_b^B(\norm{x - \xmin}) \\
                    \varphi_a^A(\norm{\grad f(x)}_*) &\leq L(f(x) - f(\xmin)).
                \end{aligned}
            \end{equation}
	\item \label{knownpower:semiAcondition} $b \geq 2$ and $B \geq 2$. $\f : \R^d \to \R$ is twice continuously differentiable for all $x \in \R^d \setminus \{\xmin\}$ and there exists $\Lf , \Df \in (0, \infty)$ such that for all $x\in\R^d \setminus \{\xmin\}$
            \begin{equation}
            	\label{knownpower:ass:semiA:condition} \l(\varphi_{b/2}^{B/2}\r)^*\l(\frac{\maxeigen{\hess f(x)}}{\Lf}\r) \leq \Df (f(x) - f(\xmin)).
            \end{equation}
	\item \label{knownpower:semiBcondition} $b \leq 2$ and $B \leq 2$. $\norm{p}_*$ is twice continuously differentiable at $p \in \R^d \setminus \{0\}$, and there exists $N \in (0, \infty)$ such that $\maxeigenconj{\hess \norm{p}_*} \leq N \norm{p}_*^{-1}$ for all $p \in \R^d \setminus \{0\}$.
\end{assumptions}

\begin{remark}
Assumption \ref{knownpower:impcondition} can be read as the requirement that $f$ is bounded above and below by $\varphi_b^B$, and in the $b=B=2$ case it is a necessary condition of strong convexity and smoothness.
\end{remark}

\begin{remark}
Assumption \ref{knownpower:semiAcondition} generalizes a sufficient condition for smoothness. Consider for simplicity the Euclidean norm case $\norm{\c} = \norm{\c}_2$ and let $\lambda_{\max}(M)$ be the maximum eigenvalue of $M \in \R^{d \times d}$. If $b=B=2$, then $(\varphi_{b/2}^{B/2})^*$ is finite only on $[0, 1]$ where it is zero. Moreover, \ref{knownpower:semiAcondition} simplifies to there existing $\Lf \in (0, \infty)$ such that $\lambda_{\max}(\hess f(x)) \leq \Lf$ everywhere, the standard smoothness condition. When $b > 2, B=2$, $(\varphi_{b/2}^{B/2})^*$ is finite on $[0,1]$ where it behaves like a power $b/(b-2)$ function for small arguments. Thus, \ref{knownpower:semiAcondition} can be satisfied in the Euclidean norm case by a function whose maximum eigenvalue is shrinking like $\norm{x - \xmin}_2^{b-2}$ as $x \to \xmin$; the balance of where the behavior switches can be controlled by $\Lf$. When $b = 2, B > 2$, the role is switched and \ref{knownpower:semiAcondition} can be satisfied by a function whose maximum eigenvalue is bounded near the minimum and grows like $\norm{x - \xmin}_2^{B-2}$ as $\norm{x - \xmin}_2 \to \infty$. When $b, B > 2$, this can be satisfied by a function whose maximum eigenvalue shrinks like $\norm{x - \xmin}_2^{b-2}$ in the body and grows like $\norm{x - \xmin}_2^{B-2}$ in the tail.
\end{remark}

\begin{restatable}[Verifying assumptions for $f$ with known power behavior and appropriate $\K$]{prop}{assvarphiaA}
\label{lem:assvarphiaA}
Given a norm $\norm{\c}_*$ satisfying \ref{knownpower:norm} and $a, A \in (1,\infty)$, take
\begin{equation*}
\K(p) = \varphi_a^A(\norm{p}_*).
\end{equation*}
with $\varphi_a^A$ defined in \eqref{eq:kindef}. The following cases hold with this choice of $\K$ on $f : \R^d \to \R$ convex.
\begin{enumerate}
    \item For the implicit method \eqref{eq:implicit}, assumptions \ref{ass:cont}, \ref{ass:imp} hold with constants
            \begin{equation}
            \label{knownpower:ass:imp:constants} \alpha=\min\{\mu^{a-1}, \mu^{A-1}, 1\} \qquad \Ca = \gamma  \qquad \Cfk = \max\{a-1, A-1, L\},
            \end{equation}
            if $f, a, A, \mu, L, \norm{\c}_*$ satisfy assumptions \ref{knownpower:fconvex}, \ref{knownpower:norm}, \ref{knownpower:fpower}, \ref{knownpower:impcondition}.
    \item For the first explicit method \eqref{eq:semiA}, assumptions \ref{ass:cont},  \ref{ass:imp}, and \ref{ass:semiA} hold with constants \eqref{knownpower:ass:imp:constants} and
            \begin{equation}
            \label{knownpower:ass:semiA:constants} \Ck = \MaA \qquad \Dfk = \Lf\alpha^{-1}\max\l\{\Df, \, 2C_{a,A} (\max\{a,A\} -1) \r\},
            \end{equation}
            if $f, a, A, \mu, L, \Lf, \Df, \norm{\c}_*$ satisfy assumptions \ref{knownpower:fconvex}, \ref{knownpower:norm}, \ref{knownpower:fpower}, \ref{knownpower:impcondition}, and \ref{knownpower:semiAcondition}.
    \item For the second explicit method \eqref{eq:semiB}, assumptions \ref{ass:cont}, \ref{ass:imp}, and \ref{ass:semiB} hold with constants  \eqref{knownpower:ass:imp:constants} and
            \begin{equation}
            \label{knownpower:ass:semiB:constants}
            \begin{aligned}
            \Ck &= \MaA \qquad \Dk = \MaA (\MaA - 1)\\
            E_{\K} &= \MaA - 1 \qquad F_{\K} = 1\\
            \Dfk &= \alpha^{-1} (\max\{a,A\} - 1 + N)\max\l\{2L, a- 2, A-2\r\},
            \end{aligned}
            \end{equation}
            if  $f, a, A, \mu, L, N, \norm{\c}_*$ satisfy assumptions \ref{knownpower:fconvex}, \ref{knownpower:norm}, \ref{knownpower:fpower}, \ref{knownpower:impcondition}, and \ref{knownpower:semiBcondition}.
\end{enumerate}
\end{restatable}

\begin{figure}[t!]
    \centering
    \includegraphics[scale=1]{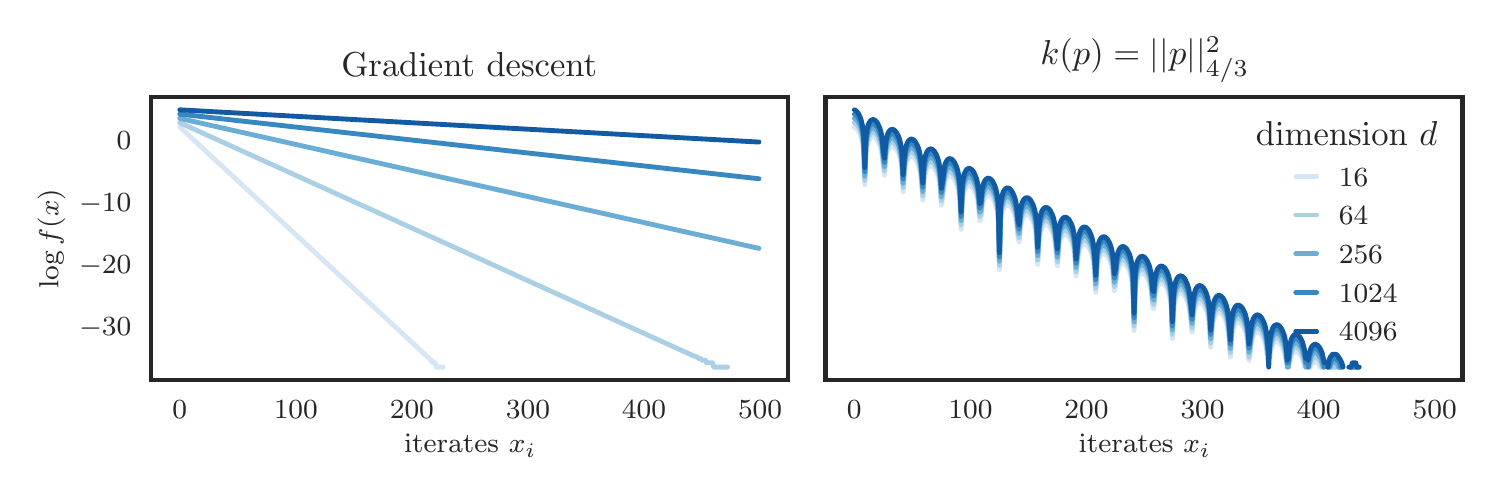}
    \caption{Dimension dependence on $f(x) = \norm{x}_4^2/2$ initialized at $x_0 = (2, 2, \ldots, 2) \in \R^{\mathrm{d}}$. Left: Gradient descent with fixed step size equal to the inverse of the smoothness coefficient, $L = 3$. Right: Hamiltonian descent with $\K(p) = \norm{p}_{4/3}^2/2$ and a fixed step size (same for all dimensions). Gradient descent convergences linearly with rate $\lambda = 3/(3  - 1/\sqrt{d})$ while Hamiltonian descent converges with dimension-free linear rates.}
    \label{fig:dimdepend}
\end{figure}

We highlight an interesting consequence of Proposition \ref{lem:assvarphiaA} for high dimensional problems. For many first-order methods using standard gradients on smooth $f$, linear convergence can be guaranteed by the Polyak-{\L}ojasiewicz (PL) inequality, $\norm{\grad f(x)}_2^2/2 \geq \mu (f(x) - f(\xmin))$, see \eg{} \cite{karimi2016linear}. The rates of convergence generally depend on $\mu$ and the smoothness constant $L$. Unfortunately, for some functions the constant $L$ or the constant $\mu$ may depend on the dimensionality of the space. Although smoothness and the PL inequality can be defined with respect to non-Euclidean norms, this does not generally overcome the issue of dimension dependence if standard gradients are used, see \cite{juditsky2014deterministic, nesterov2005smooth} for a discussion and methods using non-standard gradients. The situation is distinct for Hamiltonian descent. If $f$ is smooth with respect to a non-Euclidean norm $\norm{\c}$, then, by taking $\K(p) = \norm{p}_*^2/2$, Proposition \ref{lem:assvarphiaA} may guarantee, under appropriate assumptions, dimension independent rates when using standard gradients (dependence on the dimensionality is mediated by the constant $N$). For example, consider $f(x) = \norm{x}_4^2/2 = (\sum_{n=1}^d (x^{(n)})^4)^{1/2}/2$ defined for $d$-dimensional vectors $x \in \R^d$. It is possible to show that $f$ is smooth with respect to $\norm{\c}_2$ with constant $L = 3$ (our Lemma \ref{lem:normhess} together with an analysis analogous to Lemma 14 in Appendix A of \cite{shalev2007online} and the fact that $\norm{x}_4 \leq \norm{x}_2$) and that $f$ satisfies the PL inequality with $\mu = 1/ \sqrt{d}$ (the fact that $ \norm{x}_{4/3}^2 \leq  \sqrt{d} \norm{x}_2^2$ and Lemma \ref{kin:robust}). The iterates of a gradient descent algorithm with fixed step size $1/L$ on this $f$ will therefore satisfy the following,
\begin{align*}
	f(x_{i+1}) - f(x_i) \leq - \frac{1}{6} \norm{\grad f(x_i)}_2^2 \leq - \frac{1}{3 \sqrt{d}} f(x_i).
\end{align*}
From which we conclude that gradient descent converges linearly with rate $\lambda = 3/(3  - 1/\sqrt{d})$, worsening as $d \to \infty$. Figure \ref{fig:dimdepend} illustrates this, along with a comparison to Hamiltonian descent with $\K(p) = \norm{p}_{4/3}^2$, which enjoys dimension independence as $N = 3$.

\subsection{Matching relativistic kinetic $\grad \K$ with assumptions on $f$}
\label{sec:relativistic}

The strongest assumption of Proposition \ref{lem:assvarphiaA} is that the power behaviour of $f$ captured in the constant $b$ is exactly known. This is generally not the case and usually hard to determine. The only possible exception is $b = 2$, which can be guaranteed by lower bounding the eigenvalues of the Hessian. In our second analysis, we consider a kinetic energy generically suitable for such strongly convex functions that may not be smooth. Crucially, less information needs to be known about $f$ for this kinetic energy to be applicable. The cost imposed by this lack of knowledge is a non-constant rate of linear convergence, which begins slowly and improves towards a faster rate as $(x_i, p_i) \to (\xmin, 0)$.

In particular, we consider the use of the relativistic kinetic energy,
\begin{equation*}
	\K(p) = \varphi_2^1(\norm{p}_*) = \sqrt{\norm{p}_*^2 + 1} - 1,
\end{equation*}
which was studied by Lu \etal{} \cite{lu2016relativistic} and Livingstone \etal{} \cite{livingstone2017kinetic} in the context of Hamiltonian Monte Carlo. Consider for the moment the Euclidean norm case. In this case, we have
\begin{equation*}
	\grad \K(p) = \frac{p}{\sqrt{\norm{p}_2^2 + 1}}.
\end{equation*}
As noted by Lu \etal{} \cite{lu2016relativistic}, this kinetic map resembles the iterate updates of popular adaptive gradient methods \cite{duchi2011adaptive, zeiler2012adadelta, hinton2014rmsprop, kingma2014adam}. Because the iterate updates $x_{i+1} - x_i$ of Hamiltonian descent are proportional to $\grad \K(p)$ from some $p \in \R^d$, this suggests that the relativistic map may have favorable properties. Notice that $\norm{\grad \K(p)}_2^2 = \norm{p}_2^2/(\norm{p}_2^2 + 1) < 1$, implying that $\norm{x_{i+1} - x_i}_2 < \epsilon$ uniformly and regardless of the magnitudes of $\grad f$. The fact that the magnitude of iterate updates is uniformly bounded makes the relativistic map suitable for functions with very fast growing tails, even if the rate of growth is not exactly known.

More precisely, we consider the case of $f$ growing like $\varphi_2^B(\norm{x-\xmin})$ matched with $\K(p) = \varphi_2^1(\norm{p}_*)$ for $B \geq 2$. Assumptions \ref{ass:fastgrowth}, below, will be used in different combinations to confirm that the assumptions of the different discretizations are satisfied. Assumptions \ref{fastgrowth:fconvex}, \ref{fastgrowth:norm}, \ref{fastgrowth:fpower}, and \ref{fastgrowth:impcondition} are required for all methods. The first explicit method requires additional assumptions \ref{fastgrowth:semiAconditionBg2} for $B > 2$ and \ref{fastgrowth:semiAconditionB2} for $B = 2$. We do not include an analysis for the second explicit method. Thus, for $f:\R^d\to \R$ and $\K(p)=\varphi_2^1(\norm{p}_*)$, Proposition \ref{lem:assvarphi21} can be summarised as
\begin{align*}
\mathref{fastgrowth:fconvex}\wedge\mathref{fastgrowth:norm}\wedge\mathref{fastgrowth:fpower}\wedge\mathref{fastgrowth:impcondition}&\Rightarrow \mathref{ass:cont}\wedge\mathref{ass:imp}\\
\mathref{fastgrowth:fconvex}\wedge\mathref{fastgrowth:norm}\wedge\mathref{fastgrowth:fpower}\wedge\mathref{fastgrowth:impcondition}\wedge\mathref{fastgrowth:semiAconditionBg2}&\Rightarrow
\mathref{ass:cont}\wedge\mathref{ass:imp}\wedge\mathref{ass:semiA}\\
\mathref{fastgrowth:fconvex}\wedge\mathref{fastgrowth:norm}\wedge\mathref{fastgrowth:fpower}\wedge\mathref{fastgrowth:impcondition}\wedge\mathref{fastgrowth:semiAconditionB2}&\Rightarrow
\mathref{ass:cont}\wedge\mathref{ass:imp}\wedge\mathref{ass:semiA}
\end{align*}
Figure \ref{fig:relativistic} illustrates a consequence of this proposition; $f(x) = \varphi_{8}^2(x)$ for $x \in \R$ is a difficult function to optimize with a first order method using a fixed step size; the second derivative grows without bound as $|x| \to \infty$. Thus if the initial point is taken to be very large, gradient descent must take a very conservative choice of step size. As shown, Hamiltonian descent with the matched $\K(p) = \varphi_{8/7}^2(p)$ converges quickly and uniformly, while gradient descent with a fixed step size suffers a very slow rate for $|x_0| \gg 0$. In the middle panel, the relativistic choice converges slowly at first, but speeds up as convergence proceeds, making it a suitable agnostic choice in cases such as this.

\begin{figure}[t!]
    \centering
    \includegraphics[scale=1]{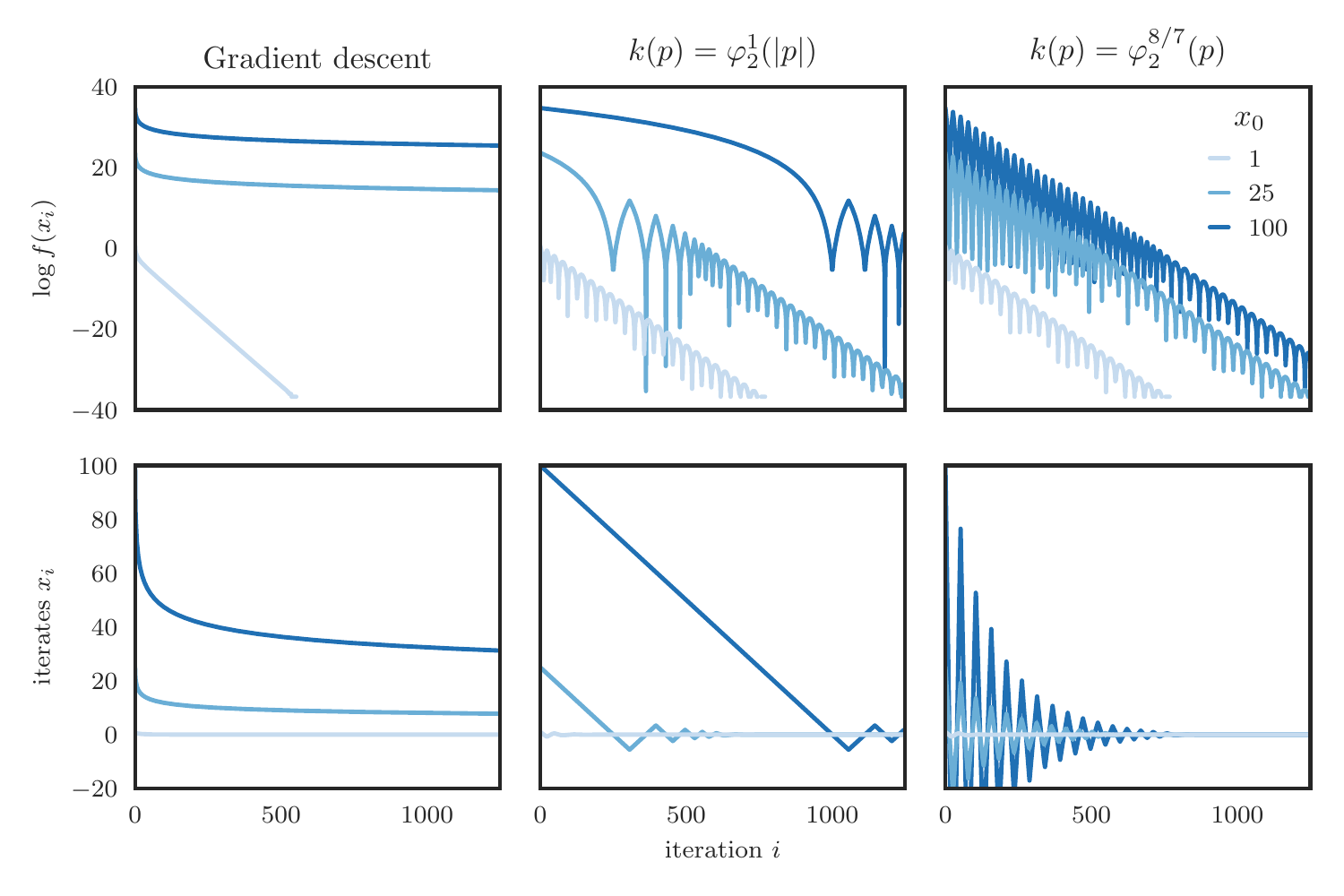}
    \caption{$f(x) = \varphi_{2}^8(x)$ with three different methods: gradient descent with the optimal fixed step size, Hamiltonian descent with relativistic kinetic energy, and Hamiltonian descent with the near dual kinetic energy.}
    \label{fig:relativistic}
\end{figure}


\begin{assumptions}{fastgrowth}
	\item \label{fastgrowth:fconvex} $f : \R^d \to \R$ differentiable and convex with unique minimum $\xmin$.
	\item \label{fastgrowth:norm} $\norm{p}_*$ is differentiable at $p \in \R^d \setminus \{0\}$ with dual norm $\norm{x} = \sup \{ \inner{x}{p} : \norm{p}_* = 1\}$.
	\item \label{fastgrowth:fpower} $B \in [2,\infty)$ and $A= B/(B-1)$.
	\item \label{fastgrowth:impcondition} There exist $\mu,L \in (0, \infty)$ such that for all $x\in\R^d$
                \begin{equation}
                \label{fastgrowth:ass:imp:condition}
                \begin{aligned}
                f(x) - f(\xmin)& \ge \mu\varphi_2^B(\norm{x - \xmin}) \\
                \varphi_2^1(\norm{\grad f(x)}_*) &\leq L(f(x) - f(\xmin)).
                \end{aligned}
                \end{equation}
	\item \label{fastgrowth:semiAconditionBg2} $B > 2$. Define
                \begin{equation}
                \label{eq:kininv} \psi(t) = \begin{cases}
                0 & 0 \leq t < 1\\
                t - 3 t^{\tfrac{1}{3}} +2 &  1  \leq t \\
                \end{cases}.
                \end{equation}
                $f : \R^d \to \R$ is twice continuously differentiable for all $x \in \R^d \setminus \{\xmin\}$ and there exists $\Lf \in (0, \infty)$ such that for all $x\in\R^d\setminus \{\xmin\}$
                \begin{equation}
                 \psi\l(\tfrac{B-1}{B-2} \varphi_{1}^{\tfrac{B-1}{B-2}}\l(\frac{ \maxeigen{\hess f(x)}}{\Lf}\r)\r) \leq 3 (f(x) - f(\xmin)).
                \end{equation}
	\item \label{fastgrowth:semiAconditionB2} $B = 2$. $f : \R^d \to \R$ is twice continuously differentiable for all $x \in \R^d \setminus \{\xmin\}$ and there exists $\Lf \in (0, \infty)$ such that for all $x\in\R^d\setminus \{\xmin\}$
                \begin{equation}
                \maxeigen{\hess f(x)} \leq \Lf.
                \end{equation}
\end{assumptions}
\begin{remark}
Assumptions \ref{ass:fastgrowth} hold in general for convex functions $f$ that grow quadratically at their minimum, and as power $B$ in the tails, for some $B\ge 2$.
\end{remark}

We include the proof of this proposition below, as it highlights every aspect of our analysis, including non-constant $\alpha$.

\begin{prop}[Verifying assumptions for $f$ with unknown power behavior and relativistic $\K$]\label{lem:assvarphi21}
Given a norm $\norm{\c}_*$ satisfying \ref{fastgrowth:norm}, take
\begin{equation*}
\K(p) = \varphi_2^1(\norm{p}_*)
\end{equation*}
with $\varphi_a^A$ in \eqref{eq:kindef}. The following cases hold with this choice of kinetic energy $\K$ on $f : \R^d \to \R$ convex.
\begin{enumerate}
    \item For the implicit method \eqref{eq:implicit}, assumptions \ref{ass:cont}, \ref{ass:imp}  hold with constants
            \begin{equation}
            \label{fastgrowth:ass:imp:constants} \Ca = \gamma \qquad \Cfk = \max\{1, L\},
            \end{equation}
            and $\alpha$ non-constant, equal to
            \begin{equation}
            \label{fastgrowth:ass:imp:alpha} \alpha(y) = \min\{\mu^{A-1}, \mu, 1\} (y+1)^{1-A},
            \end{equation}
            if $f, B, \mu, L, \norm{\c}_*$ satisfy assumptions \ref{fastgrowth:fconvex}, \ref{fastgrowth:norm}, \ref{fastgrowth:fpower}, \ref{fastgrowth:impcondition}.
    \item For the first explicit method \eqref{eq:semiA}, assumptions \ref{ass:cont}, \ref{ass:imp}, and \ref{ass:semiA} hold with constants \eqref{fastgrowth:ass:imp:constants}, $\alpha$ equal to \eqref{fastgrowth:ass:imp:alpha}, and
            \begin{equation}
            \Ck = 2 \qquad \Dfk = \frac{3\Lf}{\min\{\mu^{A-1}, \mu, 1\}},
            \end{equation}
            if $f, B, \mu, L, \Lf, \norm{\c}_*$ satisfy assumptions \ref{fastgrowth:fconvex}, \ref{fastgrowth:norm}, \ref{fastgrowth:fpower}, \ref{fastgrowth:impcondition}, and \ref{fastgrowth:semiAconditionBg2}.
    \item For the first explicit method  \eqref{eq:semiA}, assumptions \ref{ass:cont},\ref{ass:imp}, and \ref{ass:semiA} hold with constants \eqref{fastgrowth:ass:imp:constants}, $\alpha$ equal to \eqref{fastgrowth:ass:imp:alpha}, and
            \begin{equation}
            \Ck = 2 \qquad \Dfk = \frac{6\Lf}{\min\{\mu, 1\}},
            \end{equation}
            if $f, B, \mu, L, \Lf, \norm{\c}_*$ satisfy assumptions \ref{fastgrowth:fconvex},\ref{fastgrowth:norm}, \ref{fastgrowth:fpower}, \ref{fastgrowth:impcondition}, and \ref{fastgrowth:semiAconditionB2}.
\end{enumerate}
\end{prop}

\begin{proof}[Proof of Proposition \ref{lem:assvarphi21}]
	
First, by Lemma \ref{kin:robust}, this choice of $\K$ satisfies assumptions \ref{ass:cont:kconvex} and \ref{ass:semiA:gKpsem} with constant $C_{\K} = 2$. We consider the remaining assumptions of \ref{ass:cont}, \ref{ass:imp}, and \ref{ass:semiA}.

\begin{enumerate}
\item \label{fastgrowth:impproof}
Our first goal is to derive $\alpha$. By assumption \ref{fastgrowth:impcondition}, we have $\mu \varphi_b^B(\norm{x}) \leq \fc(x)$.  Lemma \ref{kin:robustonedim} in Appendix \ref{sec:proofskinetic} implies that $\varphi_2^A(\mu^{-1} t) \leq \max\{\mu^{-2}, \mu^{-A}\} \varphi_2^A(t)$ for $t \geq 0$. Since $(\mu \varphi_b^B(\norm{\c}))^* = \mu (\varphi_b^B)^*(\mu^{-1}\norm{\c}_*)$ by Lemma \ref{kin:robust} and the results discussed in the review of convex analysis, we have by assumption \ref{fastgrowth:fpower} and Lemma \ref{kin:robustonedimconj},
\begin{align*}
\fc^*(p) \leq \mu (\varphi_2^B)^*\l(\mu^{-1}\norm{p}_*\r) \leq  \mu \varphi_2^A\l(\mu^{-1}\norm{p}_*\r) \leq \max\{\mu^{-1}, \mu^{1-A}\} \varphi_2^A(\norm{p}_*).
\end{align*}
Since $\varphi_2^A(0) = \varphi_2^1(0)$, any $\alpha$ satisfies \eqref{eq:alphacond_fcstar} for $p=0$. Assume $p \neq 0$. First, for $y \in [0, \infty)$, we have by rearrangement and convexity,
\begin{equation*}
\varphi_2^A((\varphi_2^1)^{-1}(y)) = \tfrac{1}{A}(y + 1)^A -  \tfrac{1}{A} \leq y(y+1)^{A-1}.
\end{equation*}
Thus,
\begin{align*}
\frac{\K(p)}{\varphi_2^A(\norm{p}_*)} = \frac{\K(p)}{\varphi_2^A((\varphi_2^1)^{-1}(\K(p)))} = \frac{A k(p)}{(\K(p) + 1)^A - 1} \geq (\K(p) + 1)^{1-A}.
\end{align*}
From this we conclude
\begin{align*}
\K(p) &\geq (\K(p) + 1)^{1-A} \varphi_2^A(\norm{p}_*) \geq \alpha(\K(p)) \fc^*(p).
\end{align*}
Since $\K$ is symmetric, we have \eqref{eq:alphacond_fcstar} of assumption \ref{ass:cont:alpha}. To see that $\alpha$ satisfies the remaining conditions of assumption \ref{ass:cont:alpha}, note that $(y+1)^{1-A}$ is convex and decreasing for $A > 1$; $(y+1)^{1-A}$ is non-negative and $\alpha(0) = \min\{\mu^{A-1}, \mu, 1\} \leq 1$. Finally, \ref{fastgrowth:fpower} implies $1 < A \leq 2$, for which,
\begin{align}
\label{eq:xalphaxboundhelper}-\alpha'(y)y  =  \min\{\mu^{A-1}, \mu, 1\} (A-1)(y+1)^{-A}y < (A-1) \alpha(y) \leq \alpha(y).
\end{align}
So we can take $\Ca = \gamma$ and $\alpha$ satisfies assumptions \ref{ass:cont}. This implies that $\K$ satisfies assumptions \ref{ass:cont}. Assumption \ref{fastgrowth:fconvex} is the same as assumption \ref{ass:cont:fconvex}, therefore $f$ and $\K$ satisfy assumptions \ref{ass:cont}.

Now by Fenchel-Young, the symmetry of norms, Lemma \ref{kin:robust}, and assumption \ref{fastgrowth:impcondition},
\begin{align*}
|\inner{\grad \K(p)}{\grad f(x)}| \leq (\varphi_2^1)^*(\norm{\grad \K(p)}) + \varphi_2^1(\norm{\grad f(x)}_*) \leq \Cfk \Ha(x,p),
\end{align*}
where $\Cfk = \max\{1, L\}$ for assumptions \ref{ass:imp}.

\item Assume $B > 2$, so that $A < 2$. The analysis of case \ref{fastgrowth:impproof}. follows and therefore assumptions \ref{ass:cont} and \ref{ass:imp} hold along with the constants just derived. \eqref{eq:xalphaxboundhelper} implies
\begin{align*}
[\alpha(y) y]' = \alpha'(y) y + \alpha(y) =  \alpha'(y) y + (2-A)\alpha(y) + (A-1)\alpha(y) \geq (2-A)\alpha(y).
\end{align*}
Thus, $((y+1)^{2-A} - 1) \leq \alpha(y) y$. Since $2-A = \tfrac{B-2}{B-1}$ and $\tfrac{B-1}{B-2}\varphi^{\tfrac{B-1}{B-2}}_1(y)$ is the inverse function of $(y+1)^{2-A} - 1$, it would be enough to show for $p \in \R^d$ and $x \in \R^d \setminus \{\xmin\}$ that
\begin{align*}
\tfrac{B-1}{B-2}\varphi_{1}^{\tfrac{B-1}{B-2}}\l(\frac{\norm{\grad \K(p)}^2 \maxeigen{\hess f(x)}}{\Lf}\r) \leq 3\Ha(x,p),
\end{align*}
for assumptions \ref{ass:semiA} to hold with constant $\Dfk = 3\Lf/\min\{\mu^{A-1}, \mu, 1\}$. First, for $\psi$ in \ref{eq:kininv} note that for $t \in [0, 1)$,
\begin{equation}
\label{eq:kininvconj} \psi^*(t) = 2(1-t)^{-\tfrac{1}{2}} - 2,
\end{equation}
and that $\psi^*((\varphi_2^1)'(t)^2) = 2\varphi_2^1(t)$. Furthermore, by Lemma \ref{lem:norms} of Appendix \ref{sec:proofskinetic}, we have that $\norm{\grad \K(p)} = (\varphi_2^1)'(\norm{p}_*) < 1$. Lemma \ref{kin:robustonedim} in Appendix \ref{sec:proofskinetic} implies that $\varphi_1^{\tfrac{B-1}{B-2}}(\epsilon t) \leq \epsilon \varphi_1^{\tfrac{B-1}{B-2}}(t)$ for  $\epsilon < 1$ and $t \geq 0$.  All together this implies,
\begin{align*}
\tfrac{B-1}{B-2}\varphi_{1}^{\tfrac{B-1}{B-2}}\l(\frac{\norm{\grad \K(p)}^2 \maxeigen{\hess f(x)}}{\Lf}\r) &\leq \norm{\grad \K(p)}^2 \tfrac{B-1}{B-2}\varphi_{1}^{\tfrac{B-1}{B-2}}\l(\frac{ \maxeigen{\hess f(x)}}{\Lf}\r)\\
&\leq \psi^*(\norm{\grad \K(p)}^2) + \psi\l( \tfrac{B-1}{B-2}\varphi_{1}^{\tfrac{B-1}{B-2}}\l(\frac{ \maxeigen{\hess f(x)}}{\Lf}\r)\r)\\
&\leq 2\K(p) + \psi\l( \tfrac{B-1}{B-2}\varphi_{1}^{\tfrac{B-1}{B-2}}\l(\frac{ \maxeigen{\hess f(x)}}{\Lf}\r)\r) \leq 3\Ha(x,p).
\end{align*}

\item For $B = 2$, the analysis of case \ref{fastgrowth:impproof}. follows and therefore assumptions \ref{ass:cont} and \ref{ass:imp} hold along with the constants just derived.. Here $\alpha$ is equal to
\begin{equation}
\alpha(y) =  \frac{\min(\mu, 1)}{y + 1}.
\end{equation}
Considering that $z / (1-z)$ is the inverse function of $y /(y+1)$ for $z \in [0, 1)$, it would be enough to show for $p \in \R^d$ and $x \in \R^d \setminus \{\xmin\}$ that
\begin{align*}
\frac{\norm{\grad \K(p)}^2 \maxeigen{\hess f(x)}}{2\Lf - \norm{\grad \K(p)}^2 \maxeigen{\hess f(x)}} \leq 3 \Ha(x, p),
\end{align*}
for assumptions \ref{ass:semiA} to hold with constant $\Dfk = 6\Lf/\min\{\mu, 1\}$. Indeed, taking $\psi, \psi^*$ from \eqref{eq:kininv} and \eqref{eq:kininvconj}, we have again $\psi^*((\varphi_2^1)'(t)^2) = 2\varphi_2^1(t)$. Again, by Lemma \ref{lem:norms} of Appendix \ref{sec:proofskinetic}, we have that $\norm{\grad \K(p)} = (\varphi_2^1)'(\norm{p}_*) < 1$. Moreover $z / (2L-z) \leq 1$ for $z \leq L$. All together,
\begin{align*}
\frac{\norm{\grad \K(p)}^2 \maxeigen{\hess f(x)}}{2\Lf - \norm{\grad \K(p)}^2 \maxeigen{\hess f(x)}} &\leq \norm{\grad \K(p)}^2 \frac{ \maxeigen{\hess f(x)}}{2\Lf - \maxeigen{\hess f(x)}}\\
&\leq \psi^*(\norm{\grad \K(p)}^2) + \psi\l(\frac{\maxeigen{\hess f(x)}}{2\Lf - \maxeigen{\hess f(x)}}\r)\\
&\leq 2\K(p) \leq 3 \Ha(x,p).
\end{align*}
\end{enumerate}
\end{proof}

\section{Conclusion}
\label{sec:conclusion}

The conditions of strong convexity and smoothness guarantee the linear convergence of most first-order methods. For a convex function $f$ these conditions are essentially quadratic growth conditions. In this work, we introduced a family of methods, which require only first-order computation, yet extend the class of functions on which linear convergence is achievable. This class of functions is broad enough to capture non-quadratic power growth, and, in particular, functions $f$ whose Hessians may be singular or unbounded. Although our analysis provides ranges for the step size and other parameters sufficient for linear convergence, it does not necessarily provide the optimal choices. It is a valuable open question to identify those choices.

The insight motivating these methods is that the first-order information of a second function, the kinetic energy $\K$, can be used to incorporate global bounds on the convex conjugate $f^*$ in a manner that achieves linear convergence on $f$. This opens a series of theoretical questions about the computational complexity of optimization. Can meaningful lower bounds be derived when we assume access to the first order information of two functions $f$ and $\K$? Clearly, any meaningful answer would restrict $\K$---otherwise the problem of minimizing $f$ could be solved instantly by assuming first-order access to $\K = f^*$ and evaluating $\grad \K(0) = \grad f^*(0) = \xmin$. Exactly what that restriction would be is unclear, but a satisfactory answer would open yet more questions: is there a meaningful hierarchy of lower bounds when access is given to the first-order information of $N > 2$ functions? When access is given to the second-order information of $N > 1$ functions?

From an applied perspective, first-order methods are playing an increasingly important role in the era of large datasets and high-dimensional non-convex problems. In these contexts, it is often impractical for methods to require exact first-order information. Instead, it is frequently assumed that access is limited to unbiased estimators of derivative information. It is thus important to investigate the properties of the Hamiltonian descent methods described in this paper under such stochastic assumptions. For non-convex functions, the success of adaptive gradient methods, which bear a resemblance to our methods using a relativistic kinetic energy, suggests there may be gains from an exploration of other kinetic energies.  Can kinetic energies be designed to condition Hamiltonian descent methods when the Hessian of $f$ is not positive semi-definite everywhere and to encourage iterates to escape saddle points? Finally, the main limitation of the work presented herein is the requirement that a practitioner have knowledge about the behavior of $f$ near its minimum. Therefore, it would be valuable to investigate adaptive methods that do not require such knowledge, but instead estimate it on-the-fly.

\section*{Acknowledgements}
We thank David Balduzzi for reading a draft and his helpful suggestions.
This material is based upon work supported in part by the U.S. Army Research Laboratory and the U. S. Army Research Office, and by the U.K. Ministry of Defence (MoD) and the U.K. Engineering and Physical Research Council (EPSRC) under grant number EP/R013616/1. A part of this research was done while A. Doucet and D. Paulin were hosted by the Institute for Mathematical Sciences in Singapore. C.J. Maddison acknowledges the support of the Natural Sciences and Engineering Research Council of Canada (NSERC) under reference number PGSD3-460176-2014.

\bibliographystyle{plain}
\bibliography{refs} 

\begin{thebibliography}{10}

\bibitem{allen2016even}
Zeyuan Allen-Zhu, Zheng Qu, Peter Richt{\'a}rik, and Yang Yuan.
\newblock Even faster accelerated coordinate descent using non-uniform
  sampling.
\newblock In {\em International Conference on Machine Learning}, pages
  1110--1119, 2016.

\bibitem{penot1995unifconvex}
Dominique Az\'{e} and Jean-Paul Penot.
\newblock Uniformly convex and uniformly smooth convex functions.
\newblock {\em Annales de la Facult\'{e} des Sciences de Toulouse :
  Math\'{e}matiques, S\'{e}rie 6}, 4(4):705--730, 1995.

\bibitem{bertsekas2003convex}
Dimitri~P Bertsekas, Angelia Nedi, and Asuman~E Ozdaglar.
\newblock {\em Convex Analysis and Optimization}.
\newblock Athena Scientific, 2003.

\bibitem{betancourt2018symplectic}
Michael Betancourt, Michael~I Jordan, and Ashia~C Wilson.
\newblock On symplectic optimization.
\newblock {\em arXiv preprint arXiv:1802.03653}, 2018.

\bibitem{bhatt2016second}
Ashish Bhatt, Dwayne Floyd, and Brian~E Moore.
\newblock Second order conformal symplectic schemes for damped {H}amiltonian
  systems.
\newblock {\em Journal of Scientific Computing}, 66(3):1234--1259, 2016.

\bibitem{borwein2010convex}
Jonathan Borwein and Adrian~S Lewis.
\newblock {\em Convex Analysis and Nonlinear Optimization: Theory and
  Examples}.
\newblock Springer Science \& Business Media, 2010.

\bibitem{bottou2018optimization}
L{\'e}on Bottou, Frank~E Curtis, and Jorge Nocedal.
\newblock Optimization methods for large-scale machine learning.
\newblock {\em SIAM Review}, 60(2):223--311, 2018.

\bibitem{boyd2004convex}
Stephen Boyd and Lieven Vandenberghe.
\newblock {\em Convex Optimization}.
\newblock Cambridge university press, 2004.

\bibitem{bubeck2015convex}
S{\'e}bastien Bubeck.
\newblock Convex optimization: Algorithms and complexity.
\newblock {\em Foundations and Trends{\textregistered} in Machine Learning},
  8(3-4):231--357, 2015.

\bibitem{drusvyatskiy2018optimal}
Dmitriy Drusvyatskiy, Maryam Fazel, and Scott Roy.
\newblock An optimal first order method based on optimal quadratic averaging.
\newblock {\em SIAM Journal on Optimization}, 28(1):251--271, 2018.

\bibitem{drusvyatskiy2018error}
Dmitriy Drusvyatskiy and Adrian~S Lewis.
\newblock Error bounds, quadratic growth, and linear convergence of proximal
  methods.
\newblock {\em Mathematics of Operations Research}, 2018.

\bibitem{duane1987hybrid}
Simon Duane, Anthony~D Kennedy, Brian~J Pendleton, and Duncan Roweth.
\newblock Hybrid {M}onte {C}arlo.
\newblock {\em Physics Letters B}, 195(2):216--222, 1987.

\bibitem{duchi2011adaptive}
John Duchi, Elad Hazan, and Yoram Singer.
\newblock Adaptive subgradient methods for online learning and stochastic
  optimization.
\newblock {\em Journal of Machine Learning Research}, 12(Jul):2121--2159, 2011.

\bibitem{fazlyab2017analysis}
Mahyar Fazlyab, Alejandro Ribeiro, Manfred Morari, and Victor~M Preciado.
\newblock Analysis of optimization algorithms via integral quadratic
  constraints: Nonstrongly convex problems.
\newblock {\em arXiv preprint arXiv:1705.03615}, 2017.

\bibitem{flammarion2015averaging}
Nicolas Flammarion and Francis Bach.
\newblock From averaging to acceleration, there is only a step-size.
\newblock In {\em Conference on Learning Theory}, pages 658--695, 2015.

\bibitem{francca2018admm}
Guilherme Franca, Daniel~P Robinson, and Ren{\'e} Vidal.
\newblock {ADMM} and accelerated {ADMM} as continuous dynamical systems.
\newblock {\em International Conference on Machine Learning}, 2018.

\bibitem{francca2018relax}
Guilherme Franca, Daniel~P Robinson, and Ren{\'e} Vidal.
\newblock Relax, and accelerate: A continuous perspective on {ADMM}.
\newblock {\em arXiv preprint arXiv:1808.04048}, 2018.

\bibitem{hinton2014rmsprop}
{Geoffrey Hinton}.
\newblock {Neural Networks for Machine Learning}.
\newblock
  \textsc{url:}~\url{http://www.cs.toronto.edu/\~tijmen/csc321/slides/lecture_slides_lec6.pdf},
  2014.
\newblock Slides 26-31 of Lecture 6.

\bibitem{ghadimi2015global}
Euhanna Ghadimi, Hamid~Reza Feyzmahdavian, and Mikael Johansson.
\newblock Global convergence of the heavy-ball method for convex optimization.
\newblock In {\em Control Conference (ECC), 2015 European}, pages 310--315.
  IEEE, 2015.

\bibitem{girolami2011riemann}
Mark Girolami and Ben Calderhead.
\newblock Riemann manifold {L}angevin and {H}amiltonian {M}onte {C}arlo
  methods.
\newblock {\em Journal of the Royal Statistical Society: Series B (Statistical
  Methodology)}, 73(2):123--214, 2011.

\bibitem{goldstein2011classical}
Herbert Goldstein, Charles~P. Poole, and John Safko.
\newblock {\em {Classical Mechanics}}.
\newblock Pearson Education, 2011.

\bibitem{gurbuzbalaban2017convergence}
Mert Gurbuzbalaban, Asuman Ozdaglar, and Pablo~A Parrilo.
\newblock On the convergence rate of incremental aggregated gradient
  algorithms.
\newblock {\em SIAM Journal on Optimization}, 27(2):1035--1048, 2017.

\bibitem{hastings1970monte}
W~Keith Hastings.
\newblock {M}onte {C}arlo sampling methods using {M}arkov chains and their
  applications.
\newblock {\em Biometrika}, 57(1):97--109, 1970.

\bibitem{jin2017accelerated}
Chi Jin, Praneeth Netrapalli, and Michael~I Jordan.
\newblock Accelerated gradient descent escapes saddle points faster than
  gradient descent.
\newblock {\em arXiv preprint arXiv:1711.10456}, 2017.

\bibitem{juditsky2014deterministic}
Anatoli Juditsky and Yuri Nesterov.
\newblock Deterministic and stochastic primal-dual subgradient algorithms for
  uniformly convex minimization.
\newblock {\em Stochastic Systems}, 4(1):44--80, 2014.

\bibitem{karimi2016linear}
Hamed Karimi, Julie Nutini, and Mark Schmidt.
\newblock Linear convergence of gradient and proximal-gradient methods under
  the {P}olyak-{L}ojasiewicz condition.
\newblock In {\em Joint European Conference on Machine Learning and Knowledge
  Discovery in Databases}, pages 795--811. Springer, 2016.

\bibitem{kingma2014adam}
Diederik~P Kingma and Jimmy Ba.
\newblock Adam: A method for stochastic optimization.
\newblock {\em International Conference on Learning Representations}, 2015.

\bibitem{krichene2015accelerated}
Walid Krichene, Alexandre Bayen, and Peter~L Bartlett.
\newblock Accelerated mirror descent in continuous and discrete time.
\newblock In {\em Advances in Neural Information Processing Systems}, pages
  2845--2853, 2015.

\bibitem{lasalle1960some}
Joseph LaSalle.
\newblock Some extensions of {L}iapunov's second method.
\newblock {\em IRE Transactions on Circuit Theory}, 7(4):520--527, 1960.

\bibitem{lessard2016analysis}
Laurent Lessard, Benjamin Recht, and Andrew Packard.
\newblock Analysis and design of optimization algorithms via integral quadratic
  constraints.
\newblock {\em SIAM Journal on Optimization}, 26(1):57--95, 2016.

\bibitem{livingstone2017kinetic}
Samuel Livingstone, Michael~F Faulkner, and Gareth~O Roberts.
\newblock Kinetic energy choice in {H}amiltonian/hybrid {M}onte {C}arlo.
\newblock {\em arXiv preprint arXiv:1706.02649}, 2017.

\bibitem{lu2016relativistic}
Xiaoyu Lu, Valerio Perrone, Leonard Hasenclever, Yee~Whye Teh, and Sebastian
  Vollmer.
\newblock Relativistic {M}onte {C}arlo.
\newblock In {\em Artificial Intelligence and Statistics}, pages 1236--1245,
  2017.

\bibitem{mclachlan2001conformal}
Robert McLachlan and Matthew Perlmutter.
\newblock Conformal {H}amiltonian systems.
\newblock {\em Journal of Geometry and Physics}, 39(4):276--300, 2001.

\bibitem{metropolis1953equation}
Nicholas Metropolis, Arianna~W Rosenbluth, Marshall~N Rosenbluth, Augusta~H
  Teller, and Edward Teller.
\newblock Equation of state calculations by fast computing machines.
\newblock {\em The Journal of Chemical Physics}, 21(6):1087--1092, 1953.

\bibitem{neal2011mcmc}
Radford~M Neal.
\newblock {MCMC} using {H}amiltonian dynamics.
\newblock {\em Handbook of Markov Chain Monte Carlo}, pages 113--162, 2011.

\bibitem{necoara2018linear}
Ion Necoara, Yu~Nesterov, and Francois Glineur.
\newblock Linear convergence of first order methods for non-strongly convex
  optimization.
\newblock {\em Mathematical Programming}, pages 1--39, 2018.

\bibitem{nemirovskii1985optimal}
Arkaddii~S Nemirovskii and Yurii~E Nesterov.
\newblock Optimal methods of smooth convex minimization.
\newblock {\em USSR Computational Mathematics and Mathematical Physics},
  25(2):21--30, 1985.

\bibitem{nemirovsky1983problem}
Arkadii~S Nemirovsky and David~B Yudin.
\newblock {\em Problem Complexity and Method Efficiency in Optimization}.
\newblock Wiley Interscience, 1983.

\bibitem{nesterov2005smooth}
Yu~Nesterov.
\newblock Smooth minimization of non-smooth functions.
\newblock {\em Mathematical programming}, 103(1):127--152, 2005.

\bibitem{nesterov2008accelerating}
Yurii Nesterov.
\newblock Accelerating the cubic regularization of {N}ewton's method on convex
  problems.
\newblock {\em Mathematical Programming}, 112(1):159--181, 2008.

\bibitem{nesterov2013introductory}
Yurii Nesterov.
\newblock {\em Introductory Lectures on Convex Optimization: A Basic Course},
  volume~87.
\newblock Springer Science \& Business Media, 2013.

\bibitem{peano1990demonstration}
Giuseppe Peano.
\newblock D{\'e}monstration de l’int{\'e}grabilit{\'e} des {\'e}quations
  diff{\'e}rentielles ordinaires.
\newblock In {\em Arbeiten zur Analysis und zur mathematischen Logik}, pages
  76--126. Springer, 1990.

\bibitem{Perko}
Lawrence Perko.
\newblock {\em Differential Equations and Dynamical Systems}, volume~7.
\newblock Springer Science \& Business Media, 2013.

\bibitem{polyak1964some}
Boris~T Polyak.
\newblock Some methods of speeding up the convergence of iteration methods.
\newblock {\em USSR Computational Mathematics and Mathematical Physics},
  4(5):1--17, 1964.

\bibitem{polyak1987introduction}
Boris~T Polyak.
\newblock {\em Introduction to {O}ptimization}.
\newblock 1987.

\bibitem{rechtlecturenotes}
Benjamin Recht.
\newblock {CS}726 - {L}yapunov analysis and the heavy ball method.
\newblock {\em Lecture notes}, 2012.

\bibitem{rockafellar1970convex}
R.~Tyrrell Rockafellar.
\newblock {\em Convex Analysis}.
\newblock Princeton University Press, 1970.

\bibitem{roulet2017sharpness}
Vincent Roulet and Alexandre d'Aspremont.
\newblock Sharpness, restart and acceleration.
\newblock In {\em Advances in Neural Information Processing Systems}, pages
  1119--1129, 2017.

\bibitem{rudin}
Walter Rudin.
\newblock {\em Real and Complex Analysis}.
\newblock McGraw-Hill Book Co., New York, third edition, 1987.

\bibitem{shalev2007online}
Shai Shalev-Shwartz and Yoram Singer.
\newblock Online learning: Theory, algorithms, and applications.
\newblock 2007.

\bibitem{stoltz2018langevin}
Gabriel Stoltz and Zofia Trstanova.
\newblock Langevin dynamics with general kinetic energies.
\newblock {\em Multiscale Modeling \& Simulation}, 16(2):777--806, 2018.

\bibitem{su2014differential}
Weijie Su, Stephen Boyd, and Emmanuel Cand{\`e}s.
\newblock A differential equation for modeling {N}esterov’s accelerated
  gradient method: Theory and insights.
\newblock In {\em Advances in Neural Information Processing Systems}, pages
  2510--2518, 2014.

\bibitem{su2016differential}
Weijie Su, Stephen Boyd, and Emmanuel~J Cand{\`e}s.
\newblock A differential equation for modeling {N}esterov's accelerated
  gradient method: theory and insights.
\newblock {\em Journal of Machine Learning Research}, 17(1):5312--5354, 2016.

\bibitem{sutskever2013importance}
Ilya Sutskever, James Martens, George Dahl, and Geoffrey Hinton.
\newblock On the importance of initialization and momentum in deep learning.
\newblock In {\em International Conference on Machine Learning}, pages
  1139--1147, 2013.

\bibitem{wibisono2016variational}
Andre Wibisono, Ashia~C Wilson, and Michael~I Jordan.
\newblock A variational perspective on accelerated methods in optimization.
\newblock {\em Proceedings of the National Academy of Sciences},
  113(47):E7351--E7358, 2016.

\bibitem{wilson2016lyapunov}
Ashia~C Wilson, Benjamin Recht, and Michael~I Jordan.
\newblock A {L}yapunov analysis of momentum methods in optimization.
\newblock {\em arXiv preprint arXiv:1611.02635}, 2016.

\bibitem{wilson2017marginal}
Ashia~C Wilson, Rebecca Roelofs, Mitchell Stern, Nati Srebro, and Benjamin
  Recht.
\newblock The marginal value of adaptive gradient methods in machine learning.
\newblock In {\em Advances in Neural Information Processing Systems}, pages
  4148--4158, 2017.

\bibitem{yang2015rsg}
Tianbao Yang and Qihang Lin.
\newblock Rsg: Beating subgradient method without smoothness and strong
  convexity.
\newblock {\em Journal of Machine Learning Research}, 19(1):1--33, 2018.

\bibitem{zeiler2012adadelta}
Matthew~D Zeiler.
\newblock Adadelta: an adaptive learning rate method.
\newblock {\em arXiv preprint arXiv:1212.5701}, 2012.

\bibitem{zalinescu1983uniformly}
Constantin Z\u{a}linescu.
\newblock On uniformly convex functions.
\newblock {\em Journal of Mathematical Analysis and Applications},
  95(2):344--374, 1983.

\bibitem{zalinescu2002convex}
Constantin Z\u{a}linescu.
\newblock {\em Convex Analysis in General Vector Spaces}.
\newblock World {S}cientific, 2002.

\end{thebibliography}

\appendix

\section{Proofs for convergence of continuous systems}
\label{sec:proofscontinuous}

\Vderivativebound*

\begin{proof}
\begin{align*}
{\Ly_t}' &= -\gamma \inner{\grad \K(p_t)}{p_t} + \beta\inner{\grad \K(p_t)}{p_t} - \beta \gamma \inner{x_t- \xmin}{p_t} - \beta\inner{x_t-\xmin}{\grad \f(x_t)}\\
    &= -(\gamma - \beta) \inner{\grad \K(p_t)}{p_t} - \beta \gamma \inner{x_t-\xmin}{p_t} - \beta\inner{x_t-\xmin}{\grad \f(x_t)}\\
    &\leq -(\gamma - \beta) \K(p_t) - \beta \gamma \inner{x_t-\xmin}{p_t} - \beta (\f(x_t)-\f(\xmin))
\end{align*}
by convexity and $\beta \leq \gamma$. Our goal is to show that ${\Ly_t}' \le -\lambda \Ly_t$ for some $\lambda>0$, which would hold if 
\begin{align}
\nonumber&-(\gamma - \beta) \K(p_t) - \beta \gamma \inner{x_t-\xmin}{p_t} - \beta (\f(x_t)-\f(\xmin))\le -\lambda (\K(p_t) + \f(x_t)-\f(\xmin)+\beta \inner{x_t-\xmin}{p_t})\\
\intertext{which is equivalent by rearrangement to}
\label{eq:goal}&-\beta(\gamma-\lambda)\inner{x_t-\xmin}{p_t}\le (\gamma - \beta-\lambda) \K(p_t)+(\beta-\lambda) (\f(x_t)-\f(\xmin)).
\end{align}
Assume that $\lambda \le \gamma$.
By assumption on $f$, $\K$, and $\alpha$ we have \eqref{eq:Handinnerbound1}, which implies by rearrangement that $k(p_t) \ge -\alpha \inner{x_t-\xmin}{p_t} - \alpha (f(x_t)-\f(\xmin))$, so
\begin{equation}\label{eq:crosstermbnd}
-\beta (\gamma - \lambda) \inner{x_t-\xmin}{p_t} \le \frac{\beta}{\alpha} (\gamma - \lambda) (k(p_t) + \alpha (f(x_t)-\f(\xmin))),
\end{equation}
and $k(p_t) \ge 0$ and $f(x_t)-\f(\xmin) \ge 0$, hence it is enough to have $\frac{\beta}{\alpha} (\gamma - \lambda) \le \gamma-\beta-\lambda$ and $\beta (\gamma - \lambda) \le \beta-\lambda$ for showing \eqref{eq:goal}.
Thus we need $\lambda \le \min(\gamma, \frac{\alpha \gamma - \alpha \beta - \beta \gamma}{\alpha - \beta}, \frac{\beta(1-\gamma)}{1-\beta} )$.
Here $\frac{\beta}{1-\beta} (1-\gamma) \le \gamma$ for $0 < \beta \le \gamma < 1$, therefore ${\Ly_t}'\le -\lambda(\alpha,\beta,\gamma) \Ly_t$ for
\begin{equation*}
\lambda(\alpha,\beta,\gamma)=\min\l(\frac{\alpha \gamma - \alpha \beta - \beta \gamma}{\alpha - \beta}, \frac{\beta(1-\gamma)}{1-\beta} \r).
\end{equation*}
In order to obtain the optimal contraction rate, we need to maximize $\lambda(\alpha,\beta,\gamma)$ in $\beta$. Without loss of generality, we can assume that $0<\beta<\frac{\alpha \gamma}{\alpha +\gamma}$, and it is easy to see that on this interval, $\frac{\alpha \gamma - \alpha \beta - \beta \gamma}{\alpha - \beta}$ is strictly monotone decreasing, while $\frac{\beta(1-\gamma)}{1-\beta}$ is strictly monotone increasing. Therefore, the maximum will be taken when the two terms are equal. This leads to a quadratic equation with two solutions 
\[
\beta_{\pm} = \frac{1}{1+\alpha} \left(\alpha+\frac{\gamma}{2} \pm \sqrt{(1-\gamma) \alpha^2 + \frac{\gamma^2}{4}}\right).
\]
One can check that $\beta_+ > \frac{\alpha \gamma}{\alpha+\gamma}$, while $0<\beta_-<\frac{\alpha \gamma}{\alpha+\gamma}$, hence 
\[\max_{\beta \in [0,\alpha]} 
\lambda(\alpha,\beta,\gamma)=\lambda(\alpha,\beta_{-},\gamma)=\frac{1}{1-\alpha} \left((1-\gamma) \alpha + \frac{\gamma}{2} - \sqrt{(1-\gamma) \alpha^2 + \frac{\gamma^2}{4}}\right)
\]
for $\alpha<1$. For $\alpha=1$, we obtain that $\beta^{\star} = \beta_-=\frac{\gamma}{2}$, and $\lambda^{\star} = \frac{\gamma (1-\gamma)}{2-\gamma}.$  

Now assume $\beta \in (0, \alpha\gamma/2]$. Since we have shown that $\lambda(\alpha, \gamma, \beta) = \tfrac{\beta(1-\gamma)}{1-\beta}$ for $\beta < \beta_{-}$ it is enough to show that $\beta_{-} > \alpha \gamma / 2$ to get our result. Notice that $\beta_{-}$ as a function of $\gamma$, $\beta_{-}(\gamma)$, is strictly concave with $\beta_{-}(0) = 0$, and $\beta_{-}(1) = \tfrac{\alpha}{1 + \alpha}$, thus
\begin{align*}
\beta_{-} = \beta_{-}(\gamma) > \gamma \beta_{-}(1) = \frac{\gamma \alpha}{1 + \alpha} \geq \frac{\alpha \gamma}{2} \geq \beta. 
\end{align*}
Finally, the proof of \eqref{eq:goodenoughbeta2} is equivalent by rearrangement to showing that for $\lambda=(1-\gamma)\beta$,
\begin{equation}
\label{eq:goal2}-\beta(\gamma-\lambda)\inner{x_t-\xmin}{p_t}\le (\gamma - \beta-\gamma^2(1-\gamma)/4-\lambda) \K(p_t)+(\beta-\lambda) (\f(x_t)-\f(\xmin)),
\end{equation}
hence by \eqref{eq:crosstermbnd} it suffices to show that we have $\frac{\beta}{\alpha} (\gamma - \lambda) \le \gamma-\gamma^2(1-\gamma)/4-\beta-\lambda$ and $\beta (\gamma - \lambda) \le \beta-\lambda$. The latter one was already verified in the previous section, and the first one is equivalent to 
\[
\gamma-\beta-\lambda-\frac{\beta}{\alpha} (\gamma - \lambda)\ge \gamma^2(1-\gamma)/4 \text{ for every }0<\gamma<1,\, 0<\alpha\le 1,\, 0<\beta\le \alpha\gamma/2.
\]
It is easy to see that we only need to check this for $\beta=\alpha\gamma/2$, and in this case by minimizing the left hand side for $0\le \alpha\le 1$ and using the fact that $\lambda=(1-\gamma)\beta$, we obtain the claimed result.
\end{proof}

\section{Proofs for partial lower bounds}\label{sec:proofslowerbounds}

In this section, we present the proofs of the lower bounds. First, we show the existence and uniqueness of solutions. 

\Lemmasolutionsexistunique*

\begin{proof}
	Let $\Ha_t := \frac{|x_t|^{\xpower}}{{\xpower}} + \frac{|p_t|^{\ppower}}{{\ppower}}$, then $\Ha_t\ge 0$ and
	\begin{align*}
	\Ha_t' = |x_t|^{{\xpower}-1} \sgn(x_t) x_t' + |p_t|^{{\ppower}-1} \sgn(p) p_t' = - \gamma |p_t|^{\ppower},
	\end{align*}
	so $0\ge \Ha_t'\ge -\gamma {\ppower} \Ha_t$. By Grönwall's inequality, this implies that for any solution of \eqref{eq:ode},
	\begin{align}\label{eqHabnd1}
	&\Ha_t\le \Ha_0\text{ for }t\ge 0\text{, and }\\
	\label{eqHabnd2}
	&\Ha_t\le \Ha_0 \exp(-\gamma {\ppower} t)\text{ for }t<0. 
	\end{align}

	The derivatives $x'$, $p'$ are continuous functions of $(x,p)$, and these functions are locally Lipschitz if $x \neq 0$ and $p \neq 0$.
	So by the Picard-Lindel\"of theorem, if $x_{t_0} \neq 0$, $p_{t_0} \neq 0$, then there exists a unique solution in the interval $(t_0-\epsilon, t_0+\epsilon)$ for some $\epsilon>0$. 
	
	Now we will prove local existence and uniqueness for $(x_{t_0},p_{t_0}) = (x_0,0)$ with $x_0 \neq 0$, and for $(x_{t_0},p_{t_0}) = (0,p_0)$ with $p_0 \neq 0$. Because of the central symmetry, we may assume that $x_0 > 0$ and $p_0 > 0$, and we may also assume that $t_0 = 0$.
	
	First let $x_0 = 0$ and $p_0 > 0$.
	We take $t$ close enough to $0$ so that $p_t > 0$.
	Then $x'_t = p_t^{{\ppower}-1} > 0$ and $p'_t = -|x_t|^{{\xpower}-1} \sgn(x_t) - \gamma p_t$. 
	Then $p_t = \phi(x_t)$ for some function $\phi \colon(-\epsilon, \epsilon) \to \R_{>0}$, where $t$ is close enough to $0$.
	Here $\phi(0) = p_0$ and $p'_t = \phi'(x_t) x'_t$, so
	\begin{align*}
		\phi'(x_t) &= \frac{p'_t}{x'_t} = -(|x_t|^{{\xpower}-1} \sgn(x_t) + \gamma p_t) p_t^{1-{\ppower}} \\
		&= -(|x_t|^{{\xpower}-1} \sgn(x_t) + \gamma \phi(x_t)) \phi(x_t)^{1-{\ppower}},
	\end{align*}
	and hence
	\[
	\phi'(u) = -(|u|^{{\xpower}-1} \sgn(u) + \gamma \phi(u)) \phi(u)^{1-{\ppower}} \quad \textrm{and} \quad \phi(0) = p_0.
	\]
	This ODE satisfies the conditions of the Picard-Lindel\"of theorem, so $\phi$ exists and is unique in a neighborhood of $0$.
	Then $x_0 = 0$ and $x'_t = \phi(x_t)^{{\ppower}-1}$, so for the Picard-Lindel\"of theorem we just need to check that $u \mapsto \phi(u)^{{\ppower}-1}$ is Lipschitz in a neighborhood of $0$.
	This is true, because $\phi$ is $C^1$ in a neighborhood of $0$.
	So $x_t$ exists and is unique when $t$ is near $0$, hence $p_t = \phi(x_t)$ also exists and is unique there.

	Now let $x_0 = x_0 > 0$ and $p_0 = 0$.
	We take $t$ close enough to $0$ so that $x_t > 0$.
	Then $x'_t = |p_t|^{{\ppower}-1} \sgn(p_t)$ and $p'_t = -x_t^{{\xpower}-1} - \gamma p_t < 0$ for $t$ close enough to $0$.
	Then $x_t = \psi(p_t)$ for some function $\psi \colon(-\epsilon, \epsilon) \to \R_{>0}$, where $t$ is close enough to $0$.
	Here $\psi(0) = x_0$ and $x'_t = \phi'(p_t) p'_t$, so
	\[
	\psi'(p_t) = \frac{x'_t}{p'_t} = -\frac{|p_t|^{{\ppower}-1} \sgn(p_t)}{x_t^{{\xpower}-1} + \gamma p_t} = -\frac{|p_t|^{{\ppower}-1} \sgn(p_t)}{\psi(p_t)^{{\xpower}-1} + \gamma p_t},
	\]
	and thus
	\[
	\psi'(u) = -\frac{|u|^{{\ppower}-1} \sgn(u)}{\psi(u)^{{\xpower}-1} + \gamma u} \quad \textrm{and} \quad \psi(0) = x_0.
	\]
	This ODE satisfies the conditions of the Picard-Lindel\"of theorem, so $\psi$ exists and is unique in a neighborhood of $0$.
	Then $p_0 = 0$ and $p'_t = -\psi(p_t)^{{\xpower}-1} - \gamma p_t$, so for the Picard-Lindel\"of theorem we just need to check that $u \mapsto -\psi(u)^{{\xpower}-1} - \gamma u$ is Lipschitz in a neighborhood of $0$.
	This is true, because $\psi$ is $C^1$ in a neighborhood of $0$.
	So $p_t$ exists and is unique when $t$ is near $0$, hence $x_t = \psi(p_t)$ also exists and is unique there.

	Let $[0, t_{\max})$ and $(-t_{\min},0]$ be the longest intervals where the solution exists and unique. If $t_{\max}<\infty$ or $t_{\min}<\infty$, then by Theorem 3 of \cite{Perko}, page 91, the solution would have to be able to leave any compact set $K$ in the interval $[0,t_{\max})$ or $(-t_{\max},0]$, respectively. However, due to the \eqref{eqHabnd1} and \eqref{eqHabnd2}, the energy function $\Ha_t$ cannot converge to infinity in finite amount of time, so this is not possible. Hence, the existence and uniqueness for every $t\in \R$ follows.
\end{proof}

Before proving Theorem \ref{thmspeedofconv}, we need to show a few preliminary results.

\begin{lem} \label{Lemma-H-limits}
	If $(x,p)$ is not constant zero, then $\lim_{t \to -\infty} \Ha_t = \infty$ and $\lim_{t \to \infty} \Ha_t = \lim_{t \to \infty} x_t = \lim_{t \to \infty} p_t = 0$.
\end{lem}
\begin{proof}
	The limits of $\Ha$ exist, because $\Ha'_t = - \gamma |p_t|^{\ppower} \le 0$.
	First suppose that $\lim_{t \to -\infty} \Ha_t = M < \infty$.
	Then $\Ha_t \le M$ for every $t$, so $x$ and $p$ are bounded functions, therefore $x'$ and $p'$ are also bounded by the differential equation.
	Then $\Ha''_t = - \gamma {\ppower} |p_t|^{{\ppower}-1} \sgn(p_t) p'_t$ is also bounded, so $\Ha'$ is Lipschitz.
	This together with $\lim_{t \to -\infty} \Ha_t = M \in \R$ implies that $\lim_{t \to -\infty} \Ha'_t = 0$.
	So $\lim_{t \to -\infty} p_t = 0$.
	Then we must have $\lim_{t \to -\infty} x_t = x_0$ for some $x_0 \in \R \setminus \{0\}$.
	But then $\lim_{t \to -\infty} p'_t = -|x_0|^{{\xpower}-1} \sgn(x_0) \neq 0$, which contradicts $\lim_{t \to -\infty} p_t = 0$.
	So indeed $\lim_{t \to -\infty} \Ha_t = \infty$.
	
	Now suppose that $\lim_{t \to \infty} \Ha_t > 0$.
	For $t \in [0, \infty)$ we have $\Ha_t \le \Ha_0$, so for $t \ge 0$ the functions $x$ and $p$ are bounded, hence also $x'$, $p'$, $\Ha'$, $\Ha''$ are bounded there.
	So $\lim_{t \to \infty} \Ha_t \in \R$, and $H'$ is Lipschitz for $t \ge 0$, therefore $\lim_{t \to \infty} \Ha'_t = 0$, thus $\lim_{t \to \infty} p_t = 0$.
	Then we must have $\lim_{t \to \infty} x_t = x_0$ for some $x_0 \in \R \setminus \{0\}$.
	But then $\lim_{t \to \infty} p'_t = -|x_0|^{{\xpower}-1} \sgn(x_0) \neq 0$, which contradicts $\lim_{t \to \infty} p_t = 0$.
	So indeed $\lim_{t \to \infty} \Ha_t = 0$, thus $\lim_{t \to \infty} x_t = \lim_{t \to \infty} p_t = 0$.
\end{proof}

From now on we assume that $\frac{1}{{\ppower}} +\frac{1}{{\xpower}} < 1$.

\begin{lem} \label{Lemma-p-is-zero-many-times}
	If $(x_t,p_t)$ is a solution, then for every $t_0 \in \R$ there is a $t \le t_0$ such that $p_t = 0$.
\end{lem}
\begin{proof}
	The statement is trivial for the constant zero solution, so assume that $(x,p)$ is not constant zero.
	Then $\lim_{t \to -\infty} \Ha_t = \infty$.
	Suppose indirectly that $p_t < 0$ for every $t \le t_0$.
	Then $x'_t = -|p_t|^{{\ppower}-1} < 0$ for $t \le t_0$.
	If $\lim_{t \to -\infty} x_t = x_{-\infty} < \infty$, then $\lim_{t \to -\infty} p_t = -\infty$, because $\lim_{t \to -\infty} \Ha_t = \infty$.
	Then $x \to x_{-\infty} \in \R$ and $x' = -|p|^{{\ppower}-1} \to -\infty$ when $t \to -\infty$, which is impossible.
	So $\lim_{t \to -\infty} x_t = \infty$, hence there is a $t_1 \le t_0$ such that $p_t < 0 < x_t$ for every $t \le t_1$.
	Let $G_t := \frac{|p_t|^{{\ppower}-1}}{{\ppower}-1} - \gamma x_t$, then for $t \le t_1$ we have
	\[
	G_t' = -|p_t|^{{\ppower}-2} p_t' - \gamma x_t' = |p_t|^{{\ppower}-2} (x_t^{{\xpower}-1} - \gamma |p_t|) + \gamma |p_t|^{{\ppower}-1} = |p_t|^{{\ppower}-2} x_t^{{\xpower}-1} > 0.
	\]
	So $G_t \le G(t_1)$ for every $t \le t_1$.
	Thus
	\[
	|p_t| \le (({\ppower}-1) (\gamma x_t + G(t_1)))^{\frac{1}{{\ppower}-1}} = (A x_t + B)^{\frac{1}{{\ppower}-1}},
	\]
	for $t \le t_1$, where $A > 0$.
	For big enough $x$ we have $(A x + B)^{\frac{1}{{\ppower}-1}} < \frac{1}{\gamma} x^{{\xpower}-1}$, because $\frac{1}{{\ppower}-1} < {\xpower}-1$, since ${\xpower}{\ppower}>{\xpower}+{\ppower}$.
	So there is a $t_2 \le t_1$ such that $p_t < 0 < x_t$ and $|p_t| \le \frac{1}{\gamma} x_t^{{\xpower}-1}$ for every $t \le t_2$.
	Then $p_t' = -x_t^{{\xpower}-1} - \gamma p_t \le 0$, so $p_t < 0$ is monotone decreasing for $t \in (-\infty, t_2]$, hence $p_{-\infty} = \lim_{t \to -\infty} p_t \in \R$.
	But then $p'_t = -x_t^{{\xpower}-1} - \gamma p_t \to -\infty$ when $t \to -\infty$, which together with $p_{-\infty} \in \R$ is impossible.
	This contradiction shows that indeed there is a $t \le t_0$ such that $p_t \ge 0$.
	Applying this for $(-x,-p)$, we get that there is a $t \le t_0$ such that $p_t \le 0$.
	So by continuity, there is a $t \le t_0$ such that $p_t = 0$.
\end{proof}

For $A>\frac{1}{\gamma}$ let $\xi(A) := (\frac{\gamma A - 1}{({\xpower}-1) A^{\ppower}})^{\frac{1}{{\xpower}{\ppower}-{\xpower}-{\ppower}}}$ and
\begin{equation}\label{eqRAdef}
\mathcal{R}_A := \{(x,p) \in \R^2; \, 0 < x < \xi(A), \, -A x^{{\xpower}-1} < p < 0 \},
\end{equation}
\begin{lem} \label{Lemma-trapping-set}
	Let $A > \frac{1}{\gamma}$.
	If $(x,p)$ is a solution, $t_0 \in \R$ and $(x_{t_0},p_{t_0}) \in \mathcal{R}_A$, then $(x_t,p_t) \in \mathcal{R}_A$ for every $t \ge t_0$.
\end{lem}
\begin{proof}
	Suppose indirectly that there is a $t > t_0$ such that $(x_t, p_t) \notin \mathcal{R}_A$.
	Let $T$ be the infimum of these $t$'s.
	Then $T > t_0$, and $(x(T),p(T))$ is on the boundary of the region $\mathcal{R}_A$.
	We cannot have $(x(T),p(T)) = (0,0)$, because $(0,0)$ is unreachable in finite time.
	Since $x'_t = -|p_t|^{{\ppower}-1} \le 0$ for $t \in [t_0, T)$, we have $x(T) \le x_{t_0} < \xi(A)$.
	So we have $0 < x(T) < \xi(A)$ and either $p(T) = 0$ or $p(T) = -A x(T)^{{\xpower}-1}$.
	
	Suppose that $p(T) = 0$.
	Then $p'(T) = -x(T)^{{\xpower}-1} < 0$, so if $t \in (t_0, T)$ is close enough to $T$, then $p_t > 0$, which contradicts $(x_t, p_t) \in \mathcal{R}_A$.
	So $p(T) = -A x(T)^{{\xpower}-1}$.
	Let $U_t := p_t + A x_t^{{\xpower}-1}$.
	Then $U(T) = 0$, and by the definition of $T$, we must have $U'(T) \le 0$.
	Using $|p(T)| = A x(T)^{{\xpower}-1}$ we get
	\begin{align*}
		0 &\ge U'(T) = p'(T) + ({\xpower}-1) A x(T)^{{\xpower}-2} x'(T) \\
		&= \gamma |p(T)| - x(T)^{{\xpower}-1} - ({\xpower}-1) A x(T)^{{\xpower}-2} |p(T)|^{{\ppower}-1} \\
		&= x(T)^{{\xpower}-1} (\gamma A - 1 - ({\xpower}-1) A^{\ppower} x(T)^{{\xpower}{\ppower}-{\xpower}-{\ppower}}),
	\end{align*}
	so $\xi(T) = (\frac{\gamma A - 1}{({\xpower}-1) A^{\ppower}})^{\frac{1}{{\xpower}{\ppower}-{\xpower}-{\ppower}}} \le x(T) < \xi(T)$.
	This contradiction proves that $(x_t,p_t) \in \mathcal{R}_A$ for every $t \ge t_0$.
\end{proof}

The following lemma characterises the paths of every solution of the ODE in terms of single parameter $\theta$.
\begin{lem} \label{Lemma-parametrization-of-solutions}
	There is a constant $\eta>0$ such that every solution $(x_t,p_t)$ which is not constant zero is of the form $(x^{(\theta)}_{t+\Delta}, p^{(\theta)}_{t+\Delta})$ for exactly one $\theta \in [-\eta, \eta] \setminus \{0\}$ and $\Delta \in \R$.
\end{lem}
\begin{proof}
For $u>0$ let us take the solution $(x_t,p_t)$ with $(x_{t_0},p_{t_0}) = (-u,0)$ for some $t_0 \in \R$.
Then $p'_{t_0} = u^{{\xpower}-1} > 0$, so $p_{t_0} < 0$ if $t<t_0$ is close enough to $t_0$.
By Lemma \ref{Lemma-p-is-zero-many-times}, there is a smallest $\mathcal{T}(u) \in \R_{>0}$ such that $p(t_0-\mathcal{T}(u)) = 0$.
We may take $t_0 = \mathcal{T}(u)$, and call this solution $(X^{(u)}(t), P^{(u)}(t))$. Then $X^{(u)}_{\mathcal{T}(u)} = -u$, $P^{(u)}_{\mathcal{T}(u)} = P^{(u)}_0 = 0$, and $P^{(u)}_t < 0$ for $t \in (0, \mathcal{T}(u))$. Let $g(u) = X^{(u)}_0$. Here $g(u) \neq 0$, because we cannot reach $(0,0)$ in finite time due to \eqref{eqHabnd1}-\eqref{eqHabnd2}. We cannot have $g(u)<0$, because then $P_u'(0) = -|g(u)|^{{\xpower}-1} \sgn(g(u)) > 0$.
So $g(u)>0$, thus we have defined a function $g \colon \R_{>0} \to \R_{>0}$.
For $u > 0$ let us take the continuous path
\[
\mathcal{C}_u := \{(X^{(u)}_t,P^{(u)}_t); \, t \in [0, \mathcal{T}(u)]\}.
\]
Note that this path is below the $x$-axis except for the two endpoints, which are on the $x$-axis.
If $0 < u < v$, then $\mathcal{C}_u \cap \mathcal{C}_v = \varnothing$, so we must have $g(u) < g(v)$ (otherwise the two paths would have to cross).
So $g$ is strictly increasing.
Let
\[
\eta := \lim_{u \to 0} g(u) \in \R_{\ge 0}.
\]
If $0<u<v$ and $z \in (g(u), g(v))$, then going forward in time after the point $(z,0)$, the solution must intersect the $x$-axis first somewhere between the points $(-v,0)$ and $(-u,0)$, thus $z$ is in the image of $\eta$.
So $\eta \colon \R_{>0} \to (\eta, \infty)$ is a strictly increasing bijective function.
We have $g(u) > u$ for every $u>0$, because if $g(u) \le u$, then for the solution $(X^{(u)}_t, P^{(u)}_t)$ we have $\Ha_0 \le H(\mathcal{T}(u))$ and $\Ha'_t = - \gamma |P^{(u)}_t|^{\ppower} < 0$ for $t \in (0, \mathcal{T}(u))$, which is impossible.

Let $A > \frac{1}{\gamma}$ and $0 < z < \xi(A)$, and take the solution $(x_t,p_t)$ with $(x_0,p_0) = (z,0)$.
Then $p'_0 < 0$, so $(x_t,p_t) \in \mathcal{R}_A$ for $t>0$ close enough to $0$, and then by Lemma \ref{Lemma-trapping-set}, $(x_t, p_t) \in \mathcal{R}_A$ for every $t>0$.
So $z$ is not in the image of $\eta$, hence $\eta \ge \xi(A) > 0$.

Let $(x_t,p_t)$ be a solution, which is not constant zero.
Let $\mathcal{S} = \{t \in \R; \, p_t = 0\}$.
This is a closed, nonempty subset of $\R$.
Suppose that $\sup(\mathcal{S}) = \infty$.
Since $\lim_{t \to \infty} x_t = 0$, this means that there are infinitely many $t \in \R$ such that $p_t = 0$ and $|x_t| \in (0, \eta)$.
This is impossible, since there can be only one such $t$.
So $\sup(\mathcal{S}) = \max(\mathcal{S}) = T \in \R$.
We may translate time so that $T = 0$.
Then $p_0 = 0$ and $p_t \neq 0$ for every $t > 0$.
If $|x_0| > \eta$, then later we again intersect the $x$-axis, so we must have $0 < |x_0| \le \eta$.
So the not constant zero solutions can be described by their last intersection with the $x$-axis, and this intersection has its $x$-coordinate in $[-\eta, \eta] \setminus \{0\}$.
\end{proof}

By symmetry, we have $x^{(-\theta)}_t = -x^{(\theta)}_t$ and $p^{(-\theta)}_t = -p^{(\theta)}_t$. We now study the solutions $(x^{(\theta)}_t, p^{(\theta)}_t)$ for $0 < \theta \le \eta$ and $t \ge 0$.
Then $p^{(\theta)}_t < 0 < x^{(\theta)}_t$ for $t \ge 0$ and $(x^{(\theta)})'_t < 0$ and $\lim_{t \to \infty} x^{(\theta)}_t = 0$, so $x^{(\theta)}_t \in (0, \theta)$ for $t > 0$.
So we can write $p^{(\theta)}_t = -\phi^{(\theta)}(x^{(\theta)}_t)$, where $\phi^{(\theta)} \colon (0,\theta) \to \R_{>0}$ and $\lim_{z \to 0} \phi^{(\theta)}(z) = \lim_{z \to \theta} \phi^{(\theta)}(z) = 0$.
Let
\[
\Theta := \l\{\theta \in (0, \eta]; \, (z, -\phi^{(\theta)}(z)) \in \mathcal{R}_A \textrm{ for some } A > \frac{1}{\gamma} \textrm{ and } z \in (0, \theta) \r\}.
\]
The following lemmas characterize this set.

\begin{lem} \label{Lemma-non-Theta-asymptotic}
	If $\theta \in (0, \eta] \setminus \Theta$, then
	\[
	\lim_{z \to 0} \frac{\phi^{(\theta)}(z)}{(\gamma ({\ppower}-1) z)^{\frac{1}{{\ppower}-1}}} = 1.
	\]
\end{lem}
\begin{proof}
By the definition of $\phi^{(\theta)}$,
\begin{align*}
-(x^{(\theta)}_t)^{{\xpower}-1} + \gamma \phi^{(\theta)}(x^{(\theta)}_t) &= (p^{(\theta)}_t)'= (\phi^{(\theta)})'(x^{(\theta)}_t) \phi^{(\theta)}(x^{(\theta)}_t)^{{\ppower}-1}, \text{ so}\\
(\phi^{(\theta)})'(z) &=\phi^{(\theta)}(z)^{1-{\ppower}} (\gamma \phi^{(\theta)}(z) - z^{{\xpower}-1}).
\end{align*}
Because the orbits are disjoint for different $\theta$'s, we have $\phi^{(\theta_1)}(z) < \phi^{(\theta_2)}(z)$ for $0 < \theta_1 < \theta_2 \le \eta$ and $z \in (0, \theta_1)$. If $(z_0, -\phi^{(\theta)}(z_0)) \in \mathcal{R}_A$ for some $z_0 \in (0, \theta)$, then by Lemma \ref{Lemma-trapping-set}, $(z, -\phi^{(\theta)}(z)) \in \mathcal{R}_A$ for every $z \in (0,z_0]$.
So
\[
\Theta = \l\{\theta \in (0, \eta]; \, \liminf_{z \to 0} z^{1-{\xpower}} \phi^{(\theta)}(z) < \infty \r\}.
\]
If $0 < \theta_1 < \theta_2$ and $\theta_2 \in \Theta$, then $\theta_1 \in \Theta$ too, since $\phi^{(\theta_1)}(z) < \phi^{(\theta_2)}(z)$ for $z \in (0, \theta_1)$.
If $A > \frac{1}{\gamma}$ and $\theta < \xi(A)$, then $(x^{(\theta)}_t, p^{(\theta)}_t) \in \mathcal{R}_A$ for $t>0$, so $(z,-\phi^{(\theta)}(z)) \in \mathcal{R}_A$ for $z \in (0, \theta)$.
So $(0, \xi(A)) \subseteq \Theta$ for every $A > \frac{1}{\gamma}$. Let
	\[
	F(z) := \gamma^{-1} ({\ppower}-1)^{-1} \phi^{(\theta)}(z)^{{\ppower}-1} - z,
	\]
	then $\lim_{z \to 0} F(z) = 0$, and
	\[
	F'(z) = \frac{(\phi^{(\theta)})'(z)}{\gamma \phi^{(\theta)}(z)^{2-{\ppower}}} - 1 = - \gamma^{-1} (z^{1-{\xpower}} \phi^{(\theta)}(z))^{-1},
	\]
	so $\lim_{z \to 0} F'(z) = 0$, because $\lim_{z \to 0} z^{1-{\xpower}} \phi^{(\theta)}(z) = \infty$, since $\theta \notin \Theta$.
	Then for every $\epsilon>0$ there is a $\delta>0$ such that $F$ is $\epsilon$-Lipschitz in $(0,\delta)$, and then $|F(z)| \le \epsilon z$ for $z \in (0,\delta)$.
	So $\lim_{z \to 0} \frac{F(z)}{z} = 0$.
\end{proof}

\begin{lem} \label{Lemma-Theta}
	$\Theta = (0, \eta)$.
\end{lem}
\begin{proof}
	Suppose indirectly that $\eta \in \Theta$.
	Then there is an $A > \frac{1}{\gamma}$ and a $z \in (0, \eta)$ such that $(z, -\phi_{\eta}(z)) \in \mathcal{R}_A$.
	Then for $\epsilon>0$ small enough we have $(z, -\phi_{\eta}(z)-\epsilon) \in \mathcal{R}_A$ too.
	Let $(x,p)$ be the solution with $(x_0,p_0) = (z, -\phi_{\eta}(z)-\epsilon)$.
	By Lemma \ref{Lemma-p-is-zero-many-times}, there is a $T < 0$ such that $p(T) = 0$ and $p_t < 0$ for $t \in (T,0]$.
	Then $x'_t < 0$ for $t \in (T,0]$.
	Since this orbit cannot cross $\{(u, -\phi^{(\theta)}(u)); \, u \in (0, \eta) \}$, we must have $x(T) > \eta$.
	However $(x_0,p_0) \in \mathcal{R}_A$, so $(x_t,p_t) \in \mathcal{R}_A$ for every $t \ge 0$ by Lemma \ref{Lemma-trapping-set}.
	So $(x(T),0)$ is the last intersection of the solution $(x,p)$ with the $x$-axis, hence $x(T)$ is not in the image of $\eta$, so $\eta < x(T) \le \eta$.
	This contradiction proves that $\eta \notin \Theta$.
	
	Now suppose indirectly that there is a $\theta \in (0,\eta) \setminus \Theta$.
	Let us write $\phi = \phi_{\eta}$ and $\psi = \phi^{(\theta)}$ for simplicity.
	We have
	\[
	\psi'(z) = \psi(z)^{1-{\ppower}} (\gamma \psi(z) - z^{{\xpower}-1}) > 0
	\]
	for $z>0$ close enough to $0$, because $\lim_{z \to 0} \frac{\psi(z)}{z^{{\xpower}-1}} = \infty$, since $\eta \notin \Theta$.
	So $\psi$ has an inverse function $\psi^{-1}$ near $0$.
	So we can define a function $G(z) := \psi^{-1}(\phi(z))$ for $z \in (0,c)$, for some $c>0$.
	We have $\psi(z) < \phi(z)$ for every $z \in (0,\theta)$, so $G(z) > z$ for $z \in (0,c)$.
	Then
	\begin{align*}
		G'(z) &= \psi'(G(z))^{-1} \phi'(z) = \frac{\phi(z)^{1-{\ppower}}  (\gamma \phi(z) - z^{{\xpower}-1})}{\psi(G(z))^{1-{\ppower}}(\gamma \psi(G(z)) - G(z)^{{\xpower}-1})} \\
		&= \frac{\gamma \phi(z) - z^{{\xpower}-1}}{\gamma \phi(z) - G(z)^{{\xpower}-1}}.
	\end{align*}
	Let $h(z) = G(z)-z$ for $z \in (0,c)$, then $h(z)>0$, $\lim_{z \to 0} h(z) = 0$, and
	\[
	h'(z) = \frac{(z+h(z))^{{\xpower}-1} - z^{{\xpower}-1}}{\gamma \phi(z) - G(z)^{{\xpower}-1}} = z^{{\xpower}-1} \frac{(1+\frac{h(z)}{z})^{{\xpower}-1} - 1}{\gamma \phi(z) - G(z)^{{\xpower}-1}}.
	\]
	If $z \to 0$, then $\phi(z)^{{\ppower}-1} \sim \psi(z)^{{\ppower}-1} \sim \gamma ({\ppower}-1) z$ by Lemma \ref{Lemma-non-Theta-asymptotic}.
	Since $G(z) \to 0$, we also have $\gamma ({\ppower}-1) z \sim \phi(z)^{{\ppower}-1} = \psi(G(z))^{{\ppower}-1} \sim \gamma ({\ppower}-1) G(z)$.
	So $\lim_{z \to 0} \frac{G(z)}{z} = 1$ and $\lim_{z \to 0} \frac{h(z)}{z} = 0$.
	Then $\phi(z)/G(z)^{{\xpower}-1} \sim (\gamma ({\ppower}-1))^{\frac{1}{{\ppower}-1}} z^{\frac{1}{{\ppower}-1}-({\xpower}-1)}$, so $\lim_{z \to 0} \gamma \phi(z)/G(z)^{{\xpower}-1} = \infty$, because $\frac{1}{{\ppower}-1} < {\xpower}-1$.
	Therefore
	\[
	h'(z) \sim z^{{\xpower}-1} \frac{(1+\frac{h(z)}{z})^{{\xpower}-1} - 1}{\gamma \phi(z)} \sim z^{{\xpower}-1} ({\xpower}-1) \gamma^{-1} \frac{h(z)}{z} \frac{1}{(\gamma ({\ppower}-1) z)^{\frac{1}{{\ppower}-1}}} = C z^{\lambda-1} h(z),
	\]
	where $C = ({\xpower}-1) \gamma^{-1} (\gamma ({\ppower}-1))^{-\frac{1}{{\ppower}-1}} > 0$ and $\lambda = {\xpower}-1-\frac{1}{{\ppower}-1} > 0$ are constants.
	
	We know that $\phi(z) \sim (\gamma ({\ppower}-1))^{\frac{1}{{\ppower}-1}} z^{\frac{1}{{\ppower}-1}} = 1$.
	Note that $\frac{1}{{\ppower}-1} < {\xpower}-1$, so $\lim_{z \to 0} \gamma \phi(z)/G(z)^{{\xpower}-1} = \infty$.
	We also have $\lim_{z \to 0} \frac{h(z)}{z} = 0$ and $(1+\frac{h(z)}{z})^{{\xpower}-1} - 1 \sim ({\xpower}-1) \frac{h(z)}{z}$.
	So $h'(z) \sim \frac{\gamma^{-1} ({\xpower}-1)}{(\gamma ({\ppower}-1))^{\frac{1}{{\ppower}-1}}} z^{{\xpower} - 2 - \frac{1}{{\ppower}-1}} h(z)$.
	Thus $h'(z) \sim C z^{\lambda-1} h(z)$, where $C>0$ and $\lambda = {\xpower} - 2 - \frac{1}{{\ppower}-1} + 1 > 0$, because ${\xpower}{\ppower} - {\xpower} - {\ppower} > 0$.
	So $\log(h(z))' = \frac{h'(z)}{h(z)} \sim C z^{\lambda-1} = (\frac{C}{\lambda} z^{\lambda})'$.
	Applying L'H\^ospital's rule, we get
	\[
	1 = \lim_{z \to 0} \frac{\log(h(z))'}{(\frac{C}{\lambda} z^{\lambda})'} = \lim_{z \to 0} \frac{\log(h(z))}{\frac{C}{\lambda} z^{\lambda}} = -\infty.
	\]
	This contradiction proves that $\Theta = (0,\eta)$.
\end{proof}

Now we are ready to prove our lower bound.

\thmspeedofconv*

\begin{proof}
First let $\theta \in (0, \eta)$. Then $\theta \in \Theta$ by Lemma \ref{Lemma-Theta}, so there is an $A > \frac{1}{\gamma}$ and a $t_0 \in \R$ such that $(x^{(\theta)}_t, p^{(\theta)}_t) \in \mathcal{R}_A$ for $t \ge t_0$.
Then $x^{(\theta)}_t>0$ and $(x^{(\theta)}_t)' = -|p^{(\theta)}_t|^{{\ppower}-1} \ge -A^{{\ppower}-1} (x^{(\theta)}_t)^{({\xpower}-1)({\ppower}-1)}$ for $t \ge t_0$. Here $({\xpower}-1)({\ppower}-1)>1$, so
\[
((x^{(\theta)}_t)^{-({\xpower}{\ppower}-{\xpower}-{\ppower})})' = -({\xpower}{\ppower}-{\xpower}-{\ppower}) (x^{(\theta)}_t)^{-({\xpower}-1)({\ppower}-1)} (x^{(\theta)}_t)' \le ({\xpower}{\ppower}-{\xpower}-{\ppower}) A^{{\ppower}-1}
\]
for $t \ge t_0$. Let $K:= ({\xpower}{\ppower}-{\xpower}-{\ppower}) A^{{\ppower}-1}$, then $(x^{(\theta)}_t)^{-({\xpower}{\ppower}-{\xpower}-{\ppower})} \le K t + L$ for $t \ge t_0$ and some $L \in \R$.
So $(x^{(\theta)}_t)^{-1} = O(t^{\frac{1}{{\xpower}{\ppower}-{\xpower}-{\ppower}}})$ when $t \to \infty$, therefore the convergence is not linear.
	
By Lemma \ref{Lemma-non-Theta-asymptotic}, we have $p^{(\eta)}_t \sim - (\gamma ({\ppower}-1) x^{(\eta)}_t)^{\frac{1}{{\ppower}-1}}$ when $t \to \infty$.
So $(x^{(\eta)})'_t = -|p^{(\eta)}_t|^{{\ppower}-1} \sim -\gamma ({\ppower}-1) x^{(\eta)}_t$, hence $(\log(x^{(\eta)}_t))' \sim -\gamma ({\ppower}-1) = (-\gamma ({\ppower}-1) t)'$, when $t \to \infty$. So by L'H\^ospital's rule, $\log(x^{(\eta)}_t) \sim -\gamma ({\ppower}-1) t$, thus $|x^{(\eta)}_t| = O(e^{-\alpha t})$ when $t \to \infty$, for every $\alpha < \gamma ({\ppower}-1)$.
Then also $|p^{(\eta)}_t| = O(e^{-\beta t})$ when $t \to \infty$, for every $\beta < \gamma$.
So the convergence is linear.

So up to time translation there are only two solutions, $(x^{(\eta)}, p^{(\eta)})$ and $(-x^{(\eta)}, -p^{(\eta)})$, where the convergence to $(0,0)$ is linear.
\end{proof}

\section{Proofs of convergence for discrete systems}\label{sec:proofsdiscretizations}

\subsection{Implicit Method}
Firstly, we show the well-definedness of the implicit scheme.
\impwelldeflemma*
\begin{proof}
The proof is based on Theorem 26.3 of \cite{rockafellar1970convex}. We start by introducing some concepts from \cite{rockafellar1970convex} that are useful for dealing with convex functions on $\R^n$ taking values in $[-\infty, \infty]$. We say that $g:\R^n\to [-\infty,\infty]$ is \emph{convex} if the epigraph of g, $\{(x,\mu): \mu\ge g(x), x\in \R^n, \mu\in [-\infty,\infty]\}$ is convex. A convex function $g:\R^n\to [-\infty,\infty]$ is called \emph{proper convex} if $g(x)\neq -\infty$ for every $x\in \R^n$, and there is at least one $x\in \R^n$ where $g(x)<\infty$. We say that $g:\R^n\to [-\infty,\infty]$ is \emph{lower-semicontinuous} if $\lim_{x_i\to x}g(x_i)\ge g(x)$ for every sequence $x_i\to x$ such that $\lim_{x_i\to x}g(x_i)$ exists.
The \emph{relative interior} of a set $S\subset \R^n$, denoted by $\mathrm{ri}\, S$, is the interior of the set within its affine closure. We define the \emph{essential domain} of a function $g:\R^n\to [-\infty,\infty]$, denoted by $\domain g$, as the set of points $x\in \R^n$ where $g(x)$ is finite. We call a proper convex function $g:\R^n\to [-\infty,\infty]$ \emph{essentially smooth} if it satisfies the following 3 conditions for $C=\interior(\domain g)$:
\begin{enumerate}
\item[(a)] $C$ is non-empty
\item[(b)] $g$ is differentiable throughout $C$
\item[(c)]\label{conditionc} $\lim_{i\to \infty} \|\grad g(x_i)\|=+\infty$ whenever $x_1,x_2,\ldots$ is a sequence in $C$ converging to a boundary point of $C$.
\end{enumerate}
Let $\partial g(x)$ denote the subdifferential of $g$ at $x$ (which is the set of subgradients of $g$ at $x$), and denote $\domain \partial g:= \{x|\partial g(x)\ne \emptyset\}$. We say that a proper convex function $g:\R^n\to [-\infty,\infty]$ is \emph{essentially strictly convex} if $g$ is strictly convex for every convex subset of $\domain \partial g$.

By assumption \ref{ass:cont:kconvex}, $\K$ is differentiable everywhere, and it is strictly convex, hence it is both essentially smooth and essentially strictly convex (since its domain is $\domain \K=\R^n$). Moreover, since $\K$ is a proper convex function, and it is lower semicontinuous everywhere (hence closed, see page 52 of \cite{rockafellar1970convex}), it follows from Theorem 12.2 of \cite{rockafellar1970convex} that $(\K^*)^*=\K$. Therefore, by Theorem 26.3 of \cite{rockafellar1970convex}, it follows that $\K^*$ is both essentially strictly convex and essentially smooth. Since $f$ is convex and differentiable everywhere in $\R^n$, based on the definitions and the assumption $\epsilon,\gamma\in (0,\infty)$, it is straightforward to show that
\[F(x):=\epsilon \K^*(\tfrac{x-x_i}{\epsilon}) + \epsilon \delta f(x) - \delta \inner{p_i}{x}\]
is also essentially strictly convex and  essentially smooth. Now we are going to show that its infimum is reached at a unique point in $\R^n$. First, using the convexity of $f$, it follows that $f(x)\ge f(x_i)+\inner{\grad f(x_i)}{x-x_i}$, hence
\begin{align*}
F(x)&\ge \epsilon \K^*(\tfrac{x-x_i}{\epsilon}) + \epsilon \delta \inner{\grad f(x_i)}{x-x_i} - \delta \inner{p_i}{x-x_i}-\delta \inner{p_i}{x_i}+\epsilon\delta f(x_i)
\intertext{using the definition \eqref{eq:convconj} of the convex conjugate $\K^*$}
&\ge \inner{p}{x-x_i} -\K(p)+ \epsilon \delta \inner{\grad f(x_i)}{x-x_i} - \delta \inner{p_i}{x-x_i}-\delta \inner{p_i}{x_i}+\epsilon\delta f(x_i)\\
&=\inner{p+\epsilon \delta \grad f(x_i)-\delta p_i}{x-x_i}-\K(p)-\delta \inner{p_i}{x_i}+\epsilon\delta f(x_i),
\end{align*}
for any $p\in \R^n$. By setting $p=\frac{x-x_i}{\|x-x_i\|}-\l(\epsilon \delta \grad f(x_i)-\delta p_i\r)$ for $\|x-x_i\|>0$, and $p=-\l(\epsilon \delta \grad f(x_i)-\delta p_i\r)$ for $\|x-x_i\|=0$, using the continuity and finiteness of $\K$, it follows that $F(x)\ge \|x-x_i\|-c$ for some $c<\infty$ depending only on $\epsilon$, $\delta$, $x_i$ and $p_i$. Together with the lower semicontinuity of $F$, this implies that there exists at least one $y\in \R^n$ such that  $F(y)=\inf_{x\in \R^n} F(x)$, and $-\infty<\inf_{x\in \R^n} F(x)<\infty$.

It remains to show that this $y$ is unique. First, we are going to show that it falls within the interior of the domain of $F$. Let $C:=\interior (\domain F)$, then using the essential smoothness of $F$, it follows that $C\subseteq \domain F\subseteq \mathrm{cl} C$ is a non-empty convex set ($\mathrm{cl}$ refers to closure). If $y$ would fall on the boundary of $C$, then by Lemma 26.2 of \cite{rockafellar1970convex}, $F(y)$ could not be equal to the infimum of $F$. Hence every such $y$ falls within $C$. By
 Theorem 23.4 of \cite{rockafellar1970convex}, $\mathrm{ri} (\domain F)\subseteq \domain \partial F\subseteq \domain F$. Since $C$ is a non-empty open convex set, $C=\interior \domain F=\mathrm{ri} (\domain F)$, therefore from the definition of essential strict convexity, it follows that $F$ is strictly convex on $C$. This means that there the infimum $\inf_{x\in \R^n} F(x)$ is achieved at a unique $y\in \R^n$, thus \eqref{eq:implicit2} is well-defined.

Finally, we show the equivalence with \eqref{eq:implicit}. First, note that using the fact that $\K$ is essentially smooth and essentially convex, it follows from Theorem 26.5 of \cite{rockafellar1970convex} that $\grad (\K^*)(x)=(\grad \K)^{-1}(x)$ for every $x\in \interior (\domain \K^*)$. Since $F$ is differentiable in the open set $C=\interior (\domain F)$, and the infimum of $F$ is taken at some $y\in C$, it follows that $\grad F(y)=0$. From the fact that $f(x)$ and $\inner{p_i}{x}$ are differentiable for every $x\in \R^n$, it follows that  for every point $z\in C$, $\frac{z-x_i}{\epsilon}\in \interior (\domain \K^*)$. Thus in particular, using the definition $x_{i+1}=y$, we have
\begin{align*}
\grad (\K^*)\l(\frac{x_{i+1}-x_i}{\epsilon}\r)+\epsilon\delta \grad f(x_{i+1})-\delta p_i=0,
\intertext{ which can be rewritten equivalently using the second line of \eqref{eq:implicit2} as}
\grad (\K^*)\l(\frac{x_{i+1}-x_i}{\epsilon}\r)=p_{i+1}.
\end{align*}
Using the expression $\grad (\K^*)(x)=(\grad \K)^{-1}(x)$ for
$x=\frac{x_{i+1}-x_i}{\epsilon}\in \interior (\domain \K^*)$, we obtain that $(\grad \K)^{-1}\l(\frac{x_{i+1}-x_i}{\epsilon}\r)=p_{i+1}$, and hence the first line of \eqref{eq:implicit} follows by applying $\grad \K$ on both sides. The second line follows by rearrangement of the second line of \eqref{eq:implicit2}.
\end{proof}

The following two lemmas are preliminary results that will be used in deriving convergence results for both this scheme and the two explicit schemes in the next sections.
\begin{lem}\label{lem:prelimbetai}
Given $f$, $\K$, $\gamma$, $\alpha$, $\Ca$, and $\Cfk$ satisfying Assumptions \ref{ass:cont} and \ref{ass:imp}, and a sequence of points $x_i, p_i\in \R^d$ for $i\ge 0$,  we define $\Ha_i:=f(x_i)-f(\xmin)+\K(p_i)$ . Then the equation
	\begin{equation}
		\label{eq:Vkalphadef}v=\Ha_i+\frac{\Ca}{2} \alpha(2v) \inner{x_i-\xmin}{p_i}.
	\end{equation}
has a unique solution in the interval $v \in [\Ha_i/2, 3\Ha_i/2]$, which we denote by $\Ly_i$. In addition, let
\begin{equation}\label{eq:betaidef}
\beta_i:=\frac{\Ca}{2} \alpha(2\Ly_i),
\end{equation}
then	 $\Ly_i=\Ha_i+\beta_{i}\inner{x_{i}-\xmin}{p_{i}}$ and the differences $\Ly_{i+1}-\Ly_{i}$ can be expressed as
	\begin{align}
		\Ly_{i+1}-\Ly_i&=\Ha_{i+1}-\Ha_i+\beta_{i+1}\inner{x_{i+1}-\xmin}{p_{i+1}}-\beta_i\inner{x_{i}-\xmin}{p_{i}}\\
		&\label{eq:Lydiff1}=\Ha_{i+1}-\Ha_i+\beta_i(\inner{x_{i+1}-\xmin}{p_{i+1}}-
		\inner{x_{i}-\xmin}{p_{i}})+
		(\beta_{i+1}-\beta_i) \inner{x_{i+1}-\xmin}{p_{i+1}}\\
		&\label{eq:Lydiff2}=\Ha_{i+1}-\Ha_i+\beta_{i+1}(\inner{x_{i+1}-\xmin}{p_{i+1}}-
		\inner{x_{i}-\xmin}{p_{i}})+
		(\beta_{i+1}-\beta_i) \inner{x_{i}-\xmin}{p_{i}}.
	\end{align}
\end{lem}
\begin{proof}
Similarly to \eqref{eq:Kperalphaf2sided}, we have by Lemma \ref{lem:betachoicealpha},
	\begin{align}
		\label{eq:Kperalphaf2sideddisc} |\inner{x_i-\xmin}{p_i}|\le  \K(p_i)/\alpha(\K(p_i)) + \f(x_i) - \f(\xmin) \le \frac{\Ha_i}{\alpha(\K(p_i))}.
	\end{align}
	For every $i\ge 0$, we define $\Ly_i$ as the unique solution $v \in [\Ha_i/2, 3\Ha_i/2]$ of the equation
	\begin{equation}
		\label{eq:Vkalphadef}v=\Ha_i+\frac{\Ca}{2} \alpha(2v) \inner{x_i-\xmin}{p_i}.
	\end{equation}
	The existence and uniqueness of this solution was shown in the proof of Theorem \ref{lem:continuouslyap}. The fact that $\Ly_i=\Ha_i+\beta_{i}\inner{x_{i}-\xmin}{p_{i}}$ immediately follows from equation \eqref{eq:Vkalphadef}, and \eqref{eq:Lydiff1}-\eqref{eq:Lydiff2} follow by rearrangement.
\end{proof}

\begin{lem}\label{lem:prelimconv}
Under the same assumptions and definitions as in Lemma \ref{lem:prelimbetai}, if in addition we assume that for some constants $C_1, C_2\ge 0$, for every $i\ge 0$,
\begin{align}
\nonumber\Ly_{i+1}-\Ly_i&\le -\epsilon (\gamma - \beta_{i+1}-C_1 \epsilon)\K(p_{i+1})-\epsilon \gamma \beta_{i+1}\inner{x_{i+1}-\xmin}{p_{i+1}}-\epsilon \beta_{i+1} (f(x_{i+1})-f(\xmin))\\
\label{eq:Vdiffass1a}&\quad\, +C_2 \epsilon^2 \beta_{i+1} \Ly_{i+1}+(\beta_{i+1}-\beta_i)\inner{x_i-\xmin}{p_i}, \text{ and}\\
\nonumber\Ly_{i+1}-\Ly_i&\le -\epsilon (\gamma - \beta_{i}-C_1 \epsilon)\K(p_{i+1})-\epsilon \gamma \beta_{i}\inner{x_{i+1}-\xmin}{p_{i+1}}-\epsilon \beta_{i} (f(x_{i+1})-f(\xmin))\\
\label{eq:Vdiffass1b}&\quad\, +C_2 \epsilon^2 \beta_{i} \Ly_{i+1}+(\beta_{i+1}-\beta_i)\inner{x_{i+1}-\xmin}{p_{i+1}},
\end{align}
then for every $0<\epsilon\le \min\l(\frac{1-\gamma}{C_2},  \frac{\gamma^2 (1-\gamma)}{4 C_1}\r)$, for every $i\ge 0$, we have
\begin{equation}\label{eq:Vdiffass1bnd}
\Ly_{i+1}\le \l[1+\epsilon  \beta_i(1-\gamma-\epsilon C_2)/2 \r]^{-1} \Ly_i.
\end{equation}
Similarly, if in addition to the assumptions of Lemma \ref{lem:prelimbetai}, we assume that for some constants $C_1, C_2\ge 0$, for every $i\ge 0$,
\begin{align}
\nonumber\Ly_{i+1}-\Ly_i&\le -\epsilon (\gamma - \beta_{i+1}-C_1 \epsilon)\K(p_{i})-\epsilon \gamma \beta_{i+1}\inner{x_{i}-\xmin}{p_{i}}-\epsilon \beta_{i+1} (f(x_{i})-f(\xmin))\\
\label{eq:Vdiffass2a}&\quad\, +C_2 \epsilon^2 \beta_{i+1} \Ly_{i}+(\beta_{i+1}-\beta_i)\inner{x_i-\xmin}{p_i}, \text{ and}\\
\nonumber\Ly_{i+1}-\Ly_i&\le -\epsilon (\gamma - \beta_{i}-C_1 \epsilon)\K(p_{i})-\epsilon \gamma \beta_{i}\inner{x_{i}-\xmin}{p_{i}}-\epsilon \beta_{i} (f(x_{i})-f(\xmin))\\
\label{eq:Vdiffass2b}&\quad\, +C_2 \epsilon^2 \beta_{i} \Ly_{i}+(\beta_{i+1}-\beta_i)\inner{x_{i+1}-\xmin}{p_{i+1}},
\end{align}
then for every $0<\epsilon\le \min\l(\frac{1-\gamma}{C_2},  \frac{\gamma^2 (1-\gamma)}{4 C_1}\r)$, we have
\begin{equation}\label{eq:Vdiffass2bnd}
\Ly_{i+1}\le \l[1-  \epsilon \beta_i(1-\gamma-\epsilon C_2)/2 \r] \Ly_i.
\end{equation}
\end{lem}
\begin{proof}
First suppose that assumptions \eqref{eq:Vdiffass1a} and \eqref{eq:Vdiffass1b} hold. Using \eqref{eq:goodenoughbeta2} of Lemma \ref{lem:Vderivativebound} with $\alpha= \alpha(2\Ly_{i+1})$ and $\beta=\beta_{i+1}$, it follows that for $\epsilon\le \frac{\gamma^2 (1-\gamma)}{4 C_1}$,
	\begin{align*}
		&-\epsilon (\gamma - \beta_{i+1}-C_1 \epsilon) \K(p_{i+1}) - \epsilon \beta_{i+1} (\f(x_{i+1})-\f(\xmin)) -\epsilon\beta_{i+1} \gamma \inner{x_{i+1}-\xmin}{p_{i+1}}\\
		&\le -\epsilon\beta_{i+1}(1-\gamma) \Ly_{i+1},
	\end{align*}
	and by combining the terms in \eqref{eq:Vdiffass1a}, we have
	\begin{align}\label{eq:Lydiffbnd1}
		&\Ly_{i+1}-\Ly_i\le -\epsilon\beta_{i+1}[1-\gamma- \epsilon C_2] \Ly_{i+1}+(\beta_{i+1}-\beta_i) \inner{x_{i}-\xmin}{p_{i}}.
	\end{align}
Now we are going to prove that $\Ly_{i+1}\le \Ly_i$ under the assumptions of the lemma. We argue by contradiction, suppose that $\Ly_{i+1}>\Ly_i$. Then by the non-increasing property of $\alpha$, and the definition $\beta_i=\frac{\Ca}{2} \alpha(2\Ly_i)$, we have $\beta_{i+1}\le \beta_i$. Using the convexity of $\alpha$, we have $\alpha(y)-\alpha(x)\le \alpha'(y)(y-x)$ for any $x,y\ge 0$, hence we obtain that
\[|\beta_{i+1}-\beta_i|=\beta_i-\beta_{i+1}=\frac{\Ca}{2} (\alpha(2\Ly_i)-\alpha(2\Ly_{i+1}))\le \Ca (\Ly_{i+1}-\Ly_i)(-\alpha'(2\Ly_{i})),\]
and by \eqref{eq:Kperalphaf2sideddisc} and assumption \ref{ass:cont:alpha} we have
\begin{align*}
	\l|\beta_{i+1}-\beta_i\r|  \l|\inner{x_{i}-\xmin}{p_{i}}\r|\le \Ca (\Ly_{i+1}-\Ly_i)(-\alpha'(2\Ly_{i}))  \frac{2\Ly_i}{\alpha(2\Ly_i)}< \Ly_{i+1}-\Ly_i.
\end{align*}
Combining this with \eqref{eq:Lydiffbnd1} we obtain that $\Ly_{i+1}-\Ly_i<\Ly_{i+1}-\Ly_i$, which is a contradiction. Hence we have shown that $\Ly_{i+1}\le \Ly_i$, which implies that $\beta_{i+1}\ge \beta_i$.

Using \eqref{eq:goodenoughbeta2} of Lemma \ref{lem:Vderivativebound}  with $\alpha= \alpha(2\Ly_{i})$ and $\beta=\beta_{i}$, it follows that for $0< \epsilon\le \frac{\gamma^2 (1-\gamma)}{4 C_1}$,
\begin{align*}
	&-\epsilon (\gamma - \beta_{i}-C_1 \epsilon) \K(p_{i+1}) - \epsilon \beta_{i} (\f(x_{i+1})-\f(\xmin)) -\epsilon\beta_{i} \gamma \inner{x_{i+1}-\xmin}{p_{i+1}}\\
	&\le -\epsilon\beta_{i}(1-\gamma) \Ly_{i+1},
\end{align*}
and hence by substituting this to \eqref{eq:Vdiffass1a}, it follows that
\begin{align}\label{eq:Lydiffbnd2}
	&\Ly_{i+1}-\Ly_i\le -\epsilon\beta_{i}[1-\gamma-\epsilon C_2] \Ly_{i+1}+(\beta_{i+1}-\beta_i) \inner{x_{i+1}-\xmin}{p_{i+1}}.
\end{align}
Now using the convexity of $\alpha$, and the fact that $\beta_{i+1}\ge \beta_i$, we have
\[|\beta_{i+1}-\beta_i|=\beta_{i+1}-\beta_i =\frac{\Ca}{2} (\alpha(2\Ly_{i+1})-\alpha(2\Ly_{i}))\le \Ca (\Ly_{i}-\Ly_{i+1})(-\alpha'(2\Ly_{i+1})),\]
and by \eqref{eq:Kperalphaf2sideddisc} and assumption \ref{ass:cont:alpha} we have
\begin{align*}
	\l|\beta_{i+1}-\beta_i\r|  \l|\inner{x_{i+1}-\xmin}{p_{i+1}}\r|\le \Ca (\Ly_{i}-\Ly_{i+1})(-\alpha'(2\Ly_{i+1}))  \frac{2\Ly_{i+1}}{\alpha(2\Ly_{i+1})}< \Ly_{i}-\Ly_{i+1}.
\end{align*}
By combining this with \eqref{eq:Lydiff2}, we obtain that
\begin{align*}
	&\Ly_{i+1}-\Ly_i\le -\frac{\epsilon\beta_{i}}{2}[1-\gamma-\epsilon C_2] \Ly_{i+1},
\end{align*}
and the first claim of the lemma follows by rearrangement and monotonicity.

The proof of the second claim based on assumptions \eqref{eq:Vdiffass2a} and \eqref{eq:Vdiffass2b} is as follows. As previously, in the first step, we show that $\Ly_{i+1}\le \Ly_i$ by contradiction. Suppose that $\Ly_{i+1}>\Ly_{i}$, then $\beta_{i+1}\le \beta_i$. Using \eqref{eq:goodenoughbeta2} of Lemma \ref{lem:Vderivativebound} with $\alpha= \alpha(2\Ly_{i})$ and $\beta=\beta_{i+1}\le \beta_i\le \frac{\alpha \gamma}{2}$, it follows that for $\epsilon\le \frac{\gamma^2 (1-\gamma)}{4 C_1}$,
	\begin{align*}
		&-\epsilon (\gamma - \beta_{i+1}-C_1 \epsilon) \K(p_{i+1}) - \epsilon \beta_{i+1} (\f(x_{i+1})-\f(\xmin)) -\epsilon\beta_{i+1} \gamma \inner{x_{i+1}-\xmin}{p_{i+1}}\\
		&\le -\epsilon\beta_{i+1}(1-\gamma) \Ly_{i+1},
	\end{align*}
	and by combining the terms in \eqref{eq:Vdiffass2a}, we have
	\begin{align}\label{eq:Lydiffbnd1}
		&\Ly_{i+1}-\Ly_i\le -\epsilon\beta_{i+1}[1-\gamma- \epsilon C_2] \Ly_{i}+(\beta_{i+1}-\beta_i) \inner{x_{i}-\xmin}{p_{i}}.
	\end{align}
The rest of the proof follows the same steps as for assumptions \eqref{eq:Vdiffass1a} and \eqref{eq:Vdiffass1b}, hence it is omitted.
\end{proof}

Now we are ready to prove the main result of this section.

\implemma*
\begin{proof}
	We follow the notations of Lemma \ref{lem:prelimbetai}, and the proof is based on Lemma \ref{lem:prelimconv}. By rearrangement of the \eqref{eq:implicit}, we have
	\begin{equation}
	\label{eq:implicit3} \begin{aligned}
	x_{i+1} - x_i &= \epsilon\nabla \K(p_{i+1})\\
	p_{i+1} - p_i &= - \gamma\epsilon p_{i+1} - \epsilon\nabla f(x_{i+1})
	\end{aligned}
	\end{equation}
	For the Hamiltonian terms, by the convexity of $f$ and $\K$, we have
	\begin{align}
		\nonumber	\Ha_{i+1} - \Ha_i &\leq \inner{\grad \K(p_{i+1})}{p_{i+1}-p_{i}}+\inner{\grad \f(x_{i+1})}{x_{i+1}-x_{i}}\\
		&= \inner{\grad \K(p_{i+1})}{- \gamma\epsilon p_{i+1} - \epsilon\nabla f(x_{i+1})} + \epsilon \inner{\grad f(x_{i+1})}{\grad \K(p_{i+1})} \\
		\label{eq:Habnd1}	&= - \gamma \epsilon \inner{\grad \K(p_{i+1})}{p_{i+1}}
	\end{align}
	For the inner product terms, we have
	\begin{align*}
		&\inner{x_{i+1}-\xmin}{p_{i+1}}-\inner{x_i-\xmin}{p_i}=\inner{x_{i+1}-\xmin}{p_{i+1}}-\inner{x_{i+1}-\xmin-(x_{i+1}-x_i)}{p_{i+1}-(p_{i+1}-p_i)}\\
		&=\inner{x_{i+1}-\xmin}{p_{i+1}}-\inner{x_{i+1}-\xmin-\epsilon \nabla \K(p_{i+1})}{p_{i+1}+\epsilon\gamma p_{i+1} + \epsilon \nabla f(x_{i+1})}\\
		&=(\epsilon+\gamma \epsilon^2) \inner{p_{i+1}}{\nabla \K(p_{i+1})}-\epsilon\inner{x_{i+1}-\xmin}{\nabla f(x_{i+1})}-\epsilon\gamma\inner{x_{i+1}-\xmin}
		{p_{i+1}}+\epsilon^2 \inner{\nabla \K(p_{i+1})}{\nabla f(x_{i+1})},
	\end{align*}
	and by assumption \ref{ass:imp:innerproductfk} we have
	\begin{align}\label{eq:gradKgradfimp}
		\inner{\nabla \K(p_{i+1})}{\nabla f(x_{i+1})}&\le   \Cfk \Ha_{i+1}\le 2  \Cfk \Ly_{i+1},
	\end{align}
	and hence
	\begin{align}\label{eq:xpdiffimp}
	\inner{x_{i+1}-\xmin}{p_{i+1}}-\inner{x_i-\xmin}{p_i}&\le (\epsilon+\gamma \epsilon^2) \inner{p_{i+1}}{\nabla \K(p_{i+1})}
	-\epsilon\inner{x_{i+1}-\xmin}{\nabla f(x_{i+1})}\\
	\nonumber&-\epsilon\gamma\inner{x_{i+1}-\xmin}
	{p_{i+1}}+2\epsilon^2 \Cfk \Ly_{i+1}.
	\end{align}
	By assumption \ref{ass:cont:alpha} on $\Ca$ we have $\beta_{i+1} \le \frac{\gamma}{2}$, and using the condition $\epsilon<\frac{1-\gamma}{2(\Cfk+\gamma)}$ of the lemma, we have
	\begin{equation}\label{eq:gammabetaeps}\gamma-\beta_{i+1}-\epsilon\gamma \beta_{i+1}\ge \gamma-\frac{\gamma}{2}-\frac{ (1-\gamma)}{2\gamma} \gamma \frac{\gamma}{2}>0.\end{equation}
	By \eqref{eq:Lydiff2}, \eqref{eq:Habnd1}, \eqref{eq:xpdiffimp}, we have
		\begin{align*}
		\Ly_{i+1}-\Ly_i &\le
		- \epsilon(\gamma-\beta_{i+1}-\epsilon\gamma\beta_{i+1})  \inner{\grad \K(p_{i+1})}{p_{i+1}}-\epsilon\beta_{i+1}\inner{x_{i+1}-\xmin}{\nabla f(x_{i+1})}\\&
		-\epsilon\gamma \beta_{i+1} \inner{x_{i+1}-\xmin}
		{p_{i+1}}
		+2\epsilon^2 \Cfk \beta_{i+1} \Ly_{i+1}
		+(\beta_{i+1}-\beta_i) \inner{x_{i}-\xmin}{p_{i}}
		\intertext{ using the convexity of $f$ and $\K$, and inequality \eqref{eq:gammabetaeps}}
		&\le	
		 -\epsilon (\gamma - \beta_{i+1} - \epsilon \gamma \beta_{i+1}) \K(p_{i+1})- \epsilon \beta_{i+1} (\f(x_{i+1})-\f(\xmin))  -\epsilon \gamma \beta_{i+1} \inner{x_{i+1}-\xmin}{p_{i+1}} \\
		&+2\epsilon^2 \Cfk \beta_{i+1} \Ly_{i+1}+(\beta_{i+1}-\beta_i) \inner{x_{i}-\xmin}{p_{i}}.
		\end{align*}
Using the fact that $\beta_{i+1}\le \frac{C_{\alpha,\gamma}}{2}\le \frac{\gamma}{2}$, it follows that \eqref{eq:Vdiffass1a} holds with $C_1=\frac{\gamma^2}{2}$ and $C_2=2 \Cfk$.
	
By \eqref{eq:Lydiff1}, \eqref{eq:Habnd1}, \eqref{eq:xpdiffimp}, it follows that
\begin{align*}
	&\Ly_{i+1}-\Ly_i\le
		\Ha_{i+1}-\Ha_i+\beta_i(\inner{x_{i+1}-\xmin}{p_{i+1}}-\inner{x_{i}-\xmin}{p_{i}})+
	(\beta_{i+1}-\beta_i) \inner{x_{i+1}-\xmin}{p_{i+1}}\\
	&\le - \gamma \epsilon \inner{\grad \K(p_{i+1})}{p_{i+1}}+
	\beta_i(\inner{x_{i+1}-\xmin}{p_{i+1}}-\inner{x_{i}-\xmin}{p_{i}})+
	(\beta_{i+1}-\beta_i) \inner{x_{i+1}-\xmin}{p_{i+1}}\\		
	&\le - \gamma \epsilon \inner{\grad \K(p_{i+1})}{p_{i+1}}+(\beta_{i+1}-\beta_i) \inner{x_{i+1}-\xmin}{p_{i+1}}+(\beta_i\epsilon+\beta_i\gamma \epsilon^2) \inner{p_{i+1}}{\nabla \K(p_{i+1})}\\
	&-\beta_i\epsilon\inner{x_{i+1}-\xmin}{\nabla f(x_{i+1})}
	-\beta_i\epsilon\gamma\inner{x_{i+1}-\xmin}
	{p_{i+1}}+2\beta_i \epsilon^2 \Cfk \Ly_{i+1}	
\intertext{using the convexity of $f$ and $\K$, and inequality \eqref{eq:gammabetaeps}}
	&\le -\epsilon (\gamma - \beta_{i}-\epsilon \gamma \beta_i) \K(p_{i+1}) -\epsilon\beta_{i} \gamma \inner{x_{i+1}-\xmin}{p_{i+1}} - \epsilon \beta_{i} (\f(x_{i+1})-\f(\xmin))\\
	&+2\epsilon^2 \Cfk \beta_{i} \Ly_{i+1}+(\beta_{i+1}-\beta_i) \inner{x_{i+1}-\xmin}{p_{i+1}},	
\end{align*}
implying that \eqref{eq:Vdiffass1b} holds with $C_1=\frac{\gamma^2}{2}$ and $C_2=2 \Cfk$. The claim of the Lemma now follows from Lemma \ref{lem:prelimconv}.
\end{proof}

\subsection{First Explicit Method}
The following lemma is a preliminary result that will be useful for proving our convergence bounds for this discretization.

\begin{lem}\label{lem:prelimsemiA}
Given $f$, $\K$, $\gamma$, $\alpha$, $\Ca$, $\Cfk$, $\Ck$, $\Dfk$ satisfying assumptions \ref{ass:cont}, \ref{ass:imp}, and \ref{ass:semiA}, and $0<\epsilon\le \frac{\Ca}{10\Cfk+5\gamma \Ck}$, the iterates \eqref{eq:semiA} satisfy that for every $i\ge 0$,
\begin{equation}\label{eq:gradfdiffcond}
\inner{\grad f(x_{i+1})-\grad f(x_{i})}{x_{i+1}-x_i}\le 3\epsilon^2 \Dfk \min(\alpha(3\Ha_i),\alpha(3\Ha_{i+1}))  \Ha_{i+1}.
\end{equation}
\end{lem}
\begin{proof}
Let $x_{i+1}^{(t)}:=x_{i+1}-t \epsilon \grad \K(p_{i+1})$ and $\Ha_{i+1}^{(t)}:=\Ha(x_{i+1}^{(t)},p_{i+1})$. Using the assumptions that $f$ is 2 times continuously differentiable, and assumption \ref{ass:semiA:gKHfsemi}, we have
\begin{align}
\nonumber&\inner{\grad f(x_{i+1})-\grad f(x_{i})}{x_{i+1}-x_i}=\int_{t=0}^{1} \inner{x_{i+1}-x_i}{\grad^2 f(x_{i+1}-t(x_{i+1}-x_i)) (x_{i+1}-x_i) }dt\\
\label{eq:semiAprelimmain}&=\epsilon^2 \int_{t=0}^{1}\inner{\grad \K(p_{i+1})}{\grad^2 f(x_{i+1}^{(t)}) \grad \K(p_{i+1})}dt\le \epsilon^2 \Dfk \int_{t=0}^{1}
\alpha(3\Ha_{i+1}^{(t)}) \Ha_{i+1}^{(t)}dt,
\end{align}
where in the last step we have used the fundamental theorem of calculus, which is applicable since $\inner{\grad \K(p_{i+1})}{\grad^2 f(x_{i+1}^{(t)}) \grad \K(p_{i+1})}$ is piecewise continuous by assumption \ref{ass:semiA:ftwicediff}.
We are going to show the following inequalities based on the assumptions of the Lemma,
\begin{align}
\label{eq:prelimsemiAin1}\Ha_{i+1}^{(t)}&\le \frac{1}{1-\epsilon \Cfk} \Ha_{i+1},\\
\label{eq:prelimsemiAin2}\alpha(3\Ha_{i+1}^{(t)})&\le \alpha(3\Ha_{i+1})\cdot \frac{1-\epsilon \Cfk}{1-\epsilon \Cfk(1+1/\Ca)},\\
\label{eq:prelimsemiAin3}\alpha(3\Ha_{i+1}^{(t)})&\le \alpha(3\Ha_{i})\cdot \frac{1-\epsilon(2 \Cfk +\gamma \Ck)}{1-\epsilon [\Cfk(2+3/\Ca)+\gamma \Ck (1+1/\Ca)]}.
\end{align}
The claim of the lemma follows directly by combining these 3 inequalities with \eqref{eq:semiAprelimmain} and using the assumptions on $\epsilon$.

First, by convexity and assumption \ref{ass:imp:innerproductfk}, we have
\begin{align*}
	&\Ha_{i+1}^{(t)}-\Ha_{i+1}=f(x_{i+1}^{(t)})-f(x_{i+1})\le -\inner{\grad f(x_{i+1}^{(t)})}{t\epsilon \grad \K(p_{i+1})}\le t\epsilon \Cfk \Ha_{i+1}^{(t)},
\end{align*}
and \eqref{eq:prelimsemiAin1} follows by rearrangement. In the other direction, by convexity and assumption \ref{ass:imp:innerproductfk}, we have \begin{align*}&\Ha_{i+1}-\Ha_{i+1}^{(t)}=f(x_{i+1})-f(x_{i+1}^{(t)})\le \inner{\grad f(x_{i+1})}{t\epsilon \grad \K(p_{i+1})}\le t\epsilon \Cfk \Ha_{i+1},
	\intertext{so by rearrangement, it follows that}
	&\Ha_{i+1}-\Ha_{i+1}^{(t)}\le \frac{t\epsilon \Cfk}{1-t\epsilon \Cfk} \Ha_{i+1}^{(t)}.
\end{align*}
Using this, and the convexity of $\alpha$, and Assumption \ref{ass:cont:alpha}, we have
\begin{align*}
	&\alpha(3\Ha_{i+1}^{(t)})-\alpha(3\Ha_{i+1})\le -3\alpha'(3\Ha_{i+1}^{(t)}) (\Ha_{i+1}-\Ha_{i+1}^{(t)})\le  -\alpha'(3\Ha_{i+1}^{(t)})3\Ha_{i+1}^{(t)}\frac{t\epsilon \Cfk}{1-t\epsilon \Cfk}\\
	&\le \frac{1}{\Ca}  \frac{t\epsilon \Cfk}{1-t\epsilon \Cfk} \alpha(3\Ha_{i+1}^{(t)}),
\end{align*}
and \eqref{eq:prelimsemiAin2} follows by rearrangement. Finally, using the convexity of $f$ and $\K$, we have
\begin{align*}
\Ha_i-\Ha_{i+1}^{(t)}&=\K(p_i)-\K(p_{i+1})+f(x_i)-f(x_{i+1})\\
&\le \inner{\grad \K(p_i)}{\frac{\gamma \epsilon}{1+\gamma \epsilon} p_i+\frac{\epsilon}{1+\gamma \epsilon}f(x_i)}+\inner{\grad f(x_i)}{-\epsilon(1-t) \grad \K(p_{i+1})}\\
\intertext{using Assumptions \ref{ass:imp:innerproductfk} and \ref{ass:semiA:gKpsem}}
&\le \gamma \epsilon \Ck \K(p_i)+\epsilon \Cfk \Ha_i+\epsilon \Cfk (\K(p_{i+1})+f(x_i)-f(\xmin))\\
&\le \epsilon[(2\Cfk+\gamma \Ck)\Ha_i+\Cfk \Ha_{i+1}^{(t)}].
\end{align*}
By rearrangement, this implies that
\[\Ha_i-\Ha_{i+1}^{(t)}\le \frac{\epsilon(3 \Cfk+\gamma \Ck)}{1-(2\Cfk+\gamma \Ck)\epsilon}\cdot \Ha_{i+1}^{(t)}.\]
Using this, the convexity of $\alpha$, and Assumption \ref{ass:cont:alpha}, we have
\begin{align*}
	&\alpha(3\Ha_{i+1}^{(t)})-\alpha(3\Ha_{i})\le -3\alpha'(3\Ha_{i+1}^{(t)}) (\Ha_{i}-\Ha_{i+1}^{(t)})\le  -\alpha'(3\Ha_{i+1}^{(t)})3\Ha_{i+1}^{(t)}\cdot
	\frac{\epsilon(3 \Cfk+\gamma \Ck)}{1-(2\Cfk+\gamma \Ck)\epsilon}\\
	&\le \frac{1}{\Ca} \cdot \frac{\epsilon(3 \Cfk+\gamma \Ck)}{1-(2\Cfk+\gamma \Ck)\epsilon}\cdot \alpha(3\Ha_{i+1}^{(t)}),
\end{align*}
and \eqref{eq:prelimsemiAin3} follows by rearrangement.
\end{proof}

Now we are ready to prove our convergence bound for this discretization.
\semiAlemma*

\begin{proof}
	We follow the notations of Lemma \ref{lem:prelimbetai}, and the proof is based on Lemma \ref{lem:prelimconv}.
	For the Hamiltonian terms, by the convexity of $f$ and $\K$, we have
	\begin{align}
		\nonumber	\Ha_{i+1} - \Ha_i &=f(x_{i+1})-f(x_i)+\K(p_{i+1})-\K(p_i)\\		
		\nonumber&\leq \inner{\grad f(x_{i+1})}{x_{i+1}-x_i}+
		\inner{\grad \K(p_{i+1})}{p_{i+1}-p_i}\\	
		\nonumber &= \inner{\grad f(x_{i})}{x_{i+1}-x_i}+
		\inner{\grad \K(p_{i+1})}{p_{i+1}-p_i}+\inner{\grad f(x_{i+1})-\grad f(x_{i})}{x_{i+1}-x_i}\\
		\nonumber &=  \epsilon\inner{\grad f(x_{i})}{\grad \K(p_{i+1})}-\epsilon\inner{\grad \K(p_{i+1})}{\grad f(x_{i})+\gamma p_{i+1}}
		+\inner{\grad f(x_{i+1})-\grad f(x_{i})}{x_{i+1}-x_i}\\		
		\label{eq:Habnd1A}	&= - \gamma \epsilon \inner{\grad \K(p_{i+1})}{p_{i+1}}	+\inner{\grad f(x_{i+1})-\grad f(x_{i})}{x_{i+1}-x_i}
	\end{align}
	for any $\epsilon>0$.
	Note that by convexity and assumption \ref{ass:imp:innerproductfk}, we have
	\begin{align*}-f(x_i)&=-f(x_{i+1})+f(x_{i+1})-f(x_i)\le -f(x_{i+1})+\epsilon \inner{\grad f(x_{i+1})}{\grad \K(p_{i+1})}\\
		&\le -f(x_{i+1})+\epsilon \Cfk \Ha_{i+1}\le -f(x_{i+1})+2\epsilon \Cfk \Ly_{i+1}.
	\end{align*}
	For the inner product terms, using the above inequality and convexity, we have
	\begin{align}
\nonumber		&\inner{x_{i+1}-\xmin}{p_{i+1}}-\inner{x_i-\xmin}{p_i}\\
\nonumber		&=
		\inner{x_{i+1}-\xmin}{p_{i+1}}-\inner{x_{i}-\xmin}{p_{i+1}}+\inner{x_{i}-\xmin}{p_{i+1}}-\inner{x_i-\xmin}{p_i}\\
\nonumber		&=\epsilon \inner{\grad \K(p_{i+1})}{p_{i+1}}-\epsilon \inner{x_i-\xmin}{\grad f(x_i)}-\gamma \epsilon \inner{x_i-\xmin}{p_{i+1}}{p_{i+1}}\\
\nonumber		&=	(\epsilon+\gamma \epsilon^2) \inner{\grad \K(p_{i+1})}{p_{i+1}}-\epsilon \inner{x_i-\xmin}{\grad f(x_i)}-\gamma \epsilon \inner{x_{i+1}}{p_{i+1}}
		\\
\nonumber		&\le (\epsilon+\gamma \epsilon^2) \inner{\grad \K(p_{i+1})}{p_{i+1}}-\epsilon (f(x_i)-f(\xmin))-\gamma \epsilon \inner{x_{i+1}}{p_{i+1}}\\
\label{eq:xpdiffexp}		&\le (\epsilon+\gamma \epsilon^2) \inner{\grad \K(p_{i+1})}{p_{i+1}}-\epsilon (f(x_{i+1})-f(\xmin))-\gamma \epsilon \inner{x_{i+1}}{p_{i+1}}+2\epsilon^2 \Cfk \Ly_{i+1}.
	\end{align}
	Since $\Ca\le \gamma$, it follows that $\beta_{i+1}= \frac{\Ca}{2}\alpha(2\Ly_{i+1})\le \frac{\gamma}{2}$, and using the assumption on $\epsilon$, we have
	\begin{equation}\label{eq:gammabetaeps2}\gamma-\beta_{i+1}-\epsilon\gamma \beta_{i+1}\ge \gamma-\frac{\gamma}{2}-\frac{ 1-\gamma}{2\gamma} \gamma \frac{\gamma}{2}>0.
	\end{equation}
	By \eqref{eq:Lydiff2}, \eqref{eq:Habnd1A}, and \eqref{eq:xpdiffexp}, it follows that
	\begin{align*}
		&\Ly_{i+1}-\Ly_i= \Ha_{i+1}-\Ha_i+\beta_{i+1}(\inner{x_{i+1}-\xmin}{p_{i+1}}-
		\inner{x_{i}-\xmin}{p_{i}})+
		(\beta_{i+1}-\beta_i) \inner{x_{i}-\xmin}{p_{i}}\\
		&\le
		- \gamma \epsilon \inner{\grad \K(p_{i+1})}{p_{i+1}}	+\inner{\grad f(x_{i+1})-\grad f(x_{i})}{x_{i+1}-x_i}+
		(\beta_{i+1}-\beta_i) \inner{x_{i}-\xmin}{p_{i}}\\
		&+\beta_{i+1}\l( (\epsilon+\gamma \epsilon^2) \inner{\grad \K(p_{i+1})}{p_{i+1}}-\epsilon (f(x_{i+1})-f(\xmin))-\gamma \epsilon \inner{x_{i+1}}{p_{i+1}}+2\epsilon^2 \Cfk \Ly_{i+1}\r)
		\\
		&\le
		- \epsilon(\gamma-\beta_{i+1}-\epsilon\gamma\beta_{i+1})  \inner{\grad \K(p_{i+1})}{p_{i+1}}-\epsilon\beta_{i+1}f(x_{i+1})-\epsilon\gamma \beta_{i+1} \inner{x_{i+1}-\xmin}
		{p_{i+1}}\\
		&+2\epsilon^2 \beta_{i+1}  \Cfk \Ly_{i+1}+\inner{\grad f(x_{i+1})-\grad f(x_{i})}{x_{i+1}-x_i}+(\beta_{i+1}-\beta_i) \inner{x_{i}-\xmin}{p_{i}}
		\intertext{which can be further bounded using \eqref{eq:gammabetaeps2}, the convexity of $\K$, and Lemma \ref{lem:prelimsemiA} as}
		&\le	
		-\epsilon (\gamma - \beta_{i+1}-\epsilon\frac{\gamma^2}{2}) \K(p_{i+1}) -\epsilon\beta_{i+1} \gamma \inner{x_{i+1}-\xmin}{p_{i+1}} - \epsilon \beta_{i+1} (\f(x_{i+1})-\f(\xmin))\\
		&+2\epsilon^2(\Cfk+6 \Dfk/\Ca) \beta_{i+1} \Ly_{i+1}+(\beta_{i+1}-\beta_i) \inner{x_{i}-\xmin}{p_{i}},
	\end{align*}
	implying that \eqref{eq:Vdiffass1a} holds with $C_1=\frac{\gamma^2}{2}$ and $C_2=2(\Cfk+6 \Dfk/\Ca)$.

	Since $2\Ly_i\le 3\Ha_i$, and by applying Lemma \ref{lem:prelimsemiA} it follows that
	\begin{equation}\label{eq:gradfdiffbndA}
	\inner{\grad f(x_{i+1})-\grad f(x_{i})}{x_{i+1}-x_i}\le 6\epsilon^2 \Dfk \frac{\beta_i}{\Ca} \Ha_{i+1}\le 12\epsilon^2 \frac{\Dfk}{\Ca} \beta_i \Ly_{i+1}.
	\end{equation}
	By \eqref{eq:Lydiff1}, \eqref{eq:Habnd1A}, \eqref{eq:xpdiffexp}, and assumption \ref{ass:imp:innerproductfk}, we have
	\begin{align*}
		&\Ly_{i+1}-\Ly_i=\Ha_{i+1}-\Ha_i+\beta_i(\inner{x_{i+1}-\xmin}{p_{i+1}}-
		\inner{x_{i}-\xmin}{p_{i}})+
		(\beta_{i+1}-\beta_i) \inner{x_{i+1}-\xmin}{p_{i+1}}\\
		&\le - \gamma \epsilon \inner{\grad \K(p_{i+1})}{p_{i+1}}	+\inner{\grad f(x_{i+1})-\grad f(x_{i})}{x_{i+1}-x_i}+
		(\beta_{i+1}-\beta_i) \inner{x_{i+1}-\xmin}{p_{i+1}}	\\	
		&+(\beta_{i}\epsilon+\gamma \beta_i\epsilon^2) \inner{\grad \K(p_{i+1})}{p_{i+1}}-\beta_i\epsilon (f(x_{i+1})-f(\xmin))-\gamma \beta_i \epsilon \inner{x_{i+1}}{p_{i+1}}+2\epsilon^2\beta_i \Cfk \Ly_{i+1}		\\
		\intertext{using \eqref{eq:gradfdiffbndA} and the convexity of $f$ and $\K$}
		&\le -\epsilon \l(\gamma - \beta_{i}-\epsilon \frac{\gamma^2}{2}\r) \K(p_{i+1}) -\epsilon\beta_{i} \gamma \inner{x_{i+1}-\xmin}{p_{i+1}} - \epsilon \beta_{i} (\f(x_{i+1})-\f(\xmin))\\
		&+2\epsilon^2(\Cfk+6\Dfk/\Ca) \beta_{i} \Ly_{i+1}+(\beta_{i+1}-\beta_i) \inner{x_{i+1}-\xmin}{p_{i+1}},
	\end{align*}
	implying that \eqref{eq:Vdiffass1b} holds with $C_1=\frac{\gamma^2}{2}$ and $C_2=2(\Cfk+6 \Dfk/\Ca)$. The claim of the lemma now follows by Lemma \ref{lem:prelimconv}.
\end{proof}

\subsection{Second Explicit Method}
The following preliminary result will be used in the proof of the convergence bound.
\begin{lem}\label{lem:prelimsemiB}
	Given $f$, $\K$, $\gamma$, $\alpha$, $\Ca$, $\Cfk$, $\Ck$, $\Dfk$ satisfying assumptions \ref{ass:cont}, \ref{ass:imp}, and \ref{ass:semiB}, and $0<\epsilon\le \min\left(\frac{\Ca}{6 (5\Cfk+2\gamma \Ck)+12\gamma\Ca },\sqrt{\frac{1}{6\gamma^2 \Dk F_{\K}}} \right)$, the iterates \eqref{eq:semiB} satisfy that for every $i\ge 0$,
	\begin{equation}\label{eq:gradfdiffcondB}
	\inner{\grad \K(p_{i+1})-\grad \K(p_{i})}{p_{i+1}-p_i}\le \epsilon^2 C \min(\alpha(3\Ha_i),\alpha(3\Ha_{i+1}))  \Ha_{i}+\epsilon^2 D \K (p_{i}),
	\end{equation}
	where
	\begin{equation}\label{eq:gradfdiffcondBdefCD}
	C=3\Dfk, \qquad  D=2\gamma^2 D_k (1+E_k).
	\end{equation}
\end{lem}
\begin{proof}
For $0\le t\le 1$, let
\begin{align}
p_{i}^{(t)}&:=p_i+t(p_{i+1}-p_i)=(1-\epsilon \gamma t)p_i-\epsilon t \grad f(x_{i+1}),\\
\Ha_{i}^{(t)}&:=\Ha\l(x_{i+1},p_{i}^{(t)}\r)=f(x_{i+1})-f(\xmin)+\K\l(p_{i}^{(t)}\r),\\
\Pa_{i,i+1}&:=\inner{\grad \K(p_{i+1})-\grad \K(p_{i})}{p_{i+1}-p_i}.
\end{align}
Note that by rearrangement we have $p_i=\l(p_{i}^{(t)}+\epsilon t \grad f(x_{i+1})\r)/(1-\epsilon \gamma t)$, and hence
\begin{equation}\label{eq:pip1pidiff}
p_{i+1}-p_i=\frac{p_i^{(t)}-p_i}{t}=\frac{-\epsilon \gamma p_{i}^{(t)}-\epsilon \grad f(x_{i+1}) }{1-\epsilon\gamma t}.
\end{equation}
Using assumption \ref{ass:semiB:Ktwicediff}, it follows that
$\inner{p_{i+1}-p_i}{\grad^2 \K\l(p_{i}^{(t)}\r) (p_{i+1}-p_i)}$ is piecewise continuous, hence by the fundamental theorem of calculus, we have
\begin{align}
\nonumber &\Pa_{i,i+1}= \int_{t=0}^{1}
\inner{p_{i+1}-p_i}{\grad^2 \K\l(p_{i}^{(t)}\r) (p_{i+1}-p_i)} dt\\
\nonumber&=\frac{1}{(1-\epsilon\gamma t)^2}\int_{t=0}^{1}
\inner{\epsilon \gamma p_{i}^{(t)}+\epsilon \grad f(x_{i+1})}{\grad^2 \K\l(p_{i}^{(t)}\r) (\epsilon \gamma p_{i}^{(t)}+\epsilon\grad f(x_{i+1})} dt\\
\label{eq:prelimsemiBmain}
&\le \frac{2\epsilon^2 \gamma^2}{(1-\epsilon\gamma)^2}\int_{t=0}^{1}
\inner{p_{i}^{(t)}}{\grad^2 \K\l(p_{i}^{(t)}\r) p_{i}^{(t)}}+
\frac{2\epsilon^2}{(1-\epsilon\gamma)^2}\int_{t=0}^{1}
\inner{\grad f(x_{i+1})}{\grad^2 \K\l(p_{i}^{(t)}\r) \grad f(x_{i+1})} dt
\end{align}
For the first integral, using Assumptions \ref{ass:semiB:gHpsem}, the convexity of $\K$, and then \ref{ass:semiB:weird}, we have
\begin{align}
\nonumber\int_{t=0}^{1}\inner{p_{i}^{(t)}}{\grad^2 \K\l(p_{i}^{(t)}\r) p_{i}^{(t)}} dt&\le D_{\K}\int_{t=0}^{1} \K\l(p_{i}^{(t)}\r) dt\le \frac{D_{\K}}{2}\l(\K(p_{i})+\K(p_{i+1})\r)\\
\label{eq:prelimsemiBint1}
&\le \frac{D_{\K}}{2}((1+E_{\K})\K(p_i)+F_{\K} \Pa_{i,i+1} )
\end{align}
For the second integral, using Assumption \ref{ass:semiB:gfHKsemi}, we have
\begin{align}\label{eq:prelimsemiBint2}
\int_{t=0}^{1}\inner{\grad f(x_{i+1})}{\grad^2 \K\l(p_{i}^{(t)}\r) \grad f(x_{i+1})}dt\le \Dfk \int_{t=0}^{1}\Ha_{i}^{(t)} \alpha(3\Ha_i^{(t)})dt.
\end{align}
We are going to show the following 3 inequalities based on the assumptions of the Lemma.
\begin{align}
	\label{eq:prelimsemiBin1}\Ha_{i}^{(t)}&\le \frac{1-\epsilon \gamma}{1-\epsilon(\gamma+2\Cfk)} \cdot \Ha_{i},\\
	\label{eq:prelimsemiBin2}\alpha(3\Ha_{i}^{(t)})&\le \alpha(3\Ha_{i+1})\cdot \frac{1-(\Cfk+\gamma)\epsilon }{1-(\Cfk+\gamma+\Cfk/\Ca)\epsilon},\\
	\label{eq:prelimsemiBin3}\alpha(3\Ha_{i}^{(t)})&\le \alpha(3\Ha_{i})\cdot \frac{1-\epsilon(2 \Cfk +\gamma \Ck)}{1-\epsilon [\Cfk(2+3/\Ca)+\gamma \Ck (1+1/\Ca)]}.
\end{align}
The claim of the lemma follows from substituting these bounds into \eqref{eq:prelimsemiBint2}, and then substituting the bounds \eqref{eq:prelimsemiBint1} and \eqref{eq:prelimsemiBint2} into \eqref{eq:prelimsemiBmain} and rearranging.

First, by the convexity of $f$ and Assumption \ref{ass:imp:innerproductfk}, we have
\begin{align}
	\nonumber f(x_{i+1})-f(x_i)&\le \epsilon \inner{\grad f(x_{i+1})}{\grad \K (p_i)}\le \epsilon \Cfk (f(x_{i+1})-f(\xmin)+\K(p_i))\\
	\nonumber &=\epsilon \Cfk ((f(x_{i+1}-f(x_i)) +\Ha_i )
	\intertext{so by rearrangement it follows that}
	\label{eq:semiBfdiff1}f(x_{i+1})-f(x_i)&\le \frac{\epsilon \Cfk}{1-\epsilon \Cfk}\cdot \Ha_i,
	\intertext{ and similiarly}
	\label{eq:semiBfdiff2}
	f(x_{i})-f(x_{i+1})&\le -\epsilon \inner{\grad f(x_i)}{\grad \K(p_i)}\le \epsilon \Cfk \Ha_i.
\end{align}
Using \eqref{eq:semiBfdiff1}, and the convexity of $\K$, we have
\begin{align*}
	\Ha_{i}^{(t)}-\Ha_{i}&=f(x_{i+1})-f(x_i)+\K\l(p_{i}^{(t)}\r)-\K(p_{i})\\
	&\le \frac{\epsilon \Cfk}{1-\epsilon \Cfk}\cdot \Ha_i+\inner{\grad \K\l(p_{i}^{(t)}\r)}{t(p_{i+1}-p_i)}
	\intertext{now using \eqref{eq:pip1pidiff}, and then Assumption \ref{ass:imp:innerproductfk},}
	&\le \frac{\epsilon \Cfk}{1-\epsilon \Cfk}\cdot \Ha_i-\epsilon t\inner{\grad \K\l(p_{i}^{(t)}\r)}
	{\frac{\gamma p_{i}^{(t)}+ \grad f(x_{i+1}) }{1-\epsilon\gamma t}}\\
	&\le \frac{\epsilon \Cfk}{1-\epsilon \Cfk}\cdot \Ha_i+\frac{\epsilon \Cfk}{1-\epsilon \gamma}\Ha_{i}^{(t)},
\end{align*}
and inequality \eqref{eq:prelimsemiBin1} follows by rearrangement.

By the convexity of $\K$, and using  \eqref{eq:pip1pidiff} for $t=1$, we have
\begin{align*}
	&\Ha_{i+1}-\Ha_{i}^{(t)}=\K(p_{i+1})-\K\l(p_{i}^{(t)}\r)\le \inner{\grad \K(p_{i+1})}{p_{i+1}-p_{i}^{(t)}}\\
	&=\inner{\grad \K(p_{i+1})}{(1-t) (p_{i+1}-p_{i})}=
	-(1-t)\inner{\grad \K(p_{i+1})}{\frac{\epsilon \gamma}{1-\epsilon \gamma}p_{i+1}+\frac{\epsilon}{1-\epsilon \gamma} \grad f(x_{i+1})}
	\intertext{ using Assumption \ref{ass:imp:innerproductfk},}
	&\le \frac{\epsilon \Cfk}{1-\gamma \epsilon} \cdot \Ha_{i+1},
\end{align*}
so by rearrangement,
\begin{equation}
\Ha_{i+1}-\Ha_{i}^{(t)}\le \frac{\epsilon \Cfk}{1-(\Cfk+\gamma)\epsilon}\Ha_{i}^{(t)}.
\end{equation}
Using this, the convexity of $\alpha$, and Assumption \ref{ass:cont:alpha}, we have
\begin{align*}
	&\alpha(3\Ha_{i}^{(t)})-\alpha(3\Ha_{i+1})\le -3\alpha'(3\Ha_{i}^{(t)}) (\Ha_{i+1}-\Ha_{i}^{(t)})\le  -\alpha'(3\Ha_{i}^{(t)})3\Ha_{i}^{(t)}\cdot
	\frac{\epsilon \Cfk}{1-(\Cfk+\gamma)\epsilon}\\
	&\le \frac{1}{\Ca} \cdot \frac{\epsilon \Cfk}{1-(\Cfk+\gamma)\epsilon}\cdot \alpha(3\Ha_{i}^{(t)}),
\end{align*}
and \eqref{eq:prelimsemiBin2} follows by rearrangement. Finally, using inequality \eqref{eq:semiBfdiff2}, we have
\begin{align*}
	\Ha_{i}-\Ha_{i}^{(t)}&=f(x_{i})-f(x_{i+1})+\K(p_{i})-\K\l(p_{i}^{(t)}\r)\\
	&\le \epsilon \Cfk \Ha_i+\inner{\grad \K(p_{i})}{-t(p_{i+1}-p_i)}\\
	&\le \epsilon \Cfk \Ha_i+\epsilon t\inner{\grad \K(p_{i})}{\gamma p_i +\grad f(x_{i+1})}
	\intertext{now using Assumptions \ref{ass:imp:innerproductfk} and \ref{ass:semiB:gKpsem},}
	&\le \epsilon (\Cfk \Ha_i+\gamma \Ck \K(p_i)+\Cfk \K(p_i)+\Cfk(f(x_{i+1})-f(\xmin))\\
	&\le \epsilon ((2\Cfk+\gamma \Ck)\Ha_i+\Cfk \Ha_{i}^{(t)}),
\end{align*}
and by rearrangement this implies that
\begin{equation}
\Ha_{i}-\Ha_{i}^{(t)}\le \frac{(3 \Cfk+\gamma \Ck)\epsilon}{1-(2\Cfk+\gamma \Ck)\epsilon}\cdot \Ha_{i}^{(t)}.
\end{equation}
Using this, the convexity of $\alpha$, and Assumption \ref{ass:cont:alpha}, we have
\begin{align*}
	&\alpha(3\Ha_{i}^{(t)})-\alpha(3\Ha_{i})\le -3\alpha'(3\Ha_{i}^{(t)}) (\Ha_{i}-\Ha_{i}^{(t)})\le  -\alpha'(3\Ha_{i}^{(t)})3\Ha_{i}^{(t)}\cdot
	\frac{\epsilon(3 \Cfk+\gamma \Ck)}{1-(2\Cfk+\gamma \Ck)\epsilon}\\
	&\le \frac{1}{\Ca} \cdot \frac{\epsilon(3 \Cfk+\gamma \Ck)}{1-(2\Cfk+\gamma \Ck)\epsilon}\cdot \alpha(3\Ha_{i}^{(t)}),
\end{align*}
and \eqref{eq:prelimsemiBin3} follows by rearrangement.
\end{proof}

Now we are ready to prove the convergence bound.

\semiBlemma*

\begin{proof}
We follow the notations of Lemma \ref{lem:prelimbetai}, and the proof is based on Lemma \ref{lem:prelimconv}. For the Hamiltonian terms, by the convexity of $f$ and $\K$, we have
\begin{align}
\nonumber	\Ha_{i+1} - \Ha_i &=f(x_{i+1})-f(x_i)+\K(p_{i+1})-\K(p_i)\\		
\nonumber&\leq \inner{\grad f(x_{i+1})}{x_{i+1}-x_i}+
\inner{\grad \K(p_{i+1})}{p_{i+1}-p_i}\\	
\nonumber &= \inner{\grad f(x_{i+1})}{x_{i+1}-x_i}+
\inner{\grad \K(p_{i})}{p_{i+1}-p_i}+\inner{\grad \K(p_{i+1})-\grad \K(p_{i})}{p_{i+1}-p_i}\\
\nonumber &=  \epsilon\inner{\grad f(x_{i+1})}{\grad \K(p_{i})}-\epsilon\inner{\grad \K(p_{i})}{\grad f(x_{i+1})+\gamma p_{i}}
+ \inner{\grad \K(p_{i+1})-\grad \K(p_{i})}{p_{i+1}-p_i}\\		
\label{eq:Habnd1B}	&= - \gamma \epsilon \inner{\grad \K(p_{i})}{p_{i}}	+\inner{\grad \K(p_{i+1})-\grad \K(p_{i})}{p_{i+1}-p_i}
\end{align}
for any $\epsilon>0$. For the inner product terms, we have
\begin{align}
\nonumber		&\inner{x_{i+1}-\xmin}{p_{i+1}}-\inner{x_i-\xmin}{p_i}\\
\nonumber		&=
\inner{x_{i+1}-\xmin}{p_{i+1}}-\inner{x_{i+1}-\xmin}{p_{i}}+\inner{x_{i+1}-\xmin}{p_{i}}-\inner{x_i-\xmin}{p_i}\\
\nonumber		&=
\inner{x_{i+1}-\xmin}{- \epsilon \gamma p_i -\epsilon \grad \f(x_{i+1})}+\inner{x_{i+1}-x_{i}}{p_i}\\
\nonumber		&=
-\epsilon \inner{\grad \f(x_{i+1})}{x_{i+1}-\xmin}
- \epsilon \gamma \inner{x_{i+1}-\xmin -(x_{i+1}-x_{i})}{p_i}
+(1-\epsilon \gamma)\inner{x_{i+1}-x_{i}}{p_i}\\
\label{eq:xpdiffexpB}		&=
-\epsilon \inner{\grad \f(x_{i+1})}{x_{i+1}-\xmin}
+\epsilon(1-\epsilon \gamma)\inner{\grad \K(p_{i})}{p_i}- \epsilon \gamma \inner{x_{i}-\xmin}{p_i}.
\end{align}
Note that from assumption \ref{ass:imp:innerproductfk} and the convexity of $f$ it follows that
\begin{align}
\nonumber
&-(f(x_{i+1})-f(\xmin))\le -(f(x_{i})-f(\xmin))+ f(x_i)-f(x_{i+1})\\
\label{eq:fxip1xi}
&\le -(f(x_{i})-f(\xmin))+ \inner{\grad f(x_i)}{-\epsilon \grad \K(p_i)}\le -(f(x_{i})-f(\xmin))+\epsilon C_{f,K} \Ha_i.
\end{align}
By combining \eqref{eq:Lydiff2}, \eqref{eq:Habnd1B}, and \eqref{eq:xpdiffexpB}, it follows that
\begin{align*}
\Ly_{i+1}-\Ly_i&= \Ha_{i+1}-\Ha_i+\beta_{i+1}(\inner{x_{i+1}-\xmin}{p_{i+1}}-
\inner{x_{i}-\xmin}{p_{i}})+
(\beta_{i+1}-\beta_i) \inner{x_{i}-\xmin}{p_{i}}\\
&\le - \gamma \epsilon \inner{\grad \K(p_{i})}{p_{i}}	+\inner{\grad \K(p_{i+1})-\grad \K(p_{i})}{p_{i+1}-p_i}+(\beta_{i+1}-\beta_i) \inner{x_{i}-\xmin}{p_{i}}\\
&+\epsilon\beta_{i+1}\l( - \inner{\grad \f(x_{i+1})}{x_{i+1}-\xmin}
+(1-\epsilon \gamma)\inner{\grad \K(p_{i})}{p_i}-  \gamma \inner{x_{i}-\xmin}{p_i}\r)\\
&\le -\epsilon(\gamma -\beta_{i+1} )\inner{\grad \K(p_{i})}{p_{i}}-\epsilon \beta_{i+1}
\inner{\grad f(x_{i+1})}{x_{i+1}-\xmin}
-\epsilon\gamma\beta_{i+1}\inner{x_{i}-\xmin}{p_i}\\
&+\inner{\grad \K(p_{i+1})-\grad \K(p_{i})}{p_{i+1}-p_i}+(\beta_{i+1}-\beta_i) \inner{x_{i}-\xmin}{p_{i}}\\
\intertext{which can be further bounded using $\beta_{i+1}\le \frac{\gamma}{2}$, the convexity of $\K$ and $\f$, and Lemma \ref{lem:prelimsemiB} as}
&\le -\epsilon(\gamma -\beta_{i+1} -\epsilon D)\K(p_{i})-\epsilon \beta_{i+1} (f(x_{i+1})-f(\xmin))-\epsilon\gamma \beta_{i+1}\inner{x_{i}-\xmin}{p_i}\\
&+2\epsilon^2 \beta_{i+1}\cdot C/\Ca\cdot \Ha_{i} +(\beta_{i+1}-\beta_i) \inner{x_{i}-\xmin}{p_{i}}\\
\intertext{and now using \eqref{eq:fxip1xi} and $\Ha_i\le 2\Ly_i$ leads to}
&\le -\epsilon(\gamma -\beta_{i+1} -\epsilon D)\K(p_{i})-\epsilon \beta_{i+1} (f(x_{i})-f(\xmin))-\epsilon\gamma \beta_{i+1}\inner{x_{i}-\xmin}{p_i}\\
&+\epsilon^2 \beta_{i+1}\cdot (4 C/\Ca+2\Cfk)\cdot \Ly_{i} +(\beta_{i+1}-\beta_i) \inner{x_{i}-\xmin}{p_{i}},
\end{align*}
implying that \eqref{eq:Vdiffass2a} holds with $C_1=D$ and $C_2=4C/\Ca+2\Cfk$.

By combining \eqref{eq:Lydiff1}, \eqref{eq:Habnd1B}, and \eqref{eq:xpdiffexpB}, it follows that
\begin{align*}
\Ly_{i+1}-\Ly_i&= \Ha_{i+1}-\Ha_i+\beta_{i}(\inner{x_{i+1}-\xmin}{p_{i+1}}-
\inner{x_{i}-\xmin}{p_{i}})+
(\beta_{i+1}-\beta_i) \inner{x_{i+1}-\xmin}{p_{i+1}}\\
&\le - \gamma \epsilon \inner{\grad \K(p_{i})}{p_{i}}	+\inner{\grad \K(p_{i+1})-\grad \K(p_{i})}{p_{i+1}-p_i}+(\beta_{i+1}-\beta_i) \inner{x_{i+1}-\xmin}{p_{i+1}}\\
&+\epsilon\beta_{i}\l( - \inner{\grad \f(x_{i+1})}{x_{i+1}-\xmin}
+(1-\epsilon \gamma)\inner{\grad \K(p_{i})}{p_i}-  \gamma \inner{x_{i}-\xmin}{p_i}\r)\\
&\le -\epsilon(\gamma -\beta_{i} )\inner{\grad \K(p_{i})}{p_{i}}-\epsilon \beta_{i}
\inner{\grad f(x_{i+1})}{x_{i+1}-\xmin}
-\epsilon\gamma\beta_{i}\inner{x_{i}-\xmin}{p_i}\\
&+\inner{\grad \K(p_{i+1})-\grad \K(p_{i})}{p_{i+1}-p_i}+(\beta_{i+1}-\beta_i) \inner{x_{i+1}-\xmin}{p_{i+1}}\\
\intertext{which can be further bounded using $\beta_{i}\le \frac{\gamma}{2}$, the convexity of $\K$ and $\f$, and Lemma \ref{lem:prelimsemiB} as}
&\le -\epsilon(\gamma -\beta_{i} -\epsilon D)\K(p_{i})-\epsilon \beta_{i} (f(x_{i+1})-f(\xmin))-\epsilon\gamma \beta_{i}\inner{x_{i}-\xmin}{p_i}\\
&+2\epsilon^2 \beta_{i}\cdot C/\Ca\cdot \Ha_{i} +(\beta_{i+1}-\beta_i) \inner{x_{i+1}-\xmin}{p_{i+1}}\\
\intertext{and now using \eqref{eq:fxip1xi} and $\Ha_i\le 2\Ly_i$ leads to}
&\le -\epsilon(\gamma -\beta_{i} -\epsilon D)\K(p_{i})-\epsilon \beta_{i} (f(x_{i})-f(\xmin))-\epsilon\gamma \beta_{i}\inner{x_{i}-\xmin}{p_i}\\
&+\epsilon^2 \beta_{i}\cdot (4 C/\Ca+2\Cfk)\cdot \Ly_{i} +(\beta_{i+1}-\beta_i) \inner{x_{i+1}-\xmin}{p_{i+1}}
\end{align*}
implying that \eqref{eq:Vdiffass2b} holds with $C_1=D$ and $C_2=4C/\Ca+2\Cfk$ (see \eqref{eq:gradfdiffcondBdefCD} for the definition of $C$ and $D$). The claim of the lemma now follows by Lemma \ref{lem:prelimconv}.
\end{proof}

\subsection{Explicit Method on Non-Convex $f$}

\explicitmethodnonconvex*

\begin{proof}
By assumption \ref{ass:semiAsmoothness}
\begin{align*}
\Ha_{i+1} - \Ha_i &\leq \K(p_{i+1}) - \K(p_i) + \epsilon \inner{\grad \f(x_i)}{\grad \K(p_{i+1})} +  \Df\sigma (\epsilon \norm{\grad \K(p_{i+1})})
\intertext{now with convexity of $\K$}
&\leq - \epsilon \gamma \inner{\grad \K(p_{i+1})}{p_{i+1}}  + \Df \sigma (\epsilon \norm{\grad \K(p_{i+1})})\\
&\leq  - \epsilon \gamma \K(p_{i+1}) + \Df \sigma (\epsilon \norm{\grad \K(p_{i+1})})
\intertext{we have $\sigma(\epsilon t) \leq \epsilon^b \sigma(t)$ and $\sigma (\norm{\grad \K(p)}) \leq \Dk\K(p)$ by assumption \ref{ass:semiAsmoothness}}
&\leq  (\epsilon^b \Df\Dk - \epsilon \gamma) \K(p_{i+1})
\end{align*}
If $\epsilon \leq   \sqrt[b-1]{\gamma/\Df\Dk}$, then $\Ha_{i+1} \leq \Ha_i$.
If $\Ha$ is bounded below we get that $x_i, p_i$ is such that $\K(p_i) \to 0$ and thus $p_i \to 0$. Since $\norm{p_{i+1}}_2  + \delta \norm{p_i}_2 \geq \epsilon \delta \norm{\nabla f(x_i)}_2$,
we get $\norm{\nabla f(x_i)} \to 0$.
\end{proof}

\section{Proofs for power kinetic energies}\label{sec:proofskinetic}
The proofs of the results in this section will be based on the following three preliminary lemmas.


\begin{lem}
\label{lem:norms}
Let $\norm{\c}$ be a norm on $\R^d$ and $x \in \R^d \setminus \{0\}$. If $\norm{x}$ is differentiable, then
\begin{equation}
\norm{\grad \norm{x}}_* = 1 \qquad \inner{\grad \norm{x}}{x} = \norm{x},
\end{equation}
and if $\norm{x}$ is twice differentiable, then
\begin{equation}
(\hess \norm{x}) x =  0.
\end{equation}
\end{lem}

\begin{proof}
Let $x \in \R^d \setminus \{0\}$. By the convexity of the norm we have
\begin{align*}
\inner{\grad \norm{x}}{y} - \norm{y} \leq \inner{\grad \norm{x}}{x} - \norm{x}
\end{align*}
for all $y \in \R^d$. Thus
\begin{align*}
\sup_{y \in \R^d} \{\inner{\grad \norm{x}}{y} - \norm{y}\} \leq \inner{\grad \norm{x}}{x} - \norm{x}
\end{align*}
Because the right hand side is finite, we must have $\norm{\grad \norm{x}}_* \leq 1$ and the left hand side equal to 0.
\begin{align*}
0 \leq \inner{\grad \norm{x}}{x} - \norm{x} \leq \norm{\grad \norm{x}}_*\norm{x} - \norm{x} \leq 0
\end{align*}
forces $ \norm{\grad \norm{x}}_* = 1$ and $\inner{\grad \norm{x}}{x} = \norm{x}$. In fact, this argument goes through for non-differentiable $\norm{\c}$, by definition of the subderivative. For twice differentiable norms, take the derivative of $\inner{\grad \norm{x}}{x} = \norm{x}$ to get
\begin{align*}
(\hess \norm{x}) x + \grad \norm{x} = \grad \norm{x}
\end{align*}
and our result follows.
\end{proof}


\begin{lem}
\label{kin:robustonedim}
Given $a \in [1, \infty)$, $A \in [1, \infty)$, and $\varphi_a^A$ in \eqref{eq:kindef}. Define $B = A/(A-1)$, $b = a/(a-1)$. For convenience, define 
\begin{equation}
\varphi(t) = \varphi_a^A(t) \qquad \qquad \phi(t) = \varphi_b^B(t).
\end{equation}
The following hold.
\begin{enumerate}
\item Monotonicity. For $t \in (0, \infty)$, $\varphi'(t) > 0$. If $a = A = 1$, then for $t \in (0, \infty)$, $\varphi''(t) = 0$, otherwise $\varphi''(t) > 0$. This implies that $\varphi$ is strictly increasing on $[0, \infty)$.

\item Subhomogeneity. For all $t,\epsilon \in [0, \infty)$,
\begin{equation}
\label{kin:robustonedim:subhomo} \varphi(\epsilon t) \leq \max\{\epsilon^a, \epsilon^A\} \varphi(t)
\end{equation}
with equality iff $a=A$ or $t = 0$ or $\epsilon = 0$ or $\epsilon = 1$.

\item Strict Convexity. If $a> 1$ or $A > 1$, then $\varphi(t)$ is strictly convex on $[0, \infty)$ with a unique minimum at $0$.

\item Derivatives. For all $t \in (0, \infty)$,
\begin{align}
\label{kin:robustonedim:gradbound} \min\{a,A\}\varphi(t) &\leq t \varphi'(t) \leq \max\{a,A\} \varphi(t), \\
\label{kin:robustonedim:hessbound} (\min\{a,A\} - 1)\varphi'(t) &\leq t \varphi''(t) \leq (\max\{a,A\} - 1)\varphi'(t).
\end{align}
If $a,A \geq 2$, then for all $t, s \in (0, \infty)$,
\begin{equation}
\label{kin:robustonedim:unifgradbound} \varphi(t) \leq s \varphi'(s)  + (\varphi'(t) - \varphi'(s))(t-s)
\end{equation}
\end{enumerate}
\end{lem}

\begin{proof}
First, for $t \in (0, \infty)$, the following identities can be easily verified.
\begin{align}
\label{kin:onedim:grad} t\varphi'(t) &=  (t^a + 1)^{\tfrac{A-a}{a}} t^{a}\\
\label{kin:onedim:hess} t\varphi''(t) &=  \varphi'(t)\l(a-1  + (A-a)\tfrac{t^{a}}{t^a+1}\r)
\end{align}

\begin{enumerate}
\item \emph{Monotonicity}. First, for $t > 0$ we have,
\begin{align}
\varphi'(t) &=  (t^a + 1)^{\tfrac{A-a}{a}} t^{a-1} > 0
\end{align}
For $\varphi'(t)$ for $t > 0$, we have the following with equality iff $a = A = 1$
\begin{align}
	\label{eq:monofderv} \varphi''(t) =  t^{-1}\varphi'(t)\l(a-1  + (A-a)\tfrac{t^{a}}{t^a+1}\r) \geq 0.
\end{align}
Finally, $\varphi(t)$ is continuous at $0$, which gives our result.

\item \emph{Subhomogeneity}. If $a=A$ or $t = 0$ or $\epsilon = 0$ or $\epsilon = 1$, then equality clearly holds. Assume $a \neq A$, $t, \epsilon > 0$, and $\epsilon \neq 1$. Assuming $A > a$, $t^{\tfrac{A-a}{a}}$ is strictly increasing. If $\epsilon < 1$, then
\begin{align*}
\epsilon \varphi'(\epsilon t) &=  \epsilon^a (\epsilon^a t^a + 1)^{\tfrac{A-a}{a}} t^{a-1} <  \epsilon^a (t^a + 1)^{\tfrac{A-a}{a}} t^{a-1}.
\end{align*}
If $\epsilon > 1$, then
\begin{align*}
\epsilon \varphi'(\epsilon t) &=  \epsilon^A ( t^a + \epsilon^{-a})^{\tfrac{A-a}{a}} t^{a-1} <  \epsilon^A (t^a + 1)^{\tfrac{A-a}{a}} t^{a-1}.
\end{align*}
Integrating both sides gives $\varphi(\epsilon t) < \max\{\epsilon^a, \epsilon^A\} \varphi(t)$. The case $A < a$ follows similarly, using the fact that $t^{\tfrac{A-a}{a}}$ is strictly decreasing.

\item \emph{Strict convexity}. First, since $\varphi$ is strictly increasing we get $\varphi(t) > \varphi(0) = 0$, which proves that $0$ is the unique minimizer. Our goal is to prove that for $t,s \in [0, \infty)$ and $\epsilon \in (0, 1)$ such that $t \neq s$,
\begin{align*}
	\varphi(\epsilon t + (1-\epsilon) s) < \epsilon \varphi(t) + (1-\epsilon) \varphi(s)
\end{align*}
First, for $t = 0$ or $s = 0$, this reduces to a condition of the form $\varphi(\epsilon t) < \epsilon \varphi(t)$ for all $t \in [0, \infty)$ and $\epsilon \in (0, 1)$. Considering separately the cases $A=1, a > 1$ and $a = 1, A>1$ and $a,A > 1$, it is easy to see that this follows from the subhomogeneity result \eqref{kin:robustonedim:subhomo}. For $s,t > 0$, our result follows from the positivity of $\varphi''$, \eqref{eq:monofderv}.

\item \emph{Derivatives}. Since,
\begin{align*}
\min\{a,A\} - 1 \leq a-1 + (A-a) \frac{t^a}{t^a+1} \leq \max\{a,A\}-1,
\end{align*}
we get the second derivative bound \eqref{kin:robustonedim:hessbound} from identity \eqref{kin:onedim:hess}. The first derivative bound \eqref{kin:robustonedim:gradbound} follows from \eqref{kin:robustonedim:hessbound}, since
\begin{align*}
t \varphi'(t) = \int_0^t \varphi'(t) + t\varphi''(t) \; dt.
\end{align*}
Our goal is now to prove the uniform gradient bound \eqref{kin:robustonedim:unifgradbound} for $a,A \geq 2$. In the case that $0 < t < s$, the bound reduces to $(\varphi'(t) - \varphi'(s))(t-s) \geq 0$, which follows from convexity.  The remaining case is $0 < s \leq t$. Notice that for the case $0 < s < t$, convexity implies
\begin{equation}
\label{eq:convexityts} \varphi'(s) \leq \frac{\varphi(t) - \varphi(s)}{t-s} \leq \varphi'(t)
\end{equation}
Notice that in the case $s = 0$ for \eqref{eq:convexityts} we get the inequality $\varphi(t) \leq t\varphi'(t)$, again a condition of convexity. On the other hand, we have just shown that for our $\varphi$ the stronger inequality $\min\{a,A\} \varphi(t) \leq t \varphi'(t)$ holds. This motivates a strategy of searching for a stronger bound of form \eqref{eq:convexityts}, and using this to derive the uniform gradient bound \eqref{kin:robustonedim:unifgradbound}. Indeed, let $\lambda = t/s > 1$, then we will show
\begin{equation}
\label{eq:powerkints} \sigma(\lambda) \varphi'(s) \leq \frac{\varphi(\lambda s) - \varphi(s)}{\lambda s-s} \leq \varphi'(\lambda s) \tau(\lambda)
\end{equation}
where
\begin{equation}
\label{eq:powerkintssigmatau} \sigma(\lambda) = \begin{cases}
\frac{\lambda^a-1}{a(\lambda-1)} & A\geq a\\
\frac{\lambda^A-1}{A(\lambda-1)} & A\leq a
\end{cases} \qquad\qquad  \tau(\lambda) = \begin{cases}
\frac{\lambda (1-\lambda^{-a})}{a(\lambda-1)} & A\geq a\\
\frac{\lambda (1-\lambda^{-A})}{A(\lambda-1)} & A\leq a
\end{cases}
\end{equation}
First, assume $A \geq a$. We need to prove
\begin{equation*}
\frac{\lambda^a -1}{a} s \varphi'(s) \leq \varphi(\lambda s) - \varphi(s) \leq \frac{1 - \lambda^{-a}}{a}  \lambda s \varphi'(\lambda s)
\end{equation*}
We fix $s > 0$, and take $F_1(\lambda) := \varphi(\lambda s) - \varphi(s) - \frac{\lambda^a-1}{a} s \varphi'(s)$ and $F_2(\lambda) := \frac{1 - \lambda^{-a}}{a} \lambda s \varphi'(\lambda s) - \varphi(\lambda s) + \varphi(s)$.
We need to prove that $F_1(\lambda) \ge 0$ and $F_2(\lambda) \ge 0$ for $\lambda \ge 1$.
We have $F_1(1) = F_2(1) = 0$,
\begin{equation*}
F_1'(\lambda) = \frac{(\lambda s)^a}{\lambda} (((\lambda s)^a + 1)^{\frac{A-a}{a}} - (s^a + 1)^{\frac{A-a}{a}}) \ge 0
\end{equation*}
and
\begin{align*}
F_2'(\lambda) &= \frac{(1-\lambda^{-a})s}{a} \l(\varphi'(\lambda s) + \lambda s \varphi''(\lambda s) - a\varphi'(\lambda s) \r)\\
&= \frac{(1-\lambda^{-a})s}{a} \varphi'(\lambda s) (A-a) \frac{(\lambda s)^a}{(\lambda s)^a+1}  \ge 0,
\end{align*}
so indeed $F_1(\lambda) \ge 0$ and $F_2(\lambda) \ge 0$ for every $\lambda > 1$.

Now let $A \leq a$.
Then we need to prove that
\begin{equation*}
\frac{\lambda^A-1}{A} s \varphi'(s) \le \varphi(\lambda s)-\varphi(s) \le \frac{ 1 - \lambda^{-A}}{A} \lambda s \varphi'(\lambda s).
\end{equation*}
We fix $s > 0$, and take $F_3(\lambda) := \varphi(\lambda s) - \varphi(s) - \frac{\lambda^A-1}{A} s \varphi'(s)$ and $F_4(\lambda) := \frac{1 - \lambda^{-A}}{A} \lambda s \varphi'(\lambda s) - \varphi(\lambda s) + \varphi(s)$.
We need to prove that $F_3(\lambda) \ge 0$ and $F_4(\lambda) \ge 0$ for $\lambda \ge 1$.
We have $F_3(1) = F_4(1) = 0$,
\begin{equation*}
F_3'(\lambda) = \frac{(\lambda s)^a}{\lambda} (((\lambda s)^a + 1)^{-\frac{a-A}{a}} - ((\lambda s)^a + \lambda^a)^{-\frac{a-A}{a}}) \ge 0
\end{equation*}
and
\begin{align*}
F_4'(\lambda) &= \frac{(1-\lambda^{-A})s}{A} \l(\varphi'(\lambda s) + \lambda s \varphi''(\lambda s) - A\varphi'(\lambda s) \r)\\
&= \frac{(1-\lambda^{-A})s}{A} \varphi'(\lambda s)\l(a - A + (A-a) \frac{(\lambda s)^a}{(\lambda s)^a+1}\r) \ge 0,
\end{align*}
so indeed $F_3(\lambda) \ge 0$ and $F_4(\lambda) \ge 0$ for every $\lambda > 1$.

Now, we can prove \eqref{kin:robustonedim:unifgradbound}. We have so far proven the following inequalities in $\varphi(t), \varphi(s), \varphi'(t), \varphi'(s)$:
\begin{align*}
\varphi(s) &\ge 0, & s \varphi'(s) - \min(a,A) \varphi(s) &\ge 0, \\
\varphi(t) - \varphi(s) - (\lambda-1) \phi(\lambda) s \varphi'(s) &\ge 0, & (\lambda-1) \tau(\lambda) s \varphi'(t) - \varphi(t) + \varphi(s) &\ge 0.
\end{align*}
We try to express the inequality $s \varphi'(s) + (t-s) (\varphi'(t)-\varphi'(s)) - \varphi(t) \ge 0$ as a linear combination of the above four inequalities with non-negative coefficients:
\begin{align*}
&s \varphi'(s) + (t-s) (\varphi'(t)-\varphi'(s)) - \varphi(t) = c_1 \varphi(s) + c_2 (s \varphi'(s) - \min(a,A) \varphi(s)) \\
&+ c_3 (\varphi(t) - \varphi(s) - (\lambda-1) \sigma(\lambda) s \varphi'(s)) + c_4 ((\lambda-1) \tau(\lambda) s \varphi'(t) - \varphi(t) + \varphi(s)).
\end{align*}
Comparing the coefficients of $\varphi(s)$, $\varphi(t)$, $\varphi'(s)$, $\varphi'(t)$, we get the following equations:
\begin{align*}
c_1 - \min(a,A) c_2 - c_3 + c_4 &= 0, \\
c_3 - c_4 &= -1, \\
c_2 - (\lambda-1) \sigma(\lambda) c_3 &= 2 - \lambda, \\
(\lambda-1) \tau(\lambda) c_4 &= \lambda-1.
\end{align*}
This system of equations has a unique solution:
$c_4 = \frac{1}{\tau(\lambda)}$, $c_3 = \frac{1}{\tau(\lambda)} - 1$, $c_2 = 2-\lambda + (\lambda-1) \sigma(\lambda) (\frac{1}{\tau(\lambda)} - 1)$ and $c_1 = \min(a,A) c_2 - 1$.
We will prove that $c_1, c_2, c_3, c_4 \ge 0$.
Clearly $\tau(\lambda) > 0$.
We claim that $\tau(\lambda) \le 1$.
For this it is enough to check that $\lambda (1 - \lambda^{-\alpha}) \le \alpha (\lambda-1)$ for every $\lambda > 1$ and $\alpha \ge 2$.
After reordering the terms, we get $1 + (1-\alpha) (\lambda-1) \le \lambda^{1-\alpha} = (1 + (\lambda-1))^{1-\alpha}$, which follows from the generalized Bernoulli inequality.
So $0 < \sigma(\lambda) \le 1$, therefore $c_3, c_4 \ge 0$.
We just need to prove now that $\min(a,A) c_2 \ge 1$, because then $c_1, c_2 \ge 0$.
If $a \le A$, then $c_1 = \min(a,A) c_2 - 1 = a (2-\lambda + \lambda^a - \lambda^{a-1} - \frac{\lambda^a}{a})$.
If $a \ge A$, then $c_1 = A (2-\lambda + \lambda^A - \lambda^{A-1} - \frac{\lambda^A}{A})$.
So the only remaining thing to show is
\[
2\alpha- \alpha\lambda + \alpha\lambda^{\alpha} -\alpha \lambda^{\alpha-1} - \lambda^{\alpha} \ge 0
\]
for every $\lambda > 1$ and $\alpha \ge 2$. Letting $\epsilon = 1/\lambda$, this is equivalent to showing
\[
2\alpha \epsilon^{\alpha}- \alpha\epsilon^{\alpha-1} + \alpha -\alpha \epsilon - 1 \ge 0
\]
for $\epsilon \in (0, 1)$. Let $\pi(\epsilon) = 2\alpha \epsilon^{\alpha}- \alpha\epsilon^{\alpha-1} + \alpha -\alpha \epsilon - 1$. To see that $\pi(\epsilon) \geq 0$, note
\[
\pi'(\epsilon)=  \alpha \epsilon^{\alpha-2} (2 \alpha \epsilon - \alpha + 1) - \alpha
\]
from which $\pi'(\epsilon) < \pi'(1/2) \leq 0$ for $\epsilon < 1/2$ and $\pi'(\epsilon) > \pi'((\alpha-1)/\alpha) \geq 0$ for $\epsilon > \alpha / (\alpha-1)$. This implies that $\pi$ is minimized on $[1/2, (\alpha-1) / \alpha]$. Our result follows then from the fact that for $\epsilon \in [1/2, (\alpha-1) / \alpha]$,
\begin{equation*}
\pi(\epsilon) \geq \alpha - \alpha \epsilon - 1 \geq 0
\end{equation*}

\end{enumerate}
\end{proof}


\begin{lem}
\label{kin:rho}
Given $a \in [1, \infty)$, $A \in [1, \infty)$, and $\varphi_a^A$ in \eqref{eq:kindef}. Define $B = A/(A-1)$, $b = a/(a-1)$. For convenience, define 
\begin{equation}
\varphi(t) = \varphi_a^A(t) \qquad \qquad \phi(t) = \varphi_b^B(t).
\end{equation}
For $t\in [0, \infty)$ define the function
\begin{equation}
\label{kin:robust:rho} \rho(t) = \l(\frac{t^a}{t^a+1} + (t^a+1)^{-\tfrac{A-1}{a-1}}\r)^{\tfrac{a-A}{a(A-1)}},
\end{equation}
and for $a \neq A$ define the constant
\begin{equation}
\label{kin:robustonedim:nearconjcon} C_{a,A} = \l(1 - \l(\tfrac{a-1}{A-1}\r)^{\tfrac{a-1}{A-a}} + \l(\tfrac{a-1}{A-1}\r)^{\tfrac{A-1}{A-a}}\r)^{\tfrac{B-b}{b}}.
\end{equation}
We have the following results. For all $t \in (0, \infty)$,
\begin{equation}
\label{kin:robustonedim:nearconjerr} \phi'(\varphi'(t)) = \rho(t) t,
\end{equation}
which means that $\rho$ captures the relative error between $(\varphi^*)'$ and $\phi'$, because $(\varphi^*)'(t) = (\varphi')^{-1}(t)$. Finally, $\rho$ is bounded for all $t \in (0, \infty)$ between the constants,
\begin{equation}
\label{kin:robustonedim:nearconjerrbound} 1 \leq \rho(t) \leq C_{a,A}\\
\end{equation}
\end{lem}

\begin{proof}
We will show the results backwards, starting with \eqref{kin:robustonedim:nearconjerrbound}. By rearrangement,
\begin{equation*}
\rho(t) = (\tfrac{t^{a}}{t^a + 1} + (t^a + 1)^{-1}(t^a + 1)^{\tfrac{a-A}{a-1}})^{\tfrac{a - A}{a(A-1)}}.
\end{equation*}
If $a \geq A$ we have $(t^a + 1)^{\tfrac{a-A}{a-1}} \geq 1$ and $t^{\tfrac{a-A}{a(A-1)}}$ is increasing, so $\rho(t) \geq 1$. If $a < A$ we have $(t^a + 1)^{\tfrac{a-A}{a-1}} \leq 1$ and $t^{\tfrac{a-A}{a(A-1)}}$ is decreasing, so again $\rho(t)  \geq 1$. This proves the left hand inequality of \eqref{kin:robustonedim:nearconjerrbound}. Now, assume that $A \neq a$. Looking at $\tfrac{t^a}{t^a + 1} + (t^a +1)^{-\tfrac{A-1}{a-1}}$ we have
\begin{align*}
\l[\tfrac{t^a}{t^a + 1} + (t^a +1)^{-\tfrac{A-1}{a-1}}\r]'  = \tfrac{at^{a-1}}{(t^a +1)^2} - \tfrac{A-1}{a-1} (t^a +1)^{-\tfrac{A-1}{a-1}-1} a t^{a-1}
\end{align*}
Since $t \neq 0$ we see that it has a stationary point at 
\begin{align*}
(t^a+1)^{-1} - \tfrac{A-1}{a-1} (t^a +1)^{-\tfrac{A-1}{a-1}} = 0,
\end{align*}
which is equivalent to $(t^a + 1) = \l(\tfrac{A-1}{a-1}\r)^{\tfrac{a-1}{A-a}}$. This is also a stationary point of $\rho(t)$. Since $\rho(0) = \rho(\infty) = 1$ and $\rho(t) \geq 1$ this stationary point must be a maximum. Thus
\begin{align*}
\rho(t) &= \l(1 - (t^a+1)^{-1} + (t^a+1)^{-\tfrac{A-1}{a-1}}\r)^{\tfrac{B-b}{b}}\\
&\leq \l(1 - \l(\tfrac{a-1}{A-1}\r)^{\tfrac{a-1}{A-a}} + \l(\tfrac{a-1}{A-1}\r)^{\tfrac{A-1}{A-a}}\r)^{\tfrac{B-b}{b}}
\end{align*}
This proves the right hand inequality of \eqref{kin:robustonedim:nearconjerrbound}. For \eqref{kin:robustonedim:nearconjerr}, since $(b-1)(a-1) = ab - a - b +1 = 1$, we have,
\begin{align*}
\phi'(\varphi'(t)) &= ([(t^a + 1)^{\tfrac{A-a}{a}} t^{a-1}]^b + 1)^{\tfrac{B-b}{b}} [(t^a + 1)^{\tfrac{A-a}{a}} t^{a-1}]^{b-1}\\
&= ((t^a + 1)^{\tfrac{A-a}{a-1}} t^{a} + 1)^{\tfrac{B-b}{b}}  (t^a + 1)^{\tfrac{A-a}{a(a-1)}} t\\
\intertext{we have $\tfrac{B-b}{b} = \tfrac{A(a-1) - a(A-1)}{(a-1)(A-1)b} = \tfrac{a - A}{a(A-1)}$ and thus}
&= ((t^a + 1)^{\tfrac{A-a}{a-1}} t^{a} + 1)^{\tfrac{a - A}{a(A-1)}}  (t^a + 1)^{\tfrac{A-a}{a(a-1)}} t\\
&= \rho(t)t
\end{align*}
\end{proof}


Now we are ready to prove the key results in this section.

\kinrobust*

\begin{proof}
Again, for the purposes of this proof, let $\varphi(t) = \varphi_a^A(t)$ and $\phi(t) = \varphi_b^B(t)$ for $t \in [0, \infty)$. 
\begin{enumerate}

\item \emph{Convexity}. First, since norms are positive definite and $\varphi$ uniquely minimized at $0 \in \R$ by Lemma \ref{kin:robustonedim}, this proves that $0 \in \R^d$ is a unique minimizer of $\K$. Let $\epsilon \in (0,1)$ and $p, q \in \R^d$ such that $p \neq q$. By the monotonicity proved in Lemma \ref{kin:robustonedim} and the triangle inequality
\begin{align*}
	\K(\epsilon p + (1-\epsilon) q) &= \varphi(\norm{\epsilon p + (1-\epsilon) q}_*)\\
	&\leq \varphi(\epsilon \norm{p}_* + (1-\epsilon) \norm{q}_*)
	\intertext{and finally, by the strict convexity proved in Lemma \ref{kin:robustonedim}}
	&< \epsilon \K(p) + (1-\epsilon) \K(q).
\end{align*}

\item \emph{Conjugate}. Let $x \in \R^d$. First, by definition of the convex conjugate and the dual norm,
\begin{align*}
	\K^*(x) &= \sup_{p \in \R^d} \{\inner{x}{p} - k(p)\} = \sup_{p \in \R^d} \{\inner{x}{p} - \varphi(\norm{p}_*)\}\\
	&= \sup_{t \geq 0} \sup_{\norm{p}_* = t} \{\inner{x}{p} - \varphi(t)\} = \sup_{t \geq 0} \{t \norm{x} - \varphi(t)\} = \varphi^*(\norm{x})
\end{align*}

\item \emph{Gradient}.
First we argue for differentiability. For $p = 0$ (or $q=0$ in the case of \eqref{kin:robust:unifineq}), we have by the equivalence of the norms that there exists $c > 0$ such that $\norm{p}_* < c\norm{p}_2$. Thus, $\lim_{\norm{p}_2 \to 0} \K(p) \norm{p}_2^{-1} \leq \lim_{\norm{p}_2 \to 0} \varphi(c \norm{p}_2)\norm{p}^{-1}= c \lim_{t \to 0} \varphi(t)t^{-1}=0$, and thus we have $\grad \K(0) = 0$. Now for $p \neq 0$, we have $\norm{p}_* \neq 0$. Since $\varphi(t)$ is differentiable for $t > 0$ and $\norm{p}_*$ at $p \neq 0$, we have by the chain rule $\grad \K(p) = \varphi'(\norm{p}_*) \grad \norm{p}_*$.

All four results follow trivially when $p = 0$. In particular, \eqref{kin:robust:unifineq} reduces to $\K(p) \leq \inner{\grad \K(p)}{p}$ for $q=0$ and $0 \leq \inner{\grad \K(q)}{q}$ for $p=0$; both follow from convexity.

Now, assume $p \neq 0$. For \eqref{kin:robust:innerbound}, \eqref{kin:robust:slowgrowth1}, and \eqref{kin:robust:slowgrowth2}  we have, by Lemma \ref{lem:norms}, $\inner{\grad \K(p)}{p}  = \norm{p}_* \varphi'(\norm{p}_*)$ and $\varphi^*(\norm{\grad \K(p)}) = \varphi^*(\varphi'(\norm{p}_*))$ and $\phi(\norm{\grad \K(p)}) = \phi(\varphi'(\norm{p}_*)) $. Letting $t = \norm{p}_* > 0$, \eqref{kin:robust:innerbound} follows directly from \eqref{kin:robustonedim:gradbound} of Lemma \ref{kin:robustonedim}. For \eqref{kin:robust:slowgrowth1}, we have from convex analysis \eqref{eq:fenchelyoungtight} that 
\begin{equation}
\label{eq:convexanalysisresult} \varphi^*(\varphi'(t)) = t\varphi'(t) - \varphi(t).
\end{equation}
This implies that  $\varphi^*(\varphi'(t)) \leq (\max\{a,A\} -1) \varphi(t)$, again by \eqref{kin:robustonedim:gradbound} of Lemma \ref{kin:robustonedim}.  For \eqref{kin:robust:slowgrowth2} assume $a,A>1$ and consider that by \eqref{kin:robustonedim:hessbound} of Lemma \ref{kin:robustonedim} and \eqref{kin:robustonedim:nearconjerrbound} of Lemma \ref{kin:rho},
\begin{align*}
[\phi(\varphi'(t)) ]' = \phi'(\varphi'(t)) \varphi''(t) = \rho(t) t \varphi''(t) \leq \varphi'(t) C_{a,A} (\max\{a, A\} - 1)
\end{align*}
Integrating both sides of this inequality gives $\phi(\varphi'(t)) \leq C_{a,A} (\max\{a, A\} - 1) \varphi(t)$.

Finally, for the uniform gradient bound \eqref{kin:robust:unifineq}, assume $p \neq 0$ and $q \neq 0$. Lemma \ref{lem:norms} implies by Cauchy-Schwartz that $-\inner{\grad \norm{p}_*}{q} \geq - \norm{q}$ for any $p,q \in \R^d \setminus \{0\}$. Thus by Lemma \ref{lem:norms},
\begin{align*}
\inner{\grad \K(q)}{q} + \inner{\grad \K(p)  - \grad \K(q)}{p-q} \geq \varphi'(\norm{q}_*) \norm{q}_* + (\varphi'(\norm{p}_*) - \varphi'(\norm{q}_*))(\norm{p}_*-\norm{q}_*)
\end{align*}
and our result is implied by the one dimensional result \eqref{kin:robustonedim:unifgradbound} of Lemma \ref{kin:robustonedim}.

\item \emph{Hessian}.
Throughout, assume $p \in \R^d \setminus \{0\}$. First we argue for twice differentiability. We have $\grad \K(p) = \varphi'(\norm{p}_*) \grad \norm{p}_*$, which for $p \neq 0$ is a product of a differentiable function and a composition of differentiable functions. Thus, we have differentiability, and by the chain rule,
\begin{equation}
\label{eq:khess} \hess \K(p) = \varphi''(\norm{p}_*) \grad\norm{p}_* \grad\norm{p}_*^T + \varphi'(\norm{p}_*) \hess \norm{p}_*
\end{equation}
All of these terms are continuous at $p \neq 0$ by assumption or inspection of \eqref{kin:onedim:hess}.

We study \eqref{eq:khess}. For \eqref{kin:robust:hessianpp}, we have by Lemma \ref{lem:norms} and \eqref{kin:robustonedim:gradbound},\eqref{kin:robustonedim:hessbound} of Lemma \ref{kin:robustonedim},
\begin{align*}
\inner{p}{\hess \K(p) p} &= \varphi''(\norm{p}_*) \inner{p}{\grad\norm{p}_* \grad\norm{p}_*^T p} + \varphi'(\norm{p}_*)  \inner{p}{\hess \norm{p}_* p}\\
&= \varphi''(\norm{p}_*) \norm{p}_*^2\\
&\leq \max\{a,A\}(\max\{a,A\} - 1) \varphi(\norm{p}_*).
\end{align*}
For \eqref{kin:robust:hessiannorm} first note, by Lemma \ref{lem:norms}
\begin{align*}
\inner{v}{\grad\norm{p}_* \grad\norm{p}_*^T v} = (\inner{v}{\grad\norm{p}_*})^2 \leq \norm{v}_*^2
\end{align*}
and further $\inner{p}{\grad\norm{p}_* \grad\norm{p}_*^T p}  = \norm{p}_*^2$. Thus $\maxeigenconj{\grad\norm{p}_* \grad\norm{p}_*^T} = 1$. Together, along with our assumption on the Hessian of $\norm{p}_*$, we have
\begin{align*}
\maxeigenconj{\hess \K(p)} &\leq \varphi''(\norm{p}_*) \maxeigenconj{\grad\norm{p}_* \grad\norm{p}_*^T} + \varphi'(\norm{p}_*)  \maxeigenconj{\hess \norm{p}_*}\\
&\leq \varphi''(\norm{p}_*) +  N \varphi'(\norm{p}_*) \norm{p}_*^{-1}
\intertext{and by \eqref{kin:robustonedim:hessbound} of Lemma \ref{kin:robustonedim}}
&\leq \varphi'(\norm{p}_*) \norm{p}_*^{-1} (\max\{a,A\} - 1 + N)
\end{align*}
On the other hand, by Lemma \ref{lem:norms} and the monotonicity of Lemma \ref{kin:robustonedim},
\begin{align*}
	\maxeigenconj{\hess \K(p)} &\geq \varphi''(\norm{p}_*) \inner{\frac{p}{\norm{p}_*}}{\grad\norm{p}_* \grad\norm{p}_*^T \frac{p}{\norm{p}_*}} + \varphi'(\norm{p}_*)\inner{\frac{p}{\norm{p}_*}}{\hess \norm{p}_* \frac{p}{\norm{p}_*}}\\
	&= \varphi''(\norm{p}_*) > 0
\end{align*}
Taken together, we have
\begin{equation*}
0 < \frac{\maxeigenconj{\hess \K(p)}}{\MaA - 1 + N} \leq \varphi'(\norm{p}_*) \norm{p}_*^{-1}
\end{equation*}
Now, assume $a, A \geq 2$ and let $t = \norm{p}_* > 0$ and $\chi(t) = \varphi_{a/2}^{A/2}(t)$. Our goal is to show that 
\begin{equation*}
\chi^*\l(\frac{\maxeigenconj{\hess \K(p)}}{\MaA - 1 + N}\r) \leq (\MaA - 2) \varphi(\norm{p}_*)
\end{equation*}
To do this we argue that $\chi^*(t)$ is an non-decreasing function  on $(0,\infty)$ and $\chi^*(\varphi'(t)t^{-1}) \leq (\max\{a,A\} - 2) \varphi(t)$, from which our result would follow. First, we have for $r \in [0, \infty)$ and $s \in (0, \infty)$ such that for $t \geq s$,
\begin{align*}
	\chi^*(t) \geq tr - \chi(r) \geq sr - \chi(r)
\end{align*}
Taking the supremum in $r$ returns the result that $\chi^*$ is non-decreasing. Otherwise it can be verified directly from \eqref{kin:robustonedim:conjugatebound} -- \eqref{kin:robustonedim:conjugatebound5}. Thus, what remains to show is $\chi^*(\varphi'(t) t^{-1}) \leq (\max\{a,A\}-2) \varphi(t)$.
Note that $\chi(t^2) = 2 \varphi(t)$ and $\chi'(t^2) t = \varphi'(t)$. Using \eqref{eq:fenchelyoungtight} of convex analysis and \eqref{kin:robustonedim:gradbound} of Lemma \ref{kin:robustonedim},
\begin{align*}
\chi^*(\varphi(t) t^{-1}) &= \chi^*(\chi'(t^2)) =t^2 \chi'(t^2) - \chi(t^2) \leq (\max\{\tfrac{a}{2}, \tfrac{A}{2}\} - 1) \chi(t^2) = (\max\{a,A\} - 2) \varphi(t)
\end{align*}
from which our result follows.
\end{enumerate}
 \end{proof}

\normhess*

\begin{proof}
A short calculation reveals that 
\begin{equation}
\hess \norm{x}_b = \frac{(b-1)}{\norm{x}_b}\l(\diag\l(\frac{|x^{(n)}|^{b-2}}{\norm{x}_b^{b-2}}\r) -  \grad \norm{x}_b \grad \norm{x}_b^T\r)
\end{equation}
Thus, since $\inner{b}{aa^Tb} = \inner{a}{b}^2 \geq 0$ for any $a, b \in \R^d$, we have $\maxeigennorm{(1-b)\grad \norm{x}_b \grad \norm{x}_b^T}{\norm{\c}_b} \leq 0$ and
\begin{align*}
\norm{x}_b \maxeigennorm{\hess \norm{x}_b}{\norm{\c}_b} \leq (b-1) \maxeigennorm{\diag\l(|x^{(n)}|^{b-2}\norm{x}_b^{2-b}\r)}{\norm{\c}_b}.
\end{align*}
First, consider the case $b > 2$. Given $v \in \R^d$ such that $\norm{v}_b = 1$, we have by the H{\"o}lder's inequality along with the conjugacy of $b/2$ and $b/(b-2)$,
\begin{align*}
\inner{v}{\diag\l(|x^{(n)}|^{b-2}\norm{x}_b^{2-b}\r) v} \leq \l(\sum_{n=1}^d \frac{|x^{(n)}|^b}{\norm{x}_b^b}\r)^{\tfrac{b-2}{b}} \l(\sum_{n=1}^d |v^{(n)}|^b \r)^{\tfrac{2}{b}} = 1
\end{align*}
Now, for the case $b=2$ we get $\diag\l(|x^{(n)}|^{b-2}\norm{x}_2^{2-b}\r) = I$ and $\maxeigennorm{I}{\norm{\c}_2} = 1$.
Our result follows.
\end{proof}

\robustonedimconj*

\begin{proof}
For convenience, define 
\begin{equation}
\varphi(t) = \varphi_a^A(t) \qquad \qquad \phi(t) = \varphi_b^B(t).
\end{equation}
As a reminder, for $t \in (0, \infty)$, the following identities can be easily verified.
\begin{align}
\label{kin:onedim:grad} \varphi'(t) &=  (t^a + 1)^{\tfrac{A-a}{a}} t^{a-1}\\
\label{kin:onedim:hess} \varphi''(t) &=  t^{-1}\varphi'(t)\l(a-1  + (A-a)\tfrac{t^{a}}{t^a+1}\r)
\end{align}

\begin{enumerate}
\item \emph{Near Conjugate}. As a reminder $\varphi^*(t) = \sup_{s \geq 0} \{ts - \varphi(s)\}$. First, since $\varphi^*(0) = - \inf_{s \geq 0} \{\varphi(s)\}$ = 0, this result holds for $t=0$. Assume $t > 0$. Our strategy will be to show that for $s \in [0, \infty)$ we have $st - \varphi(s) \leq \phi(t)$. This is true for $s=0$ by the monotonicity of $\phi$ in Lemma \ref{kin:robustonedim}, so assume $s>0$. Now, consider $s = \phi'(\varphi'(r))$ for some $r \in (0, \infty)$. To see that this is a valid parametrization for $s$, notice that $\lim_{r \to 0} \phi'(\varphi'(r)) = 0$ and
\begin{align*}
[\phi'(\varphi'(r))]' &= \phi''(\varphi'(r)) \varphi''(r) > 0
\end{align*}
Thus $s(r) = \phi'(\varphi'(r))$ is one-to-one and onto $(0, \infty)$. Further we have by Lemma \ref{kin:rho} that
\begin{equation}
t \leq \rho(t)t = \phi'(\varphi'(t))
\end{equation}
and thus $(\phi')^{-1}(t) \leq \varphi'(t)$, since $\phi''(t) > 0$.
All together, using convexity we have
\begin{align*}
\phi(t) &\geq \phi(\varphi'(r)) +  \phi'(\varphi'(r)) (t - \varphi'(r)) = st  + \phi((\phi')^{-1}(s)) - s (\phi')^{-1}(s)
\intertext{taking the derivative of $\phi( (\phi')^{-1}(s)) - s (\phi')^{-1}(s)$ we get $- (\phi')^{-1}(s)$. Since $- (\phi')^{-1}(s) \geq - \varphi'(s)$, we finally get}
&\geq s t -\varphi(s).
\end{align*}
Taking the supremum in $s$ gives us our result.

\item \emph{Conjugate}. As a reminder $\varphi^*(t) = \sup_{s \geq 0} \{ts - \varphi(s)\}$. Since $\varphi^*(0) = - \inf_{s \geq 0} \{\varphi(s)\}$ = 0, these results all hold when $t=0$. \eqref{kin:robustonedim:conjugatebound1} is a standard result, since $\varphi_a^a(t)= \tfrac{1}{a} t^a$. \eqref{kin:robustonedim:conjugatebound5} is a standard result, since $\varphi_1^1(s) = s$. Thus, we assume $a,A > 1$ and $t > 0$ for the remainder. For \eqref{kin:robustonedim:conjugatebound4}, assume just $A = 1$. The stationary condition of the supremum of $ts - \varphi(s)$ in $s$ is
\begin{align*}
t = (s^a + 1)^{-\tfrac{1}{b}} s^{a-1}
\end{align*}
Raising both sides to $b$ we get $t^b = \tfrac{s^a}{s^a+1}$, whose solution for $t \in [0,1]$ is $s = \l(\tfrac{t^b}{1-t^b}\r)^{\tfrac{1}{a}}$. Thus,
\begin{align*}
\varphi^*(t) &= t \l(\frac{t^b}{1-t^b}\r)^{\tfrac{1}{a}} - \l(\frac{1}{1-t^b}\r)^{\tfrac{1}{a}} + 1\\
&= \frac{t^b}{(1-t^b)^{\tfrac{1}{a}}} - \frac{1}{(1-t^b)^{\tfrac{1}{a}}}+ 1\\
&= 1 - (1-t^b)^{1 - \tfrac{1}{a}}
\end{align*}
When $t > 1$, $ts$ dominates $\varphi(s)$ and the supremum is infinite.
Now for \eqref{kin:robustonedim:conjugatebound3}, assume just $a = 1$. We have the stationary condition of the conjugate equal to $t = (s+1)^{A-1}$, which corresponds to
\begin{align*}
s = \max\{t^{\tfrac{1}{A-1}} - 1, 0\}
\end{align*}
Thus, when $t > 1$ we have
\begin{align*}
\varphi^*(t) &= t (t^{\tfrac{1}{A-1}} - 1)  - \tfrac{1}{A}t^B + \tfrac{1}{A}\\
&= \tfrac{1}{B}t^B - t + \tfrac{1}{A}
\end{align*}
otherwise $\varphi^*(s) = 0$.
\end{enumerate}
 \end{proof}

\assvarphiaA*

\begin{proof}
	First, by Lemma \ref{kin:robust}, this choice of $\K$ satisfies assumptions \ref{ass:cont:kconvex} and \ref{ass:semiA:gKpsem} / \ref{ass:semiB:gKpsem} with constant $C_{\K} = \max\{a, A\}$. We consider the remaining assumptions of \ref{ass:cont}, \ref{ass:imp}, \ref{ass:semiA}, and \ref{ass:semiB}..
	
\begin{enumerate}
\item Our first goal is to derive $\alpha$. By assumption \ref{knownpower:impcondition}, we have $\mu \varphi_b^B(\norm{x}) \leq \fc(x)$. Since $(\mu \varphi_b^B(\norm{\c}))^* = \mu (\varphi_b^B)^*(\mu^{-1}\norm{\c}_*)$ by Lemma \ref{kin:robust} and the results discussed in the review of convex analysis, we have by assumption \ref{knownpower:fpower},
\begin{align*}
\fc^*(p) \leq \mu (\varphi_b^B)^*\l(\mu^{-1}\norm{p}_*\r) \leq \max\{\mu^{1-a}, \mu^{1-A}\} \K(p)
\end{align*}
Thus, we have $\alpha=\min\{\mu^{a-1}, \mu^{A-1}, 1\}$ constant. Moreover, we can take $\Ca = \gamma$. This along with \ref{knownpower:fconvex} implies that $f$ and $\K$ satisfy assumptions \ref{ass:cont}

By Lemma \ref{kin:robust} and assumption \ref{knownpower:impcondition} we have
\begin{align*}
\varphi_a^A(\norm{\grad f(x)}_*) &\leq L (f(x) - f(\xmin)),\\
(\varphi_a^A)^*(\norm{\grad \K(p)}) &= (\MaA-1) \K(p),
\end{align*}
By Fenchel-Young and the symmetry of norms, we have $|\inner{\grad f(x)}{\grad \K(p)}| \leq \max\{\MaA-1, L\} \Ha(x,p)$, from which we derive $\Dfk$ and the fact that $f,\K$ satisfy assumptions \ref{ass:imp}.

\item The analysis of 1. holds for assumptions \ref{ass:cont} and \ref{ass:imp}. Now, to derive the conditions for assumptions \ref{ass:semiA} consider 
\begin{align*}
\norm{\grad \K(p)}^2 =  [(\varphi_a^A)'(\norm{p}_*)]^{2}
\end{align*}
Note that,
\begin{align*}
\varphi_{b/2}^{B/2}([(\varphi_a^A)'(t)]^{2}) = 2 \varphi_b^B((\varphi_a^A)'(t))
\end{align*}
Thus $\varphi_{b/2}^{B/2}(\norm{\grad \K(p)}^2) \leq 2 C_{a,A} (\MaA - 1) \K(p)$ for all $p \in \R^d$, by Lemma \ref{kin:robust}.
Now all together, by the Fenchel-Young inequality and assumption \ref{knownpower:semiAcondition}, we have for $p \in \R^d$ and $x \in \R^d \setminus \{\xmin\}$
\begin{align*}
\inner{\grad \K(p)}{\hess f(x) \grad \K(p)} &\leq \norm{\grad \K(p)}^2 \maxeigen{\hess f(x)} \\
&\leq \Lf\varphi_{b/2}^{B/2}(\norm{\grad \K(p)}^2) + \Lf(\varphi_{b/2}^{B/2})^*\l(\frac{\maxeigen{\hess f(x)}}{\Lf}\r)\\
&\leq \Lf2 C_{a,A} (\MaA - 1) \K(p) + \Lf(\varphi_{b/2}^{B/2})^*\l(\frac{\maxeigen{\hess f(x)}}{\Lf}\r)\\
&\leq \Dfk \alpha \Ha(x,p).
\end{align*}
This gives us assumptions \ref{ass:semiA}.

\item The analysis of 1. holds again for assumptions \ref{ass:cont} and \ref{ass:imp}. Now, for assumptions \ref{ass:semiB}, we first note that Lemma \ref{kin:robust} gives us constants $\Dk = \max\{a, A\} (\max\{a, A\} - 1), E_{\K} = \max\{a,A\} - 1, F_{\K} =1$. 

For the remaining constant $\Dfk$ we follow a similar path as 2. First, note that since $b,B \leq 2$ we have that $a,A \geq 2$. This, along with assumption \ref{knownpower:semiBcondition}, let's us use \eqref{kin:robust:hessiannorm} of Lemma \ref{kin:robust} for $p \in \R^d \setminus \{0\}$. Now, letting $M = (\max\{a, A\} - 1 + N)$ and applying \eqref{kin:robust:hessiannorm} of Lemma \ref{kin:robust} along with the Fenchel-Young inequality, we have for $p \in \R^d \setminus \{0\}$ and $x \in \R^d$
\begin{align*}
\inner{\grad \f(x)}{\hess \K(p) \grad f(x)} &\leq \norm{\grad f(x)}_*^2 \maxeigenconj{\hess \K(p)} \\
&\leq M \varphi_{a/2}^{A/2}(\norm{\grad f(x)}_*^2) + M(\varphi_{a/2}^{A/2})^*\l(\frac{\maxeigenconj{\hess \K(p)}}{M}\r)\\
&\leq 2 L M (f(x) - f(\xmin)) + M (\max\{a,A\} - 2) \K(p)\\
&\leq \Dfk \alpha \Ha(x,p).
\end{align*}
This gives us assumptions \ref{ass:semiB}.
\end{enumerate}

\end{proof}

\end{document}